\documentclass[11pt]{article}

\usepackage[utf8]{inputenc}
\usepackage[T1]{fontenc}
\usepackage[main=english,ngerman]{babel}
\usepackage{lmodern}
\usepackage{tcolorbox}
\usepackage{enumerate}
\usepackage{multirow}
\usepackage{nicefrac}
\usepackage{caption}
\usepackage{graphicx}
\usepackage{tikz}
\usetikzlibrary{patterns}
\usetikzlibrary{arrows}
\usetikzlibrary{decorations}
\usetikzlibrary{decorations.pathreplacing}
\usepackage{tikz-3dplot}
\usepackage{subfigure}
\usepackage{pgfplots}
\usepackage{hyperref}
\usepackage{picture}
\usepackage[absolute,overlay]{textpos}

\usepackage[left=2.4cm,right=2.4cm,top=1.2cm,bottom=1.4cm,includeheadfoot]{geometry}
\usepackage{media9}

\usepackage{amsmath,amssymb,amsthm}
\usepackage{mathtools}
\usepackage[mathscr]{euscript}
\usepackage{xpatch}
\usepackage{asymptote}
\usepackage{mhchem}

\newtheorem{definition}{Definition}
\theoremstyle{remark}
\newtheorem{remark}{Remark}
\theoremstyle{theorem}
\newtheorem{theorem}{Theorem}
\theoremstyle{theorem}
\newtheorem{lemma}{Lemma}

\newtheorem{corollary}{Corollary}
\newtheorem{assumption}{Assumption}

\pgfplotsset{
	log x ticks with fixed point/.style={
		xticklabel={
			\pgfkeys{/pgf/fpu=true}
			\pgfmathparse{exp(\tick)}%
			\pgfmathprintnumber[fixed relative, precision=3]{\pgfmathresult}
			\pgfkeys{/pgf/fpu=false}
		}
	},
	log y ticks with fixed point/.style={
		yticklabel={
			\pgfkeys{/pgf/fpu=true}
			\pgfmathparse{exp(\tick)}%
			\pgfmathprintnumber[fixed relative, precision=3]{\pgfmathresult}
			\pgfkeys{/pgf/fpu=false}
		}
	}
}

\usepackage[rightcaption]{sidecap}
\renewcommand{\d}{\textup{div}}
\renewcommand{\c}{\textup{curl}}
\renewcommand{\o}{\Omega}
\newcommand{\s}{\mathbb{S}}

\newcommand{\h}{\hat{\partial}}

\newcommand{\p}{\textbf}
\newcommand{\f}{}
\renewcommand\t[1]{{\mathcal{#1}}}

\newcommand\norm[1]{\left\lVert#1\right\rVert}
\newcommand\n[1]{\left\lVert#1\right\rVert}
\graphicspath{{images/}}

	\title{Mixed Isogeometric Methods for Linear Elasticity with Weakly Imposed Symmetry }
	
	\providecommand{\keywords}[1]
	{
		\small	
		\textbf{\textit{Keywords---}} #1
	}
	
	\author{Jeremias Arf \footnotemark[2]}
	
\begin{document}
		
		\maketitle
		\renewcommand{\thefootnote}{\fnsymbol{footnote}}
		
		\footnotetext[2]{ TU Kaiserslautern, Dept. of Mathematics, Gottlieb-Daimler-Str. 48, 67663 Kaiserslautern,\\ Germany ( E-mail: {\tt arf@mathematik.uni-kl.de}).}
		
		\begin{abstract}
		We consider and discretize a mixed formulation for linear elasticity with weakly imposed symmetry in two and three dimensions. Whereas existing methods mainly deal with simplicial or  polygonal meshes, we take advantage of isogeometric analysis (IGA) and consequently allow for shapes with curved boundaries. To introduce the discrete spaces we use   isogeometric discrete  differential forms defined by  proper B-spline spaces. For the proposed schemes a proof of well-posedness and an error estimate are  given. Further we discuss our ansatz by means of different numerical examples.
	\end{abstract}
\normalsize
	\section{Introduction}
In general it is hard to define  numerical methods for partial differential equations (PDEs)  which guarantee convergence and stability such that the approximate solution depends continuously on (initial-)boundary conditions as well as on source functions. Thus, it is remarkable that  Arnold,  Falk and Winther established with the Finite Element Exterior Calculus (FEEC) \cite{Arnold2010FiniteEE,ArnoldBook} an abstract guideline for the design of  several mixed formulations in order to have both, stability and convergence.   The underlying structure is related to the so-called Hodge-Laplacians associated to some   Hilbert complex and the structure-preserving discretization of such complexes. In \cite{ArnoldBook} the needed key properties for finite-dimensional ansatz and test function spaces are derived and moreover the polynomial differential forms on simplicial meshes are introduced. Latter spaces are fundamental objects in the development of mixed methods. However, as already indicated the underlying approach of FEEC is quite abstract and the results allow for other discrete spaces.   For instance, Buffa et al.  \cite{Buffa2011IsogeometricDD} utilized isogeometric analysis (IGA) to define  structure-preserving discrete de Rham chains. With such discrete isogeometric differential forms one obtains stable approximations e.g. within the field of Maxwell's equations, where curved geometries can be handled due to the IGA ansatz. Although we find in literature different Hilbert complexes corresponding to interesting PDEs, like the \emph{Hessian complex} in linearized General Relativity (see \cite{QuennevilleBlair2015ANA}), it is not always easy to apply directly the steps from \cite{ArnoldBook}. Especially the definition of discrete subcomplexes equipped with bounded cochain projections might be complicated.
This point becomes evident if complexes with higher derivatives are studied.  One well-known example is the \emph{elasticity complex } in which second derivatives appear; cf. \cite{Pauly_2022_2}. Admittedly, suitable Finite Elements for this chain in $2D$ were introduced, e.g. in \cite{Arnold_rectangle_1,Arnold2002MixedFE_1}, but one needs a lot of degrees of freedom per element and the basis definition is not straightforward. Another way to use FEEC  is the ansatz  \emph{complexes from complexes}, where two or more Hilbert complexes are  used to discretize PDEs. This approach can be seen in \cite{QuennevilleBlair2015ANA}, but also in \cite{Arnold2007MixedFE} in the context of linear elasticity. Due to the increased number of complexes the authors in the last two publications are able to reduce the discretization  problem to classical de Rham sequences, which are well-understood and easy to approximate.\\
We want to follow this idea of coupled complexes in this article. To be more precise, we use the connection between de Rham chains and  linear elasticity as described in \cite{Arnold2007MixedFE,Falk} and combine the  underlying formulation with weakly imposed symmetry with the concept of isogeometric differential forms from \cite{Buffa2011IsogeometricDD}. Whereas already existing methods like in \cite{Arnold2007MixedFE,Rettung,Falk} mainly use  Finite Elements on   simplicial or polygonal meshes, we want to exploit IGA to incorporate computational domains with curved boundaries in $2D$ and $3D$.
Even if we are mainly interested in  linear elasticity, we start the article with a rather abstract setting, since it improves clarity. The document has the following structure.\\
Before we specify the abstract problem in Sec.  \ref{Sec:abstract_framework}, we briefly introduce  some notation. In Sec. \ref{Sec:application_lin_elas} we apply the ideas of Sec. \ref{Sec:abstract_framework} to the example of linear elasticity. For this purpose we define suitable B-spline spaces and prove the well-posedness of the scheme with weak symmetry. Key part of the well-posedness proof is an inf-sup stability result which is verified utilizing the \emph{macroelement technique}.
In  Sec. \ref{Sec:Numerical_examples} we consider several numerical examples for the linear elasticity application. \\
We finish with a short conclusion at the end of the document in Sec. \ref{Sec:Conclusion} .\\
Our main references are \cite{Arnold2007MixedFE,Falk,Bressan,BressanandSangalli,Buffa2011IsogeometricDD} and we use results from them at different places in the article.
 
 \subsection*{Notation}

 	In the section we introduce some notation and define several spaces. \\
 For some bounded Lipschitz domain $D \subset \mathbb{R}^d, \ d \in \mathbb{N}$ we write for the standard Sobolev spaces $H^0(D)=L^2(D),$ $ H^k(D), \ k \in \mathbb{N}$, where $L^2(D)$ stands for the Hilbert space of square-integrable functions endowed with the inner product $\langle \cdot , \cdot \rangle= \langle \cdot , \cdot \rangle_{L^2}$. The norms $\norm{\cdot}_{H^k}$ denote the classical Sobolev norms in $H^k(D)$, where  $\n{\cdot}_{H^0}= \n{\cdot}_{L^2}$. In case of vector- or matrix-valued mappings we can define Sobolev spaces, too, by requiring the component functions to be in suitable Sobolev spaces. We  use the following notation:  for ${v} \coloneqq (v_1,\dots , v_d), \ {v} \in {H}^k(D, \mathbb{R}^d) \colon \Leftrightarrow \ v_i \in H^k(D), \, \forall i  $ and for ${M} \coloneqq \big(M_{ij}\big)_{i,j=1}^{d,d}, \ {M} \in {H}^k(D,\mathbb{M}) \colon \Leftrightarrow \ M_{ij} \in H^k(D), \, \forall i, \forall j$. In particular, depending on the situation we have  $\mathbb{M}=\mathbb{R}^{2 \times 2}$ or $\mathbb{M}=\mathbb{R}^{3 \times 3}$ in the document below.  
 The inner product $\langle \cdot , \cdot \rangle_{L^2}$ introduces straightforwardly an inner product on ${L}^2(D,\mathbb{R}^d)$. For the definitions of the next  spaces and norms we follow partly \cite{Pauly2016OnCA} to introduce further notation. First, let us consider vector-valued mappings.  Then if $d=3$ we set 
 \begin{alignat*}{3}
 {H}(D,\textup{curl}) &\coloneqq \{ {v} \in {L}^2(D,\mathbb{R}^3) \ | \ \nabla \times {v} \in {L}^2(D,\mathbb{R}^3)  \}, \  \hspace{0.3cm} && \norm{ {v}}_{\c}^2  \coloneqq \norm{{v}}^2_{L^2} + \norm{\nabla \times {v}}^2_{L^2},\\
 \textup{and} \ \ {H}(D,\textup{div}) &\coloneqq \{ {v} \in {L}^2(D,\mathbb{R}^d) \ | \ \nabla \cdot {v} \in {L}^2(D)  \}, \   && \norm{ {v}}_{\d}^2  \coloneqq \norm{{v}}^2_{L^2} + \norm{\nabla \cdot {v}}^2_{L^2} \ \ \ \textup{if} \ d \in \mathbb{N}.
 \end{alignat*}
 \normalsize
 Above we wrote $\nabla$  for the classical nabla operator. The definitions for ${H}(D,\textup{curl}) , \ {H}(D,\textup{div}), $ and the corresponding norms can be generalized to the matrix setting by requiring that all the rows (as vector-valued mappings) are in the respective spaces. Here the curl $\nabla \times$ and divergence $\nabla \cdot$ act row-wise, too.  We  write $ {H}(D,\textup{curl},\mathbb{M}), \ {H}(D,\textup{div},\mathbb{M})$ and
  \begin{alignat*}{2}
 {H}(D,\textup{div},\mathbb{S}) &\coloneqq   {L}^2(D,\mathbb{S}) \cap {H}(D,\textup{div},\mathbb{M}), \    
 \end{alignat*} 
 where we denote with ${L}^2(D,\mathbb{S})$ the space of square-integrable symmetric matrix fields. If we write $H_0^1(D,\mathbb{R}^d)$ we mean the space of $H^1$ functions with component-wise zero boundary values in the sense of the trace theorem; see \cite{steinbach} for more information. But if we have a subscript '$0$' of the form $L^2_0(D,\mathbb{R}^d)$ it stands for the $L^2$ functions with zero mean.\\
  Later the next operators become useful
 \begin{align*}
  \textup{Skew}(\begin{bmatrix}
m_{11}& m_{12} \\ m_{21} &  m_{22} 
\end{bmatrix}) \coloneqq m_{21}-m_{12}, \ \  \ \ \ 
\textup{Skew}(\begin{bmatrix}
m_{11}& m_{12} & m_{13} \\ m_{21} &  m_{22} & m_{23} \\ m_{31} & m_{32} & m_{33}
\end{bmatrix}) \coloneqq \begin{bmatrix}
m_{32}-m_{23} \\ m_{13}-m_{31} \\ m_{21}-m_{12}
\end{bmatrix},
 \end{align*}
 which can be applied to matrix fields.
 
 Further, for a matrix $\f{M}=(M_{ij})$ we use an upper index $\f{M}^j$ to denote the j-th column and a lower index $\f{M}_i$ for the $i$-th row. Finally, we introduce the $2D$ curl operator $$\textup{curl}(v) \coloneqq (\partial_2v,-\partial_1 v)^T,  \ \ \ \ \textup{curl}((v_1,v_2)^T) \coloneqq \begin{bmatrix}
 \partial_2v_1 & -\partial_1 v_1 \\ \partial_2v_2 & -\partial_1 v_2
 \end{bmatrix},$$ i.e.
 $\partial_i$ is the partial derivative w.r.t. the $i$-th coordinate.
 
 After stating some basic notation we proceed  with the explanation of the abstract problem.

 \section{Abstract framework}
 \label{Sec:abstract_framework}
 In this section we study a special class of variational formulations derived from coupled Hilbert complexes. Although we are interested in the case of linear elasticity later, we choose here an abstract notion  for reasons of clearness.\\
For the theoretical background we use and recommend the publications \cite{ArnoldBook,Falk}. 
\begin{definition}[Closed Hilbert complexes]
A closed Hilbert complex is a sequence of Hilbert spaces $(W^k)_k$  together with closed densely defined linear operators $d^k \colon V^k \subset W^k \rightarrow W^{k+1}$, where $V^k$ is the domain of $d^k$ and further we have closed ranges $\mathfrak{R}(d^k) \subset \mathfrak{Ker}(d^{k+1}) \subset W^{k+1}$. This means  $d^{k+1} \circ d^{k} = 0, \ \forall k$.
\end{definition}
The domain spaces  $(V^k,d^k)$ define the corresponding domain complex which determines a bounded Hilbert complex, where the $V^k$ are equipped with the graph inner product $\langle v,w  \rangle_V \coloneqq \langle v,w  \rangle + \langle d^kv,d^kw  \rangle $. We write for the associated norms $\n{v }_V^2 \coloneqq \n{v }^2 + \n{d^k v}^2$. Here and in the following $\n{\cdot}$ stands for the norm induced by the underyling Hilbert spaces $W^k$ and we write just $\langle \cdot, \cdot \rangle= \langle \cdot, \cdot \rangle_{W^k}$. Such complexes are key objects within the theory of FEEC, where mixed formulations of Hodge Laplacians are considered. Especially the stable discretization  of various mixed problems are closely related to Hilbert complexes. That is why we will exploit this notion, too, in order to define an abstract mixed weak formulation below.

\subsection{Continuous problem}
Now  we restrict ourselves to  a special case of Hilbert complexes as explained in the assumption below.

\begin{assumption}[Coupled Hilbert complexes ]
	\label{Assumption:1}
	Let $(V^k,d^k)$ and $(\bar{V}^k,\bar{d}^k)$ be the domain complexes associated to the closed Hilbert complexes  $( W^k,d^k)$ and $( \bar{W}^k,\bar{d}^k)$ s.t.:
	\begin{itemize}
		\item[I.] $\bar{d}^{n-1} \colon \bar{V}^{n-1} \rightarrow \bar{V}^{n}$ and $d^{n-1} \colon V^{n-1} \rightarrow V^n$  are surjective.
		\item[II.] There are bounded mappings $S_{k} \colon V^{k} \rightarrow \bar{V}^{k+1}, \ k \in \{n-2,n-1 \}$ with $S_{n-1} {d}^{n-2}= \bar{d}^{n-1} S_{n-2}$.
		\item[III.] The mapping   $\bar{d}^{n-1}S_{n-2} \colon {V}^{n-2} \rightarrow \bar{V}^n$ is surjective.
	\end{itemize}
\end{assumption}

	 	The next lemma adapted from \cite{Arnold2015} connects an inf-sup condition with the existence of a bounded right inverse. This will show us that assumption point  \emph{III} is fulfilled if a special inf-sup condition is valid.
\begin{lemma}
	\label{Lemma:right-inverse}
	Let $R, Q$ be two Hilbert spaces.
	Let $T \colon R  \rightarrow Q  $ be a  bounded linear mapping   w.r.t the norms $\n{\cdot}_{R} , \ \n{\cdot}_{Q}$ induced by the inner products. Further assume 
	\begin{align*}
	\underset{{q} \in Q }{\inf} \ \underset{{w} \in R}{\sup}   \ \ \frac{ \langle  T {w}, {q}\rangle_Q}{\n{{w}}_{R} \ \n{{q}}_{Q}} \geq C_{\dag} >0.
	\end{align*} 
	Then there is for all ${q} \in Q$ an element ${w} \in R$ s.t. $T {w} = {q}$ and $\n{{w}}_{R} \leq \frac{1}{C_{\dag}} \n{{q}}_{Q}$.
\end{lemma}

\begin{proof}
	We follow the proof steps of \cite[Lemma 2]{Arnold2015}.\\
	For reasons of clarification we write for the $Q$ and $R$ inner product $\langle \cdot , \cdot \rangle_{Q}$,  $\langle \cdot , \cdot \rangle_{R}$ respectively. The Hilbert space adjoint $T^*$ exists and satisfies
	
	\begin{equation*}
	C_{\dag} \n{{q}}_{Q}  \leq \underset{{w} \in R}{\sup}   \ \ \frac{ \langle  T {w}, {q}\rangle_{Q}}{\n{{w}}_{R} } = \underset{{w} \in R}{\sup}   \ \ \frac{ \langle   {w}, T^*{q}\rangle_{R}}{\n{{w}}_{R} } = \n{T^* {q}}_{R}, \ \ q \in Q.
	\end{equation*} 
	Thus we have injectivity of $T^*$ and the left-inverse  of $T^*$ is bounded with operator norm $ \leq 1/C_{\dag}$. This implies the surjectivity of $T$ and  the right-inverse $T^{\dag}$ is bounded by $1/C_{\dag}$. 
\end{proof}

\begin{corollary}[Auxiliary inf-sup condition]
	If there is a constant $C_{IS}>0$ s.t. 	
\begin{align}
\label{eq:aux_inf_sup}
	\underset{{q} \in \bar{V}^n }{\inf} \ \underset{{w} \in V^{n-2}}{\sup}   \ \ \frac{ \langle  \bar{d}^{n-1}S_{n-2} {w}, {q}\rangle}{\n{{w}}_{V} \ \n{{q}}} \geq C_{IS} >0,
	\end{align}
	then the property III is valid.
\end{corollary}
\begin{proof}
	This follows directly from Lemma \ref{Lemma:right-inverse} with $T = \bar{d}^{n-1}S_{n-2} \colon V^{n-2} \rightarrow \bar{V}^n$.
\end{proof}

In Fig. \ref*{Fig:1} the relations between the different spaces are illustrated.
	\begin{figure}[h!]
	\centering
	\begin{tikzpicture}
	{
		
		\node at (4,1.4) {$V^{n-2} $};
		\node at (4,3) {$\bar{V}^{n-2} $};
		\node at (2,3) {$\dots $};
		\node at (2,1.4) {$\dots $};
		\node at (9.85,3) {$\{0\} $};
		\node at (9.85,1.4) {$\{0\} $};
		\node at (3.05-0.2,1.65) {$ {d}^{n-3} $};
		\node at (3.05-0.2,3.25) {$ \bar{d}^{n-3} $};
		
		\draw[->] (2.6-0.2,3) to (3.55-0.2,3);
		\draw[->] (2.6-0.2,1.4) to (3.55-0.2,1.4);

		{	\node at (8.1,1.4) {$V^{n}$};}

		{\node at (6+0.2,1.4) {$V^{n-1}$};}

		{		\node at (8.1,3) {$\bar{V}^{n}$};}

		\node at (8.9,1.6) {$0$};
		
		\node at (8.9,3.2) {$0$};
		\node at (5+0.1,1.65) {$d^{n-2} $};
		\node at (5+0.1,3.25) {$\bar{d}^{n-2} $};
		
		\node at (7.3,1.65) {$d^{n-1} $};

		\draw[->] (4.55+0.1,3) to (5.5+0.1,3);
		
		\draw[->] (4.55+0.1,1.4) to (5.5+0.1,1.4);

		\draw[->] (8.5,3) to (9.4,3);
		
		\draw[->] (8.5,1.4) to (9.4,1.4);

	}
	
	\node at (6+0.2,3) {$\bar{V}^{n-1}$};
	\node at (7.3,3.25) {$\bar{d}^{n-1}  $};
	From de Rham to linear elasticity (3D ) - 1

	\draw[->] (6.5+0.35,3) to (7.4+0.3,3);
	
	{	\draw[->] (3.96,1.7) to (6-0.4,2.6);
		\node at (4.35,2.35) {$S_{n-2}$};	 }

	{		\draw[->] (6.2,1.7) to (7.6,2.6);}

	{				\draw[->] (6.5+0.35,1.4) to (7.4+0.3,1.4);}



	\node at (6.5,2.35) {$S_{n-1}$};

	
	\end{tikzpicture}
	\caption{The diagram commutes and relates the different spaces.}
	\label{Fig:1}
\end{figure}
For the spaces and operators in Assumption \ref{Assumption:1}  we can introduce an abstract mixed weak form.
\begin{definition}[Abstract mixed weak form]
	\label{Def:1}
	Find $(\sigma,u,p) \in V^{n-1} \times V^n \times \bar{V}^n$ s.t. 
		\begin{alignat}{5}
		\label{eq:saddle_point_problem_cont}
	& \mathfrak{A}(  {\sigma} , {\tau} ) \ + \ &&\langle {u}, d^{n-1} {\tau}  \rangle \  + \ && \langle {p}, S_{n-1}{\tau} \rangle  \color{black}  &&= l_{n-1}(\tau), \ \hspace{0.2cm} &&\forall {\tau} \in V^{n-1}, \nonumber \\
	&\langle d^{n-1} {\sigma} , {v} \rangle &&  && &&= l_{n}(v), \ \hspace{0.4cm} && \forall {v} \in V^n, \\
	&\langle S_{n-1}{\sigma}, {q} \rangle &&  &&  &&= 0,  \ \hspace{0.4cm} &&\forall {q} \in \bar{V}^n, \nonumber
	\end{alignat}
	where $l_{k} \colon W^{k} \rightarrow \mathbb{R}$ are continuous linear forms and $\mathfrak{A} \colon W^{n-1} \times W^{n-1} \rightarrow \mathbb{R}$ is a continuous and coercive bilinear form.
\end{definition} 

To have a well-defined problem above  the obvious question arises: When is the weak form from above well-posed, i.e. which conditions lead to a unique solution? Since we observe a saddle-point problem structure, we just have to apply the Brezzi conditions from \cite{Brezzi} which can be summarized in our context  as follows.

  \begin{definition}[Brezzi stability conditions] There are positive constants
  	\label{Brezzi:inf-sup-condition_abstract} $C_{S1}, \ C_{S2}$ s.t. 
	\begin{align*}
	&\textup{(S1)}	\hspace*{0.5cm}	\mathfrak{A}(  \f{\tau}, \f{\tau} ) \geq C_{S1} \n{\f{\tau}}_{V}^2, \ \forall \f{\tau} \in kB,  	 \\	 		
	& \hspace*{1.4cm}	kB  \coloneqq \{ \f{\tau} \in V^{n-1} \ | \  \langle  d^{n-1} \f{\tau}, \f{v} \rangle + \langle  S_{n-1}\f{\tau}, \f{q} \rangle  =0,  \ \forall (\f{v} , \f{q}) \in {V}^n \times \bar{V}^n \}, \\
	&\textup{(S2)} \hspace*{0.5cm} \underset{(\f{v},\f{q} ) \in \f{V^n} \times \bar{V}^n }{\inf}    \ \ \  \underset{\f{\tau} \in \f{V}^n}{\sup} \hspace{0.1cm}\frac{\langle S_{n-1}\f{\tau}, \f{q} \rangle + \langle d^{n-1} \f{\tau}, \f{v} \rangle }{\norm{\f{\tau}}_{V} \big(\norm{\f{v}}+ \norm{\f{q}} \big)}\geq C_{S2} >0	.
	\end{align*}
\end{definition}
The first condition (S1) is obviously fulfilled  by Assumption  \ref{Assumption:1}, i.e. by  the surjectivity of $d^{n-1} \colon V^{n-1} \rightarrow V^n$, i.e.
$$\mathfrak{A}(  \f{\tau}, \f{\tau} )  \geq C \n{\f{\tau}}^2 =  C \n{\f{\tau}}_{V}^2, \ \  \forall \tau \in kB.$$
The inf-sup condition (S2) seems more problematic and deserves a closer look.
\begin{lemma}[Well-posedness]
	In view of Assump. \ref{Assumption:1} and Def.  \ref{Def:1} we obtain the well-posedness of the saddle-point problem \eqref{eq:saddle_point_problem_cont} in the mentioned definition. 
\end{lemma}

\begin{proof}
	With the remarks above, it is enough to check the validity of (S2). We orient ourselves towards the proof of Theorem 7.2 in \cite{Falk}. \\ Let $(v,q) \in V^n \times \bar{V}^n \backslash \{0 \times 0\}$ arbitrary, but fixed. By the surjectivity assumption regarding $d^{n-1}$ we find $\tau \in V^{n-1}$ with $v = d^{n-1}\tau$. 
	 Then, in view of the surjectivity of $\bar{d}^{n-1}S_{n-2}$ one can choose $\rho \in V^{n-2}$ s.t. $$\bar{d}^{n-1}S_{n-2}\rho = q - S_{n-1}\tau \in \bar{V}^n.$$
	Finally, we set $$\sigma \coloneqq d^{n-2}\rho + \tau \in V^{n-1}.$$ 
	In particular, we see $d^{n-1} \sigma =0 + d^{n-1} \tau =v$ and \begin{align*}
	S_{n-1}\sigma &= S_{n-1}d^{n-2}\rho + S_{n-1}\tau  =  
	\bar{d}^{n-1}S_{n-2}\rho + S_{n-1} \tau \\ &=  q - S_{n-1}\tau + S_{n-1}\tau \\ & = q.
	\end{align*}
		By the arbitrariness of $(v,q)$ we get that  $$T \colon V^{n-1} \rightarrow V^n \times \bar{V}^{n} \ , \ \kappa \mapsto d^{n-1}\kappa  \times S_{n-1}\kappa$$ is a continuous and surjective mapping between Banach spaces.
		By the open mapping theorem we have the existence of a $0 < \delta$ s.t.
		$$B_{\delta}(0 \times 0) \subset T(B_1(0)),$$
		where $B_r( \cdot )$ denote the respective open $r$-neighborhoods.\\
		Let $\n{v} + \n{q} = C >0$. Then we know from above the existence of $\sigma \in V^{n-1}$ with $T\sigma = (v,q)$. Hence, 
		$$\frac{\delta}{2C} (v,q) \in B_{\delta}(0 \times 0) \subset T(B_1(0)).$$ And consequently, there is a $\tilde{\sigma} \in V^{n-1}, \ \n{\tilde{\sigma}}_V < 1$ with $T \tilde{\sigma} = \frac{\delta}{2C} (v,q)$ and $$T \frac{2C}{\delta} \tilde{\sigma} = (v,q), \ \ \n{\frac{2C}{\delta} \tilde{\sigma}}_V \leq \frac{2C}{\delta} \leq \frac{2}{\delta} (\n{v} + \n{q}).$$ So we see the existence of a bounded right-inverse of the auxiliary mapping $T$ which implies directly the inf-sup condition (S2).

\end{proof}

After we justified the well-posedness of \eqref{eq:saddle_point_problem_cont}, we look at structure-preserving discretizations of the continuous bottom complex in Fig. \ref{Fig:1}. Using the theoretical foundations derived in \cite{ArnoldBook}, we  clarify the term \emph{structure-preserving} in the context of Hilbert complexes.

\subsection{Discretization}
For the next definition we follow \cite[Section 5.2]{ArnoldBook}
\begin{definition}[Structure-preserving discretization of Hilbert complexes]
A family of finite-dimensional spaces $(V_h^k)_h$ discretize the domain complex $(V^k,d^k)$ associated to the closed Hilbert complex $(W^k,d^k)$  structure-preserving, if the next three conditions are fulfilled:
\begin{itemize}
	\item \textbf{Approximation property}: $\lim_{h \rightarrow 0} \ \underset{v_h \in V_h^j}{\inf}\n{w - v_h}_V = 0 , \ \forall w \in V^j, \  \forall j,$
	\item \textbf{Subcomplex property}: $V_h^j \subset V^j$ and $d^{j} V_h^j \subset V_h^{j+1}$ for all $j$,
	\item \textbf{Bounded cochain projections}: There are projections $\Pi_h^j \colon V^j \rightarrow V_h^j$ that are  uniformly $W$-bounded w.r.t. $h$ and that are compatible with the operators $d^j$, i.e. it holds $d^j \Pi_h^j = \Pi_h^{j+1}d^j$.  
\end{itemize}
\end{definition}
Such special discretizations preserve the key properties of the continuous pendant s.t. proofs and properties regarding the continuous complex can be mimicked in the discrete setting. This is the starting point for stable mixed methods. For more details we again refer to \cite{ArnoldBook}. \\
 That is the reason why we state the next assumption. 
  \begin{assumption}[]
  	\label{Assumption:2}
  	In following let $V_h^k \subset V^k$  define structure-preserving discretizations of the bottom complex from 
  	 Fig \ref{Fig:1}   and let $\bar{V}_h^n \subset \bar{V}^n$.
  \end{assumption}
  For the $W$-orthogonal projection from $\bar{V}^n$ onto $\bar{V}_h^n$ we write $\bar{\Pi}_h^n$, i.e. 
  we get the subsequent commuting diagram:
	\begin{figure}[h!]
	\centering
	\begin{tikzpicture}
	
	\node at (4,1.4-1.6) {$V_h^{n-2} $};
		\node at (2,1.4-1.6) {$\dots $};		
		\node at (3.05-0.2,1.65-1.6) {$ {d}^{n-3} $};
		\draw[->] (2.6-0.2,1.4-1.6) to (3.55-0.2,1.4-1.6);
		\node at (8.1,1.4-1.6) {$V^{n}_h$};
		\node at (6+0.2,1.4-1.6) {$V^{n-1}_h$};	
	\node at (5+0.1,1.65-1.6) {$d^{n-2} $};	
	\node at (7.3,1.65-1.6) {$d^{n-1} $};			
	\draw[->] (4.55+0.1,3-1.6) to (5.5+0.1,3-1.6);	
	\draw[->] (4.55+0.1,1.4-1.6) to (5.5+0.1,1.4-1.6);	
	\draw[->] (6.5+0.35,1.4-1.6) to (7.4+0.3,1.4-1.6);

	
\node at (8.1,1.4+3.2) {$\bar{V}^{n}_h$};


\draw[<-] (8.1,1.4+2.8) to (8.1,1.4+2);	

\draw[->] (8.1,1.4+2.8-3.2) to (8.1,1.4+2-3.2);	
\draw[->] (6+0.2,1.4+2.8-3.2) to (6+0.2,1.4+2-3.2);	
\draw[->] (4,1.4+2.8-3.2) to (4,1.4+2-3.2);

\node[left] at (8.2,1.4+2.8-3.2-0.3) {${\Pi}_h^{n}$};	
\node[left] at (6+0.3,1.4+2.8-3.2-0.3) {${\Pi}_h^{n-1} $};	
\node[left] at (4+0.1,1.4+2.8-3.2-0.3) {${\Pi}_h^{n-2} $};

\node[left] at (8.2,1.4+2.8-0.4) {$\bar{\Pi}_h^{n}$};	

	{
		
		\node at (4,1.4) {$V^{n-2} $};
		\node at (4,3) {$\bar{V}^{n-2} $};
		\node at (2,3) {$\dots $};
		\node at (2,1.4) {$\dots $};
		\node at (9.85,3) {$\{0\} $};
		\node at (9.85,1.4) {$\{0\} $};
		\node at (3.05-0.2,1.65) {$ {d}^{n-3} $};
		\node at (3.05-0.2,3.25) {$ \bar{d}^{n-3} $};
		
		\draw[->] (2.6-0.2,3) to (3.55-0.2,3);
		\draw[->] (2.6-0.2,1.4) to (3.55-0.2,1.4);

		{	\node at (8.1,1.4) {$V^{n}$};}

		{\node at (6+0.2,1.4) {$V^{n-1}$};}

		{		\node at (8.1,3) {$\bar{V}^{n}$};}

		\node at (8.9,1.6) {$0$};
		
		\node at (8.9,3.2) {$0$};
		\node at (5+0.1,1.65) {$d^{n-2} $};
		\node at (5+0.1,3.25) {$\bar{d}^{n-2} $};
		
		\node at (7.3,1.65) {$d^{n-1} $};

		\draw[->] (4.55+0.1,3) to (5.5+0.1,3);
		
		\draw[->] (4.55+0.1,1.4) to (5.5+0.1,1.4);

		\draw[->] (8.5,3) to (9.4,3);
		
		\draw[->] (8.5,1.4) to (9.4,1.4);

	}
	
	\node at (6+0.2,3) {$\bar{V}^{n-1}$};
	\node at (7.3,3.25) {$\bar{d}^{n-1}  $};

	\draw[->] (6.5+0.35,3) to (7.4+0.3,3);
	
	{	\draw[->] (3.96,1.7) to (6-0.4,2.6);
		\node at (4.35,2.35) {$S_{n-2}$};	 }

	{		\draw[->] (6.2,1.7) to (7.6,2.6);}

	{				\draw[->] (6.5+0.35,1.4) to (7.4+0.3,1.4);}



	\node at (6.5,2.35) {$S_{n-1}$};

	
	\end{tikzpicture}
	\caption{The diagram commutes.}
	\label{Fig:2}
\end{figure}

It is natural to consider the discretized version of the weak form in Def.  \ref{Def:1}.
\begin{definition}[Discrete weak form]
	\label{Def:2}
Find $(\sigma_h,u_h,p_h) \in V_h^{n-1} \times V_h^n \times \bar{V}^n_h$ s.t. 
\begin{alignat}{5}
\label{eq:abstract_mixed_form_discrete}
& \mathfrak{A}(  {\sigma}_h , {\tau}_h ) \ + \ &&\langle {u}_h, d^{n-1} {\tau}_h  \rangle \  + \ && \langle {p}_h, S_{n-1}{\tau}_h \rangle  \color{black}  &&= l_{n-1}(\tau_h), \ \hspace{0.2cm} &&\forall {\tau_h} \in V_h^{n-1}, \nonumber\\
&\langle d^{n-1} {\sigma_h} , {v}_h \rangle &&  && &&= l_{n}(v_h), \ \hspace{0.4cm} && \forall {v_h} \in V_h^n, \\
&\langle S_{n-1}{\sigma}_h, {q_h} \rangle &&  &&  &&= 0,  \ \hspace{0.4cm} &&\forall {q}_h \in \bar{V}_h^n. \nonumber
\end{alignat}
\end{definition}

Now the question arises, whether we get well-posedness in the discrete setting. Unfortunately, it is not enough to just use the spaces defined by  two separate structure-preserving discretizations of both complexes in Fig. \ref{Fig:1}. The connecting mappings $S_k$ claim  a suitable connection of the three spaces $V_h^{n-1} \times V_h^n \times \bar{V}^n_h$. 
By the next lemma it is clear, that we need also the discrete version of \eqref{eq:aux_inf_sup}.
\begin{lemma}[Well-posedness in the discrete setting]
	\label{Lemma:well-posedness_discrete_abstract}
	Let the Assump. \ref{Assumption:2} hold true. Further, we assume the next auxiliary inf-sup condition
		\begin{align}
		\label{eq:aux_inf_sup_discrete}
	\underset{{q}_h \in \bar{V}_h^n }{\inf} \ \underset{{w}_h \in V_h^{n-2}}{\sup}   \ \ \frac{ \langle  \bar{d}^{n-1}S_{n-2} {w}_h, {q}_h\rangle}{\n{{w}_h}_{V} \ \n{{q}_h}} \geq C_{IS} >0, \ \ \forall h,
	\end{align}
	where $C_{IS}$ does not dependent on $h$. \\
	Then the discrete formulation   \eqref{eq:abstract_mixed_form_discrete} is well-posed for all $h$.
\end{lemma}
\begin{proof}
	We have well-posedness if the Brezzi stability conditions are fulfilled.
	In fact, the proof here is  similar to the continuous setting. Hence, on the one hand we look at the discrete analogous of (S1).
Set $kB_h \coloneqq \{ \f{\tau}_h \in V_h^{n-1} \ | \  \langle  d^{n-1} \f{\tau}_h, \f{v}_h \rangle + \langle  S_{n-1}\f{\tau}_h, \f{q}_h \rangle  =0,  \ \forall (\f{v}_h , \f{q}_h) \in {V}^n_h \times \bar{V}^n_h \}$. The subcomplex property implies $d^{n-1} \f{\tau}_h=0$ and consequently the coercivity relation $\mathfrak{A}(  \f{\tau}_h, \f{\tau}_h ) \geq C \n{\f{\tau}_h}_{V}^2, \ \forall \f{\tau}_h \in kB_h$ is obvious. \\
On the other hand, we check next the inf-sup condition (S2) for the finite-dimensional spaces. In the following, let $(v_h,q_h) \in V_h^n \times \bar{V}_h^n \backslash \{0 \times 0 \}$ arbitrary but fixed. By the surjectivity assumption regarding $d^{n-1}$ we find $\tau \in V^{n-1}$ with $v_h = d^{n-1}\tau$, i.e. for $\tau_h \coloneqq \Pi_h^{n-1}\tau$ we see $$d^{n-1}\tau_h = d^{n-1}\Pi_h^{n-1}\tau= \Pi_h^{n} d^{n-1}\tau = v_h.$$  Furthermore we find $\tau^{\mathfrak{Ker}}_h \in \mathfrak{Ker}(d^{n-1}) \cap V_h^{n-1}, \ \tau^{\perp}_h \in \big(\mathfrak{Ker}(d^{n-1})\cap V_h^{n-1} \big)^{\perp}$ s.t. $\tau_h= \tau^{\mathfrak{Ker}}_h + \tau^{\perp}_h$, where the orthogonal complement is taken within $V_h^{n-1}$. Since $(V_h^k,d^k)$ define a structure-preserving complex discretization we can apply the discrete Poincar\'{e}  inequality \cite[Theorem 3.6]{Arnold2010FiniteEE}, i.e.  
$$C_P \n{d^{n-1}\tau_h} = C_P \n{d^{n-1}\tau^{\perp}_h} \geq  \n{\tau^{\perp}_h}_V.$$
Above, $C_P$ does not dependent on $h$, but on the operator norm  $\n{\Pi_h^{n-1}}$. \\ Then, obviously it is $$q_h-\bar{\Pi}_h^nS_{n-1} \tau_h^{\perp} \in \bar{V}_h^{n},$$
where $\bar{\Pi}_h^n \colon \bar{W}^n \rightarrow \bar{V}_h^n$ is the $\bar{W}^n$-orthogonal projection onto $\bar{V}_h^n$. 
Due to the assumed inf-sup stability, we can apply Lemma \ref{Lemma:right-inverse}, where we choose $$ T \colon V_h^{n-2} \rightarrow \bar{V}_h^n \ , \ w \mapsto \bar{\Pi}_h^n\bar{d}^{n-1}S_{n-2}w.$$ 
One notes the fact $\langle  \bar{\Pi}_h^n\bar{d}^{n-1}S_{n-2} {w}, {q}\rangle= \langle  \bar{d}^{n-1}S_{n-2} {w}, {q}\rangle, \ \forall w \in V_h^{n-2}, \ \forall q \in \bar{V}_h^n$.
Hence, there is a $w_h \in V_h^{n-2}$ with $\n{w_h}_V \leq \frac{1}{C_{IS}} \n{q_h-\bar{\Pi}_h^nS_{n-1} \tau_h^{\perp} }$  and $$\bar{\Pi}_h^n\bar{d}^{n-1}S_{n-2}w_h =T w_h =q_h-\bar{\Pi}_h^nS_{n-1} \tau_h^{\perp}   .$$
Next we define $$\sigma_h \coloneqq d^{n-2}w_h + \tau_h^{\perp} \in V_h^{n-1}.$$ We have  $d^{n-1} \sigma_h = d^{n-1} \tau_h^{\perp}= d^{n-1} \tau_h = v_h$. Besides, it is \begin{align*}
\langle S_{n-1}\sigma_h, q_h \rangle  &=  \langle \bar{\Pi}_h^nS_{n-1}\sigma_h, q_h \rangle \\&= \langle \bar{\Pi}_h^nS_{n-1}d^{n-2}w_h, q_h \rangle + \langle \bar{\Pi}_h^nS_{n-1}\tau_h^{\perp}, q_h \rangle \\
&= \langle \bar{\Pi}_h^n\bar{d}^{n-1}S_{n-2}w_h, q_h \rangle + \langle \bar{\Pi}_h^nS_{n-1}\tau_h^{\perp}, q_h \rangle \\
&= \langle q_h-\bar{\Pi}_h^nS_{n-1} \tau_h^{\perp} , q_h \rangle + \langle \bar{\Pi}_h^nS_{n-1}\tau_h^{\perp}, q_h \rangle = \n{q_h}^2.
\end{align*} 
Thus, we obtain $\langle S_{n-1}\f{\sigma}_h, \f{q}_h \rangle + \langle d^{n-1} \f{\sigma}_h, \f{v}_h \rangle  = \n{q_h}^2 + \n{v_h}^2$. In view of the next  inequality chain \begin{align*}
\n{d^{n-1}\sigma_h} +  \n{\sigma_h}  &\leq \n{v_h} + \n{w_h}_V + \n{\tau_h^{\perp}} \\ 
& \leq  \n{v_h} +  \ \frac{1}{C_{IS}} \n{q_h } + \frac{1}{C_{IS}} \n{\bar{\Pi}_h^nS_{n-1} \tau_h^{\perp}} +  C_P\n{v_h} \\
& \leq \n{v_h} +  \ \frac{1}{C_{IS}} \n{q_h } + \frac{1}{C_{IS}} C_{S_{n-1}} \n{ \tau_h^{\perp}} +  C_P\n{v_h} \\ 
& \leq C_2 \Big( \n{v_h}+ \n{q_h} \Big),
\end{align*} we see the validity of (S2) in the discrete case. We remark that in the last line we can choose $C_2= 1+C_P+\frac{1}{C_{IS}}+\frac{C_P C_{S_{n-1}}}{C_{IS}}$, with $C_{S_{n-1}}$ denoting the continuity constant of $S_{n-1}$. This finishes the proof.
\end{proof}

The derived well-posedness result gives us directly a quasi-optimal estimate.
\begin{corollary}[Quasi-optimality]
	\label{Corollary:estimate}
	We assume the requirements of the last lemma. Furthermore, let $(\sigma,u,p)$ be the solution of the continuous problem \eqref{eq:saddle_point_problem_cont}.  Then, there is a constant \\
	$C_{QO} = C_{QO}(C_{S1},C_{S2},C_{\mathfrak{A}},C_{S_{n-1}}) < \infty$ independent of $h$ s.t. we have the error estimate
	   \begin{equation*}
	\norm{(\f{\sigma}-\f{\sigma}_h,  \f{u}-\f{u}_h,\f{p}-\f{p}_h)}_\mathcal{B} \leq C_{QO} \underset{\tau_h \in V_h^{n-1},  \ v_h \in V_h^{n}, \ q_h \in \bar{V}_h^n}{\inf}\Big(  \norm{(\f{\sigma}-\f{\tau}_h,\f{u}-\f{v}_h,\f{p}-\f{q}_h)}_\mathcal{B}\Big),
	\end{equation*}
	where
	$\norm{(\f{\sigma},\f{u},\f{p})}_\mathcal{B}^2 \coloneqq \norm{\f{\sigma}}_{V}^2 + \norm{\f{u}}^2 + \norm{\f{p}}^2$ and $C_{\mathfrak{A}}, \ C_{S_{n-1}}$ denote the continuity constants of $\mathfrak{A}$ and $S_{n-1}$ and $C_{S1}, \ C_{S2}$ are the constants from Def. \ref{Brezzi:inf-sup-condition_abstract}. 
	
\end{corollary}
\begin{proof}
	This follows from the general theory on saddle-point problems; see \cite{Brezzi,Brezzi2}.
\end{proof}

In the next part we focus on a weak formulation with weakly imposed symmetry in the context of linear elasticity and show how it fits to the above abstract setting.
\section{Application to linear elasticity}
\label{Sec:application_lin_elas}
Before we use the connections between Hilbert complexes and the theory of linear elasticity as established in \cite{Arnold2007MixedFE,Arnold2015,Falk}, we briefly explain the elasticity formulation with weakly imposed symmetry.

\subsection{Linear elasticity with weakly imposed symmetry}
\label{Sec:conncetion_deRham_elasticity}
 In linear elasticity a domain $\o \subset \mathbb{R}^{n}, \ n \in \{2,3 \}$ is identified with some elastic body that is loaded with some force distribution ${f}$.  In our case $\o$ will be  a bounded Lipschitz domain. The state of the body can be expressed through the displacement field $\f{u} \colon \o \rightarrow \mathbb{R}^n$ that measures the difference between current configuration $\o + \f{u}$ and initial configuration, namely $\o$.  Governing equations of the underlying model guaranteeing a physically reasonable description of the deformation process are 
\begin{alignat}{2}
\label{linear_elas_1}
\f{A} \f{\sigma} &= \f{\varepsilon}(\f{u}) \hspace{0.3cm}  &&\textup{in} \ \ \o, \\
\nabla \cdot \sigma &= \f{f} \hspace{0.1cm} &&\textup{in} \ \ \o.
\label{linear_elas_2}
\end{alignat}
Above $\f{\varepsilon}(\f{u})= \frac{1}{2} (\nabla\f{u}+(\nabla \f{u})^T) $ denotes the symmetric gradient and  corresponds to the mechanical strain in the linear theory. Further, $ \f{A} \colon \mathbb{S} \rightarrow\mathbb{S} $ is the  compliance tensor, which comprises the material  properties and $\f{\sigma} \colon  \o \rightarrow \mathbb{S}$ is the symmetric stress tensor. $\mathbb{S}$ stands for the space of symmetric matrices. 
  A classical conservation of momentum argument relates the  stress tensor and the force $\f{f}$ and as a result one obtains equation \eqref{linear_elas_2}. An important, often considered special case are homogeneous and isotropic materials. Then the compliance tensor simplifies to $$\f{A} \sigma = \frac{1}{2 \mu} \Big( \f{\sigma}- \frac{\lambda}{n\lambda+2\mu} \textup{tr}(\f{\sigma}) \f{I}\Big),$$ and the material properties are condensed in form of  Lamé coefficients $\lambda \geq 0, \ \mu>0$.  Within this document we assume that $A$ defines a bounded and coercive  bilinear form $\langle A \ \cdot, \cdot \rangle  $ on $L^2(\o,\mathbb{M})$. Actually, we have to specify some boundary conditions to complete the model, where we first consider the easy case of pure Dirichlet conditions, i.e. we assume $\f{u}= \f{u}_D $ on $\Gamma = \partial \o$ for some given $\f{u}_D \in H^{1/2}(\Gamma,\mathbb{R}^n)$\footnote{$H^{1/2}(\Gamma,\mathbb{R}^n)$ is the image of the trace mapping $\gamma \colon H^1(\o,\mathbb{R}^n) \rightarrow H^{1/2}(\Gamma,\mathbb{R}^n)$. For more information about trace spaces we refer to \cite{steinbach}.}. Supposing the invertibility of $\f{A}$ one gets a formulation only dependent on the displacement variable as combination of \eqref{linear_elas_1} and \eqref{linear_elas_2}, namely \begin{equation}
\label{primal_formulation_Dirichlet}
\nabla \cdot (\f{C} \f{\varepsilon}(\f{u})) = \f{f} ,  \ \ \ \f{C} \coloneqq \f{{A}}^{-1} , 
\end{equation}where $\f{C}$ is the so-called  elasticity tensor.  
One standard approach to solve the latter equation in case of homogeneous boundary conditions  weakly, is the primal formulation: $$ \textup{Find} \ \ \f{u}\in \f{H}_0^1(\o,\mathbb{R}^n)  \ \ \textup{s.t.} \ \  \langle \f{C} \f{\varepsilon}(u), \varepsilon(\f{v}) \rangle = \langle \f{f}, \f{v}\rangle,  \ \ \forall \f{v} \in \f{H}_0^1(\o,\mathbb{R}^n), $$  
 In general the latter formulation is  unstable in the (nearly) incompressible regime, i.e.  if $\lambda$ is large. Another point is the non-locality of stress-strain relations in more general material situations as viscoelasticity. Consequently, the primal weak formulation is not optimal also in view of generalizability. Therefore one can find in literature  the Hellinger-Reissner mixed formulation  that reads: Find $\f{\sigma} \in {H}(\Omega,\textup{div},\s),$  $\f{v} \in {L}^2(\Omega,\mathbb{R}^n)$ s.t. 
\begin{alignat}{4}
\label{weak_form_strong_symmetry_contiuous_1}
&\langle {A} {\sigma} , {\tau} \rangle \ + \ &&\langle u,  \nabla \cdot {\tau}  \rangle \  \color{black}  &&= \langle {u}_D, {\tau} \cdot \nu \rangle_{\Gamma}, \ \hspace{0.4cm} &&\forall {\tau} \in {H}(\Omega,\textup{div},\s), \nonumber\\
&\langle \nabla \cdot {\sigma} , {v} \rangle &&  &&= \langle \f{f},{v}\rangle, \ \hspace{0.4cm} && \forall {v} \in {L}^2(\Omega,\mathbb{R}^n),
\end{alignat}
where $\nu$  in $\langle {u}_D, {\tau} \cdot \nu \rangle_{\Gamma}$ stands for the outer unit normal to the boundary $\Gamma$.

	\label{section_weak_form_weak_symmetry}
To obtain a mixed weak formulation for the elasticity system that does not require test spaces with strong symmetry, i.e. subspaces of  $H(\o,\d,\mathbb{S}) $, Arnold et al. introduced in \cite{Arnold2007MixedFE} a weak form with weakly imposed symmetry. It can be rewritten equivalently as: \emph{Find}	${\sigma} \in {H}(\Omega,\textup{div},\mathbb{M}), \ {u} \in {L}^2(\Omega,\mathbb{R}^n)$ and $\f{p} \in {{{L}}}^2(\Omega,\mathbb{R}^{s(n)}),\  s(n)= 2n-3$ s.t. 
\begin{alignat}{5}
\label{weak_form_wek_symmetry_contiuous}
&\langle {A} {\sigma} , {\tau} \rangle \ + \ &&\langle u,\nabla \cdot {\tau}  \rangle \  + \ && \langle   {p}, \textup{Skew} \, {\tau} \rangle  \color{black}  &&= \langle {u}_D, {\tau} \cdot \nu \rangle_{\Gamma}, \ \hspace{0.4cm} &&\forall {\tau} \in {H}(\Omega,\textup{div},\mathbb{M}), \nonumber\\
&\langle \nabla \cdot {\sigma} , {v} \rangle &&  && &&= \langle {f},{v}\rangle, \ \hspace{0.4cm} && \forall {v} \in {L}^2(\Omega,\mathbb{R}^n), \\
&\langle \textup{Skew} \, {\sigma}, {q} \rangle &&  &&  &&= 0,  \ \hspace{0.4cm} &&\forall {q} \in L^2(\Omega,\mathbb{R}^{s(n)}). \nonumber
\end{alignat}
We get the last form if one drops  the symmetry condition  and replaces it by a variational statement, namely we assume 
\begin{equation*}
\langle \textup{Skew} \, \t{\sigma} , q \rangle  = 0 , \  \forall q \in L^2(\Omega,\mathbb{R}^{s(n)}), 
\end{equation*}	 
which is equivalent to 
$$  \int_{\Omega}  q_i \ (\sigma_{ij}-\sigma_{ji}) dx =0, \  \ \forall q_i \in L^2(\Omega), \ i \neq j \ .$$
The mentioned constraint is incorporated by means of a Lagrange multiplier
and this explains  the additional variable $\f{p}$  and the last line in \eqref{weak_form_wek_symmetry_contiuous}. It is easy to see the equivalence of the two formulations in the continuous setting, where the Lagrange multiplier is related to the solution $(\f{\sigma},\f{u})$ through $\f{p}= \frac{1}{2}\textup{Skew}(\nabla\f{u}) $. 
 Main advantage of the latter problem statement is the obtained flexibility for discretization approaches compared to \eqref{weak_form_strong_symmetry_contiuous_1}, but the price to pay is the increased number of degrees of freedom.

We observe the similar structure of \eqref{weak_form_wek_symmetry_contiuous} and \eqref{eq:saddle_point_problem_cont}. In fact, the former is a special case of the abstract problem as we show below. First we look at the two-dimensional case ($n=2$) and define the spaces 
\begin{alignat*}{3}
V^{0} & ={H}^1(\o, \mathbb{R}^2), \  \    V^1 &&= {H}(\o,\d, \mathbb{M}), \  \ V^2 &&= {L}^2(\o, \mathbb{R}^2),\\
W^{0} & =L^2(\o, \mathbb{R}^2), \   \    \ W^1 &&= L^2(\o, \mathbb{M}), \  \ \ \ \  \ W^2 &&= {L}^2(\o, \mathbb{R}^2),\\
\bar{V}^{0} & ={H}^1(\o, \mathbb{R}), \  \  \ \, \bar{V}^1 &&= {H}(\o,\d,\mathbb{R}^2), \    \, \bar{V}^2 &&= {L}^2(\o, \mathbb{R}),\\
\bar{W}^{0} & =L^2(\o, \mathbb{R}), \   \  \,  \ \bar{W}^1 &&= L^2(\o, \mathbb{R}^2), \  \ \ \ \  \ \bar{W}^2 &&= {L}^2(\o, \mathbb{R}).
\end{alignat*}
For $n=3$, the three-dimensional case, we choose:
\begin{alignat*}{4}
V^{0} & ={H}^1(\o, \mathbb{R}^3), \ V^1 && ={H}(\o,\c, \mathbb{M}), \   V^2 &&= {H}(\o,\d, \mathbb{M}), \   V^3 &&= {L}^2(\o, \mathbb{R}^3),\\
W^{0} & =L^2(\o, \mathbb{R}^3), \ W^1 && =L^2(\o, \mathbb{M}), \   \ \ \  \ \ W^2 &&= L^2(\o, \mathbb{M}), \  \ \ \ \  W^3 &&= {L}^2(\o, \mathbb{R}^3).
\end{alignat*}
And further we set $\bar{V}^i = V^i, \ \bar{W}^i=W^i$ for the $3D$ case.
Then, with the above choice of spaces we can construct closed Hilbert complexes with underyling $L^2$ inner products. More precisely, we obtain the vector-valued de Rham chains in Fig.  \ref{Fig:3} and Fig. \ref{Fig:4}, where \begin{equation*}
\Xi M \coloneqq M^T - \textup{tr}(M)I, \ \textup{i.e.}   \ \ \ \ \ \ \Xi^{-1} = M^T - \frac{1}{2}\textup{tr}(M)I, \ \ \textup{see} \ \cite[\textup{Sec. 4.1}]{Falk}.
\end{equation*}

	\begin{figure}[h!]
	\centering
	\begin{tikzpicture}
	
		\node at (2.15,1.4) {$\mathbb{R}^2 $};
		\node at (4,1.4) {$H^1(\o,\mathbb{R}^2) $};
		\node at (7.1,1.4) {${H}(\o,\d, \mathbb{M})$};
		\node at (10.1,1.4) {${L}^2(\o, \mathbb{R}^2)$};		
		\node at (12,1.4) {$\{0\} $};
		
		\node at (2.7,1.6) {$\iota $};
		\node at (5.4,1.65) {$\textup{curl} $};
		\node at (8.7,1.65) {$\nabla \cdot  $};		
		\node at (11.2,1.6) {$0$};
	    			
     \draw[->] (2.4,1.4) to (3,1.4);
     \draw[->] (4.55+0.4,1.4) to (5.5+0.4,1.4);
     \draw[->] (8.3,1.4) to (9.25,1.4);
     \draw[->] (11,1.4) to (11.6,1.4);

	\node at (2.2,3) {$\mathbb{R} $};
	\node at (4,3) {$H^1(\o,\mathbb{R}) $};
	\node at (7.1,3) {${H}(\o,\d, \mathbb{R}^2)$};
	\node at (10.1,3) {${L}^2(\o, \mathbb{R})$};		
	\node at (12,3) {$\{0\} $};
	
	\node at (2.7,3.2) {$\iota $};
	\node at (5.4,3.25) {$\textup{curl} $};
	\node at (8.7,3.25) {$\nabla \cdot  $};		
	\node at (11.2,3.2) {$0$};
	
	\draw[->] (2.4,3) to (3,3);
	\draw[->] (4.55+0.4,3) to (5.5+0.4,3);
	\draw[->] (8.3,3) to (9.25,3);
	\draw[->] (11,3) to (11.6,3);

		{\draw[->] (6.2+1.5,1.7) to (7.6+1.5,2.6);		
		\node at (8,2.35) {$\textup{Skew}$};}
	
		{	\draw[->] (3.96+0.5,1.7) to (6.1,2.6);
		\node at (5.05,2.35) {$\textup{I}$};	 }
			
	\end{tikzpicture}
	\caption{Here we display the diagram in Fig. \ref{Fig:1} adapted to the $2D$ Hellinger-Reissner formulation with weak symmetry.}
	\label{Fig:3}
\end{figure}

	\begin{figure}[h!]
	\centering
	\begin{tikzpicture}
	
	\node at (2.15,1.4) {$\mathbb{R}^3 $};
	\node at (4,1.4) {$H^1(\o,\mathbb{R}^3) $};
	\node at (7.1,1.4) {${H}(\o,\c, \mathbb{M})$};
	\node at (10.45,1.4) {${H}(\o,\d, \mathbb{M})$};	
	\node at (13.5,1.4) {$L^2(\o,\mathbb{R}^3) $};	
	\node at (15.3,1.4) {$\{0\} $};
	
	\node at (2.7,1.6) {$\iota $};
	\node at (5.4,1.65) {$ \nabla $};
	\node at (8.7,1.65) {$\nabla \times  $};	
	\node at (12,1.6) {$\nabla \cdot $};	
	\node at (14.6,1.6) {$0$};
	
	\draw[->] (2.4,1.4) to (3,1.4);
	\draw[->] (4.55+0.4,1.4) to (5.5+0.4,1.4);
	\draw[->] (8.3,1.4) to (9.25,1.4);
	\draw[->] (11.65,1.4) to (12.6,1.4);
	\draw[->] (14.4,1.4) to (14.9,1.4);

\node at (2.15,3) {$\mathbb{R}^3 $};
\node at (4,3) {$H^1(\o,\mathbb{R}^3) $};
\node at (7.1,3) {${H}(\o,\c, \mathbb{M})$};
\node at (10.45,3) {${H}(\o,\d, \mathbb{M})$};	
\node at (13.5,3) {$L^2(\o,\mathbb{R}^3) $};	
\node at (15.3,3) {$\{0\} $};

\node at (2.7,3.2) {$\iota $};
\node at (5.4,3.25) {$ \nabla $};
\node at (8.7,3.25) {$\nabla \times  $};	
\node at (12,3.2) {$\nabla \cdot $};	
\node at (14.6,3.2) {$0$};

\draw[->] (2.4,3) to (3,3);
\draw[->] (4.55+0.4,3) to (5.5+0.4,3);
\draw[->] (8.3,3) to (9.25,3);
\draw[->] (11.65,3) to (12.6,3);
\draw[->] (14.4,3) to (14.9,3);

	{\draw[->] (6.2+5,1.7) to (7.6+5,2.6);		
		\node at (11.5,2.35) {$\textup{Skew}$};}
	
	{	\draw[->] (3.96+3.7,1.7) to (6.1+3.2,2.6);
		\node at (8.2,2.35) {$\Xi$};	 }
	
	\end{tikzpicture}
	\caption{The relevant complexes for the $3D$ weak symmetry Hellinger-Reissner form.}
	\label{Fig:4}
\end{figure}

Since the differential operators act row-wise, we can think of row-wise classical de Rham sequences and thus the next lemma is clear. 

\begin{lemma}[Hilbert complexes for linear elasticity]
	\label{Lemma:Commutativty_double_complex_lin_ealsticity}
	The top and bottom sequences in Fig. \ref{Fig:3} and Fig. \ref{Fig:4} define closed Hilbert complexes with underlying $L^2$ inner products. Furthermore, the mentioned diagrams commute, i.e. we have the relations	
	\begin{equation*}
	\nabla \cdot v = \textup{Skew}(\textup{curl}(v)), \ \forall v \in H^1(\o,\mathbb{R}^2) \ \ \ \textup{and} \ \ \ \nabla \cdot\big( \Xi w\big) = \textup{Skew}(\nabla \times w), \ \forall w \in H(\o,\c,\mathbb{M}).
	\end{equation*} 
\end{lemma}

\begin{proof}
	First regarding the situation $n=2$. Then, using suitable  proxy fields,  the top complex in Fig. \ref{Fig:3} corresponds to the standard $2D$ de Rham chain and hence it is a closed Hilbert complex; see \cite[Section 3]{Falk}. The bottom complex is built up of two row-wise de Rham sequences and the closedness is clear, too. 
	On the other hand, in  Fig. \ref{Fig:4} we have in both rows  the classical vector-valued de Rham chain. This means, we have indeed closed Hilbert complexes.\\
	 It remains to show the last two commutativity relations. Since they are explained and derived in \cite[Section 4.1]{Falk}, we postpone the calculations to the appendix; see Sec. \ref{Proof1_appendix}.
\end{proof}

\begin{lemma}[Inf-sup condition]
		\label{Lemma:lin_ealsticity_inf_sup_cont}
	For $n \in \{2, \, 3 \}$ it holds
\begin{equation*}
	\underset{{q} \in L^2(\o) }{\inf} \ \underset{{w} \in H^1(\o,\mathbb{R}^n)}{\sup}   \ \ \frac{ \langle  \nabla \cdot  {w}, {q}\rangle}{\n{{w}}_{H^1} \ \n{{q}}_{L^2}}  >0,
\end{equation*}
as well as for $n=3$ it is also 
\begin{equation*}
\underset{{v} \in L^2(\o,\mathbb{R}^n) }{\inf} \ \underset{{m} \in H^1(\o,\mathbb{M})}{\sup}   \ \ \frac{ \langle  \nabla \cdot (\Xi {m}), {v}\rangle}{\n{{m}}_{\c} \ \n{{v}}_{L^2}}  >0,
\end{equation*}
We note that we have a domain $\o \subset \mathbb{R}^n$.
\end{lemma}
\begin{proof}
	The first part of the assertion is a classical result within PDE theory. And since $\Xi$ defines a bounded bijection from $H^1(\o,\mathbb{M})$ to itself, the second statement follows by the first one. 
\end{proof}

Thus in view of the Lemmata \ref{Lemma:Commutativty_double_complex_lin_ealsticity}, \ref{Lemma:lin_ealsticity_inf_sup_cont} and the two diagrams  in Figs. \ref{Fig:3}-\ref{Fig:4}, we obtain the wanted result, namely the mixed linear elasticity formulation with weakly imposed symmetry fits to the abstract setting of the previous chapter.
In particular we see the well-posedness of the continuous problem \eqref{weak_form_wek_symmetry_contiuous}.
Besides, there is a guideline for the choice of discretization spaces. We need structure-preserving de Rham chain discretizations and we have to check the inf-sup condition. In the next sections we define suitable approximation spaces exploiting the concept of isogeometric de Rham spaces. 

\subsection{Discretization}
Our main reference for this section is the paper \cite{Buffa2011IsogeometricDD}, where structure-preserving de Rham discretizations are developed and applied in the context of Maxwell's equations.\\
We begin with a brief explanation of B-splines and the corresponding isogeometric de Rham spaces in three and two dimensions. For details on IGA we refer to \cite{IGA1,IGA2,IGA3}.
\subsubsection{B-splines and isogeometric de Rham spaces}

	\label{section:splines}
Here, we state a short overview of B-spline functions, spaces respective, and some basic results in the univariate as well as in the multivariate case. \\
Following \cite{IGA1,IGA3} for a brief
exposition, we call an  increasing sequence of real numbers $\Xi \coloneqq \{ \xi_1 \leq  \xi_2  \leq \dots \leq \xi_{k+\textsf{p}+1}  \}$ for some $\textsf{p} \in \mathbb{N}$   \emph{knot vector}, where we assume  $0=\xi_1=\xi_2=\dots=\xi_{\textsf{p}+1}, \ \xi_{k+1}=\xi_{k+2}=\dots=\xi_{k+\textsf{p}+1}=1$, and call such knot vectors $\textsf{p}$-open. 
Further the multiplicity of the $j$-th knot is denoted by $m(\xi_j)$.
Then  the univariate B-spline functions $\widehat{B}_{j,\textsf{p}}(\cdot)$ of degree $\textsf{p}$ corresponding to a given knot vector $\Xi$ is defined recursively by the \emph{Cox-DeBoor formula}:
\begin{align*}
\widehat{B}_{j,0}(\zeta) \coloneqq \begin{cases}
1, \ \ \textup{if}  \ \zeta \in [\xi_{j},\xi_{j+1}) \\
0, \ \ \textup{else},
\end{cases}
\end{align*}
\textup{and if }  $\textsf{p} \in \mathbb{N}_{\geq 1} \ \textup{we set}$ 
\begin{align*}
\widehat{B}_{j,\textsf{p}}(\zeta)\coloneqq \frac{\zeta-\xi_{j}}{\xi_{j+\textsf{p}}-\xi_j} \widehat{B}_{j,\textsf{p}-1}(\zeta)  +\frac{\xi_{j+\textsf{p}+1}-\zeta}{\xi_{j+\textsf{p}+1}-\xi_{j+1}} \widehat{B}_{j+1,\textsf{p}-1}(\zeta),
\end{align*}
where one puts $0/0=0$ to obtain  well-definedness. The knot vector $\Xi$ without knot repetitions is denoted by $\{ \psi_1, \dots , \psi_N \}$. \\
The multivariate extension of the last spline definition is achieved by a tensor product construction. In other words, we set for a given  knot vector   $\boldsymbol{\Xi} \coloneqq \Xi_1 \times   \dots \times \Xi_d $, where the $\Xi_{l}=\{ \xi_1^{l}, \dots , \xi_{k_l+\textsf{p}_l+1}^{l} \}, \ l=1, \dots , d$ are $\textsf{p}_l$-open,   and a given \emph{degree vector}   $\p{\textsf{p}} \coloneqq (\textsf{p}_1, \dots , \textsf{p}_d)$ for the multivariate case
\begin{align}
\widehat{B}_{\p{i},\p{\textsf{p}}}(\boldsymbol{\zeta}) \coloneqq \prod_{l=1}^{d} \widehat{B}_{i_l,\textsf{p}_l}(\zeta_l), \ \ \ \ \forall \, \p{i} \in \mathit{\mathbf{I}}, \ \  \boldsymbol{\zeta} \coloneqq (\zeta_1, \dots , \zeta_d),
\end{align}
with  $d$ as  the underlying dimension of the parametric domain $\widehat{\Omega} \coloneqq (0,1)^d$ and $\textup{\p{I}}$ the multi-index set $\textup{\p{I}} \coloneqq \{ (i_1,\dots,i_d) \  | \  1\leq i_l \leq k_l, \ l=1,\dots,d  \}$.\\
B-splines  fulfill several properties and for our purposes the most important ones are:
\begin{itemize}
	\item If  for all internal knots the multiplicity satisfies $1 \leq m(\xi_j) \leq m \leq \textsf{p} , $ then the B-spline basis functions $\widehat{B}_{i,\textsf{p}}({\cdot})$ are globally $C^{\textsf{p}-m}$-continuous. Therefore we define  in this case the regularity integer $r \coloneqq \textsf{p}-m$. Obviously, by the product structure, we get splines $\widehat{B}_{\p{i},\p{\textsf{p}}}(\cdot)$ which are $C^{r_l}$-smooth  w.r.t. the $l$-th coordinate direction if the internal multiplicities fulfill $1 \leq m(\xi_j^l) \leq m_l  \leq \textsf{p}_l, \ \forall j, \ r_l \coloneqq \textsf{p}_l-m_l,$ in the multivariate case. 
	In case of $r_i <0$ we have discontinuous splines w.r.t. the $i$-th coordinate direction. 
	\item The B-splines $ \{\widehat{B}_{\p{i},\p{\textsf{p}}} \ | \ \ \p{i} \in \p{I} \}$ are linearly independent.
	\item For univariate splines $\widehat{B}_{i,\textsf{p}}, \ \textsf{p} \geq 1, r\geq 0$ we have
	\begin{align}
	\label{eq:soline_der}
	\hat{\partial}_{\zeta} \widehat{B}_{i,\textsf{p}}(\zeta) = \frac{\textsf{p}}{\xi_{i+\textsf{p}}-\xi_i}\widehat{B}_{i,\textsf{p}-1}(\zeta) -  \frac{\textsf{p}}{\xi_{i+\textsf{p}+1}-\xi_{i+1}}\widehat{B}_{i+1,\textsf{p}-1}(\zeta),
	\end{align} 
	with $\widehat{B}_{1,\textsf{p}-1}(\zeta)\coloneqq \widehat{B}_{k+1,\textsf{p}-1}(\zeta) \coloneqq 0$.
	\item  The support of the spline $\widehat{B}_{i,\textsf{p}} $ is in the interval $[\xi_i,\xi_{i+\textsf{p}+1}]$. Moreover,  the knots $\psi_j$ define a subdivision of the interval $(0,1)$ and for each element $I = (\psi_j,\psi_{j+1})$ we find a $i$ with $(\psi_j,\psi_{j+1})= (\xi_i,\xi_{i+1})$.
\end{itemize}
The space spanned by all univariate splines $\widehat{B}_{i,\textsf{p}}$ corresponding to a given knot vector with degree $\textsf{p}$ and global regularity $r$  is denoted by $$S_\textsf{p}^r \coloneqq \textup{span}\{ \widehat{B}_{i,\textsf{p}} \ | \ i = 1,\dots , k \}.$$
For the multivariate case we just define the spline space as the product space $$S_{\textsf{p}_1, \dots , \textsf{p}_d}^{r_1,\dots,r_d} \coloneqq S_{\textsf{p}_1}^{r_1} \otimes \dots \otimes S_{\textsf{p}_d}^{r_d} = \textup{span} \{\widehat{B}_{\p{i},\p{\textsf{p}}} \ | \  \p{i} \in \mathit{\mathbf{I}}  \}$$ of proper univariate spline spaces.\\
To define discrete spaces based on splines we require a parametrization mapping $$\p{F} \colon \widehat{\Omega} \coloneqq (0,1)^n \rightarrow \o \subset \mathbb{R}^n$$ which parametrizes the computational domain.
The knots stored in the knot vector $  \boldsymbol{\Xi} $, corresponding to  the underlying  splines, determine a mesh in the parametric domain $\widehat{\Omega} $, namely  $\widehat{\mathcal{T}} \coloneqq \{ K_{\p{j}}\coloneqq (\psi_{j_1}^1,\psi_{j_1+1}^1 ) \times \dots \times (\psi_{j_{n}}^{n},\psi_{j_{n}+1}^{n} ) \ | \  \p{j}=(j_1,\dots,j_{n}), \ \textup{with} \ 1 \leq j_i <N_i\},$ and
with ${\boldsymbol{\Psi}}= \{\psi_1^1, \dots, \psi _{N_1}^1\}  \times \dots \times \{\psi_1^{n}, \dots, \psi _{N_{n}}^{n}\}$  \  \textup{as  the knot vector} \ ${\boldsymbol{\Xi}}$ \ 
\textup{without knot repetitions}.   
The image of this mesh under the mapping $\p{F}$, i.e. $\mathcal{T} \coloneqq \{{\p{F}}(K) \ | \ K \in \widehat{\mathcal{T}} \}$, gives us a mesh structure in the physical domain. By inserting knots without changing the parametrization  we can refine the mesh, which is the concept of $h$-refinement \cite{IGA2,IGA1}.
For a mesh $\mathcal{T}$ we define the global mesh size $h \coloneqq \max\{h_{{K}} \ | \ {K} \in \mathcal{T} \}$, where for ${K} \in \mathcal{T}$ we denote with $h_{{K}} \coloneqq \textup{diam}({K}) $ the \emph{element size}.

For our IGA spaces we have the next underlying basic assumptions.
\begin{assumption}{(Regular mesh)}\\
		\label{Assumption:regular_triangul}
	We denote with $\mathcal{T}_h$ the mesh with global mesh size $h$ and assume shape-regularity and quasi-uniformity. In particular, there exists a constant $c_u $ independent from the mesh size such that $h_{{K}} \leq h \leq c_u \, h_{{K}}$ for all mesh elements ${K} \in \mathcal{T}_h$. 
	Let the parametrization mapping $\p{F}$ be diffeomorphic with $\p{F} \in C^2(\overline{\widehat{\o}}), \ \p{F}^{-1} \in C^2(\overline{\o})$ and let the restrictions of $\p{F}$ to mesh elements be smooth. Hence the Jacobian $\p{J} \coloneqq D\p{F}$ is a uniformly continuous mapping.
\end{assumption}
\begin{remark}[Isogeometric parametrization]
Following the idea of isogeometric analysis, we assume $\p{F}$ to be defined utilizing B-splines or NURBS, where \emph{NURBS} (Non-uniform rational B-splines) are a generalization of the mentioned B-splines by incorporating a B-spline weight function.  In other words we use the B-splines for both, the discrete function space definition  and the geometry description.
\end{remark}

Basis for the stable discretization of the mixed linear elasticity formulation with weak symmetry is the connection between the de Rham complex 
 as described in \cite{Arnold2007MixedFE} and in the previous Sec \ref{Sec:conncetion_deRham_elasticity}.
Therefore we will exploit later structure-preserving discretizations of the de Rham chain utilizing proper spline spaces. A detailed discussion on underlying  isogeometric discrete differential forms can be found in \cite{Buffa2011IsogeometricDD} and below we state very shortly some  results from the latter reference. They point out to be useful for approximation methods.

Firstly, for  $\Omega \subset \mathbb{R}^3$ and $\textsf{p}_l > r_l \geq 0$,  the spaces 
\begin{alignat*}{3}
V_{h,3}^0 &\coloneqq  \mathcal{Y}_0^{-1}(\widehat{V}^0_{h,3}) \ \ \  \textup{with}  	\hspace{0.3cm}&&\widehat{V}^0_{h,3} \coloneqq  S_{\textsf{p}_1,\textsf{p}_2,\textsf{p}_3}^{r_1,r_2,r_3}  \ \ \ \textup{and}  \ \ \ \mathcal{Y}_0(q)\coloneqq  q \circ \p{F},\nonumber \\
{V}_{h,3}^1 &\coloneqq \mathcal{Y}_1^{-1}(\widehat{{V}}^1_{h,3})\ \ \  \textup{with}  	\hspace{0.3cm} &&\widehat{{V}}^1_{h,3} \coloneqq  \big(S_{\textsf{p}_1-1,\textsf{p}_2,\textsf{p}_3}^{r_1-1,r_2,r_3} \times S_{\textsf{p}_1,\textsf{p}_2-1,\textsf{p}_3}^{r_1,r_2-1,r_3} \times S_{\textsf{p}_1,\textsf{p}_2,\textsf{p}_3-1}^{r_1,r_2,r_3-1} \big)^T   \\
& &&  \ \ \ \textup{and}  \ \ \ \mathcal{Y}_1({v})\coloneqq  \p{J}^T \big({v} \circ \p{F}) ,\nonumber\\
{V}_{h,3}^2 &\coloneqq \mathcal{Y}_2^{-1}(\widehat{{V}}^2_{h,3})\ \ \  \textup{with}  	\hspace{0.3cm} &&\widehat{{V}}^2_{h,3} \coloneqq  \big(S_{\textsf{p}_1,\textsf{p}_2-1,\textsf{p}_3-1}^{r_1,r_2-1,r_3-1} \times S_{\textsf{p}_1-1,\textsf{p}_2,\textsf{p}_3-1}^{r_1-1,r_2,r_3-1} \times S_{\textsf{p}_1-1,\textsf{p}_2-1,\textsf{p}_3}^{r_1-1,r_2-1,r_3} \big)^T   \\
& &&  \ \ \ \textup{and}  \ \ \ \mathcal{Y}_2({v})\coloneqq  \textup{det}(\p{J}) \p{J}^{-1} \big({v} \circ \p{F} \big) ,\nonumber \\
V_{h,3}^3 &\coloneqq  \mathcal{Y}_3^{-1}(\widehat{V}^3_{h,3}) \ \ \  \textup{with}  	\hspace{0.3cm}&&\widehat{V}^3_{h,3} \coloneqq  S_{\textsf{p}_1-1,\textsf{p}_2-1,\textsf{p}_3-1}^{r_1-1,r_2-1,r_3-1}  \ \ \ \textup{and}  \ \ \ \mathcal{Y}_3(q)\coloneqq  \textup{det}(\p{J}) \big(q \circ \p{F} \big),\nonumber 
\end{alignat*}
form an exact subcomplex of the de Rham chain, i.e. there are  continuous projections 
\begin{align*}
{\Pi}^i_{h,3} \colon L^2(\Omega) \rightarrow {V}_{h,3}^i, \ i=0,3\ \ \textup{and} \ {\Pi}^j_{h,3} \colon {L}^2(\Omega,\mathbb{R}^3) \rightarrow {{V}}_{h,3}^j, \ j=1,2, 
\end{align*} such that we obtain the commuting diagram in Fig. \ref{Fig:Neu2}.
\begin{figure}[h!]
	\centering
	\begin{tikzpicture}
	\node at (0,1.4) {$H^1(\Omega)$};
	
	\node at (3,1.4) {${H}(\Omega,\textup{curl})$};

	\node at (6.5,1.4) {${H}(\Omega,\textup{div})$};

	\node at (9.5,1.4) {${L}^2(\Omega)$};

	\node at (1.3,1.7) {$\nabla$};

	\node at (4.7,1.7) {$\nabla \times $};

	\node at (8.2,1.7) {$\nabla \cdot $};

	\draw[->] (0.8,1.4) to (1.8,1.4);

	\draw[->] (4.2,1.4) to (5.3,1.4);
	
	\draw[->] (7.7,1.4) to (8.75,1.4);
	\draw[->] (10.2,1.4) to (10.9,1.4);
	\draw[->] (-1.6,1.4) to (-0.72,1.4);
	\node at (10.55,1.7) {$0$};
	\node at (-1.09,1.65) {$\iota$};
	\node at (11.45,1.4) {$\{0\}$};
	\node at (-2.25,1.4) {$\mathbb{R}$};

	\draw[->] (3,1.1) -- (3,0.2);
	\draw[->] (-2.25,1.1) -- (-2.25,0.2);
	\draw[->] (11.45,1.1) -- (11.45,0.2);
	\draw[->] (0,1.1) -- (0,0.2);
	\draw[->] (6.5,1.1) -- (6.5,0.2);
	\draw[->] (9.5,1.1) -- (9.5,0.2);

	\node at (0,-0.1) {${V}_{h,3}^{0}$};
	
	\node at (3,-0.1) {${V}_{h,3}^{1}$};

	\node at (6.5,-0.1) {${V}_{h,3}^{2}$};

	\node at (9.5,-0.1) {${V}_{h,3}^{3}$};

	\node at (1.3,0.2) {${\nabla}$};

	\node at (4.7,0.2) {${\nabla} \times $};

	\node at (8.2,0.2) {${\nabla} \cdot $};

	\draw[->] (0.8,-0.1) to (1.8,-0.1);

	\draw[->] (4.2,-0.1) to (5.3,-0.1);
	
	\draw[->] (7.7,-0.1) to (8.75,-0.1);
	\draw[->] (10.2,-0.1) to (10.9,-0.1);
	\draw[->] (-1.6,-0.1) to (-0.72,-0.1);
	\node at (10.55,0.2) {$0$};
	\node at (-1.09,0.15) {$\iota$};
	\node at (11.57,-0.1) {$\{0\} \ .$};
	\node at (-2.25,-0.1) {$\mathbb{R}$};
	
	\node[left] at (-2.2,0.7) { \small $\textup{id}$};	
	\node[left] at (0.05,0.7) { \small $\Pi_{h,3}^{0}$};
	\node[left] at (3.05,0.7) { \small $\Pi_{h,3}^{1}$};
	\node[left] at (6.55,0.7) { \small $\Pi_{h,3}^{2}$};
	\node[left] at (9.55,0.7) { \small $\Pi_{h,3}^{3}$};
	\node[left] at (11.45,0.7) { \small $\textup{id}$};
	\end{tikzpicture}
		\caption{Discretization of the $3D$ de Rham chain.}
		\label{Fig:Neu2}
\end{figure} 

Further due to the smoothness properties of the splines we get indeed a subcomplex meaning the discrete spaces are subspaces of the continuous pendants. And the pull-backs $\mathcal{Y}_i $ are compatible with the chain, too. In other words the diagram in Fig. \ref{Fig:Neu3}
\begin{figure}[h!]
	\centering
	\begin{tikzpicture}
	\node at (0,1.4) {$H^1(\Omega)$};
	
	\node at (3,1.4) {${H}(\Omega,\textup{curl})$};

	\node at (6.5,1.4) {${H}(\Omega,\textup{div})$};

	\node at (9.5,1.4) {${L}^2(\Omega)$};

	\node at (1.3,1.7) {$\nabla$};

	\node at (4.7,1.7) {$\nabla \times $};

	\node at (8.2,1.7) {$\nabla \cdot $};

	\draw[->] (0.8,1.4) to (1.8,1.4);

	\draw[->] (4.2,1.4) to (5.3,1.4);
	
	\draw[->] (7.7,1.4) to (8.75,1.4);
	\draw[->] (10.2,1.4) to (10.9,1.4);
	\draw[->] (-1.6,1.4) to (-0.72,1.4);
	\node at (10.55,1.7) {$0$};
	\node at (-1.09,1.65) {$\iota$};
	\node at (11.45,1.4) {$\{0\}$};
	\node at (-2.25,1.4) {$\mathbb{R}$};

	\draw[->] (3,1.1) -- (3,0.2);
	\draw[->] (-2.25,1.1) -- (-2.25,0.2);
	\draw[->] (11.45,1.1) -- (11.45,0.2);
	\draw[->] (0,1.1) -- (0,0.2);
	\draw[->] (6.5,1.1) -- (6.5,0.2);
	\draw[->] (9.5,1.1) -- (9.5,0.2);

	\node at (0,-0.1) {$H^1(\widehat{\Omega})$};
	
	\node at (3,-0.1) {${H}(\widehat{\Omega},\textup{curl})$};

	\node at (6.5,-0.1) {${H}(\widehat{\Omega},\textup{div})$};

	\node at (9.5,-0.1) {$L^2(\widehat{\Omega})$};

	\node at (1.3,0.2) {$\widehat{\nabla}$};

	\node at (4.7,0.2) {$\widehat{\nabla} \times $};

	\node at (8.2,0.2) {$\widehat{\nabla} \cdot $};

	\draw[->] (0.8,-0.1) to (1.8,-0.1);

	\draw[->] (4.2,-0.1) to (5.3,-0.1);
	
	\draw[->] (7.7,-0.1) to (8.75,-0.1);
	\draw[->] (10.2,-0.1) to (10.9,-0.1);
	\draw[->] (-1.6,-0.1) to (-0.72,-0.1);
	\node at (10.55,0.2) {$0$};
	\node at (-1.09,0.15) {$\iota$};
	\node at (11.57,-0.1) {$\{0\} \ .$};
	\node at (-2.25,-0.1) {$\mathbb{R}$};
	
	\node[left] at (-2.2,0.7) { \small $\textup{id}$};	
	\node[left] at (0.05,0.7) { \small $\mathcal{Y}_0$};
	\node[left] at (3.05,0.7) { \small $\mathcal{Y}_1$};
	\node[left] at (6.55,0.7) { \small $\mathcal{Y}_2$};
	\node[left] at (9.55,0.7) { \small $\mathcal{Y}_3$};
	\node[left] at (11.45,0.7) { \small $\textup{id}$};
	\end{tikzpicture}
		\caption{One relates the original complex with the complex in the reference domain.}
		\label{Fig:Neu3}
\end{figure} 
commutes. The  hat notation $\widehat{\nabla}$ shows us that we take the derivatives w.r.t. the parametric coordinates.\\
For reasons of simplification we assume in the following that $ \textsf{p}\coloneqq \textsf{p}_1=\textsf{p}_2=\textsf{p}_3$ and $r \coloneqq r_1 =r_2=r_3$. Moreover, as shown in  \cite{Buffa2011IsogeometricDD}, we can define the mentioned $\Pi_{h,3}^i$ projections in such a way that the next approximation estimates are valid.

\begin{lemma}[Spline approximation results in $3D$]
	\label{lemma:spline_approx_estimates_3D}
	Let $0 \leq r < \textsf{p}$. Assuming sufficient regularity for $\phi$ and ${v}$, it is
	\begin{alignat*}{3}
	& \norm{{\phi}-\Pi^0_{h,3}{\phi}}_{H^{l}} \leq C \ h^{s-l} \ \norm{{\phi}}_{H^{s}},  \ &&0  \leq l \leq s \leq \textsf{p}+1, \\
	& \norm{{{v}}-\Pi^1_{h,3}{{v}}}_{\c} \leq C \ h^{s},  \ \norm{{{v}}}_{{H}^s(\textup{curl})}, \ \ \  &&0   \leq s \leq \textsf{p},  \hspace{0.5cm}\\
	& \norm{{{v}}-\Pi^2_{h,3}{{v}}}_{\d} \leq C \ h^{s} \ \norm{{{v}}}_{{H}^s(\textup{div})},  \ &&0   \leq s \leq \textsf{p},\\
	& \norm{{\phi}-\Pi^3_{h,3}{\phi}}_{L^2} \leq C \ h^{s} \ \norm{{\phi}}_{{H}^s},  \ &&0  \leq s \leq \textsf{p}.
	\end{alignat*}
	We use the notation  $\norm{{{v}}}_{{H}^s(\textup{curl})}^2 \coloneqq \norm{ {{v}}}_{{H}^s}^2 + \norm{ \nabla \times {{v}}}_{{H}^s}^2$ and $\norm{{{v}}}_{{H}^s(\textup{div})}^2 \coloneqq \norm{ {{v}}}_{{H}^s}^2 + \norm{ \nabla \cdot {{v}}}_{{H}^s}^2$.	
\end{lemma}
\begin{proof}
	See \cite[Theorem 5.3 and Remark 5.1]{Buffa2011IsogeometricDD}.
\end{proof}

Here and in the rest of the article $C<\infty$ denotes a constant only depending on the domain $\Omega, \ \p{F}$ and the polynomial degrees as well as on the regularity parameter $r$ and may vary at different occurrences.

If we consider the case of $\Omega \subset \mathbb{R}^2$ one can introduce different spline spaces which also yield an exact subcomplex of the  $2D$ de Rham complex. To be more precise the approach in \cite{Buffa2011IsogeometricDD} is based on the product structure of the B-splines and it is used for definitions of  the projections  and the discrete function spaces. That is way we can easily adapt the procedure and the steps to the planar case. In other words, 
we obtain  another commuting diagram, 

	\begin{center}
		\begin{tikzpicture}
		\node at (0,1.4) {$H^1(\Omega)$};			
		\node at (3,1.4) {${H}(\Omega,\textup{div})$};
		\node at (6.5,1.4) {$L^2(\Omega)$};			
		\node at (9.5,1.4) {$\{0 \}$};						
		\node at (1.3,1.7) {$\textup{curl}$};			
		\node at (4.7,1.7) {$\nabla \cdot $};			
		\node at (8.2,1.7) {$0 $};			
		\draw[->] (0.8,1.4) to (1.8,1.4);			
		\draw[->] (4.2,1.4) to (5.3,1.4);			
		\draw[->] (7.7,1.4) to (8.75,1.4);			
		\draw[->] (-1.6,1.4) to (-0.72,1.4);			
		\node at (-1.09,1.65) {$\iota$};			
		\node at (-2.25,1.4) {$\mathbb{R}$};			
		\draw[->] (3,1.1) -- (3,0.2);
		\draw[->] (-2.25,1.1) -- (-2.25,0.2);			
		\draw[->] (0,1.1) -- (0,0.2);
		\draw[->] (6.5,1.1) -- (6.5,0.2);
		\draw[->] (9.5,1.1) -- (9.5,0.2);			
		\node at (0,-0.1) {${V}_{h,2}^{0}$};			
		\node at (3,-0.1) {${V}_{h,2}^{1}$};		
		\node at (6.5,-0.1) {${V}_{h,2}^{2}$};			
		\node at (9.62,-0.1) {$\{0 \} \ ,$};			
		\node at (1.3,0.2) {${\textup{curl}}$};			
		\node at (4.7,0.2) {${\nabla} \cdot $};			
		\node at (8.2,0.2) {$0 $};		
		\draw[->] (0.8,-0.1) to (1.8,-0.1);		
		\draw[->] (4.2,-0.1) to (5.3,-0.1);
		\draw[->] (7.7,-0.1) to (8.75,-0.1);		
		\draw[->] (-1.6,-0.1) to (-0.72,-0.1);			
		\node at (-1.09,0.15) {$\iota$};			
		\node at (-2.25,-0.1) {$\mathbb{R}$};			
		\node[left] at (-2.2,0.7) { \small $\textup{id}$};	
		\node[left] at (0.05,0.7) { \small $\Pi_{h,2}^{0}$};
		\node[left] at (3.05,0.7) { \small $\Pi_{h,2}^{1}$};
		\node[left] at (6.55,0.7) { \small $\Pi_{h,2}^{2}$};
		\node[left] at (9.55,0.7) { \small $\textup{id}$};
		\end{tikzpicture}
	\end{center}
where  \begin{alignat*}{3}
V_{h,2}^0 &\coloneqq  \mathcal{Y}_0^{-1}(\widehat{V}^0_{h,2}) \ \ \  \textup{with}  	\hspace{0.3cm}&&\widehat{V}^0_{h,2} \coloneqq  S_{\textsf{p},\textsf{p}}^{r,r},  \\
{V}_{h,2}^1 &\coloneqq \mathcal{Y}_2^{-1}(\widehat{{V}}^1_{h,2})\ \ \  \textup{with}  	\hspace{0.3cm} &&\widehat{{V}}^1_{h,2} \coloneqq  \big(S_{\textsf{p},\textsf{p}-1}^{r,r-1} \times S_{\textsf{p}-1,\textsf{p}}^{r-1,r}  \big)^T  
,\nonumber\\
V_{h,2}^2 &\coloneqq  \mathcal{Y}_3^{-1}(\widehat{V}^2_{h,2}) \ \ \  \textup{with}  	\hspace{0.3cm}&&\widehat{V}^2_{h,2} \coloneqq  S_{\textsf{p}-1,\textsf{p}-1}^{r-1,r-1},  \ \ \ 
\end{alignat*}
and with $L^2$-continuous projections $\Pi_{h,2}^i$ onto the respective spaces. One notes the same notation for the pull-backs in  $2D$ as for $3D$ and we mean the straightforward restriction to the two-dimensional setting. An important aspect showing the appropriateness of  the $\mathcal{Y}_i$ are the properties \begin{align}
\label{eq_trans_compa}
\widehat{\nabla} \cdot \mathcal{Y}_2(\f{v}) = \mathcal{Y}_3(\nabla \cdot \f{v}) , \ \ \ \ \widehat{\textup{curl}}(\mathcal{Y}_0(\phi)) = \mathcal{Y}_2(\textup{curl}(\phi)), \ \ \textup{in} \ 2D,
\end{align}
which can be proven easily.
Completely analogously to the $3D$ case one can derive approximation estimates for the latter spaces. 
\begin{lemma}[Spline approximation results in $2D$]
	\label{lemma_approx_spline_spaces_2D}
	Let $0 \leq r < \textsf{p}$. It holds 
	\begin{alignat*}{3}
	& \norm{{\phi}-\Pi^0_{h,2}{\phi}}_{H^{l}} \leq C \ h^{s-l} \ \norm{{\phi}}_{H^{s}},  &&0  \leq l \leq s \leq \textsf{p}+1,\\
	& \norm{{{v}}-\Pi^1_{h,2}{{v}}}_{\d} \leq C \ h^{s} \ \norm{{v}}_{{H}^s(\textup{div})},  \ \ \  &&0   \leq s \leq \textsf{p},\\
	& \norm{\phi-\Pi^2_{h,2}{\phi}}_{L^2} \leq C \ h^{s} \ \norm{\phi}_{{H}^s},  &&0   \leq s \leq \textsf{p},
	\end{alignat*}	
	for sufficiently smooth $\phi,v$.
\end{lemma}
\begin{proof}
	With the analogous proof steps like for the $3D$ de Rham chain, this follows from \cite{Buffa2011IsogeometricDD}; compare \cite[Theorem 5.3 and Remark 5.1]{Buffa2011IsogeometricDD}. 
\end{proof}

\begin{remark}
	The spaces $V_{h,n}^i$ from above depend on the regularity and on the polynomial degree. This is the reason why we will write later sometimes $V_{h,n}^i(\textsf{p},r)$ to emphasize the underlying regularity $r$ and degree $\textsf{p}$.
\end{remark}
After the introduction of elementary notions and spline spaces we can  choose now proper discrete spaces for our linear elasticity formulation with weak symmetry. The choice of these spaces and  the well-posedness proof for the arising saddle-point problem is part of the next section.

\subsubsection{Choice of discrete spaces}
\label{Sec:Chioce of spces}
As already mentioned, the isogeometric de Rham spaces $V_{h,2}^j$ and $V_{h,3}^j$ from the last section lead to structure-preserving discretizations of the de Rham sequences in $2D$ and $3D$. Hence, due to Lemma \ref{Lemma:well-posedness_discrete_abstract} it is natural to choose for the discretization of $(\sigma,u,p)$ in \eqref{weak_form_wek_symmetry_contiuous} the spaces $U_h, \ \Sigma_h, \ P_h$ defined below:
	\begin{definition}[Choice of spaces]
	\label{Defnition of discrete spaces}
	For $\textsf{p} \geq 2, \ \textsf{p}-1 > r \geq 0 $ and using the abbreviations $$W_{h}^1 \coloneqq V_{h,3}^1(\textsf{p}+1,r), \ 	W_{h}^2 \coloneqq V_{h,3}^2(\textsf{p}+1,r), \ W_h^3 \coloneqq V_{h,3}^3(\textsf{p}+1,r),  \ \bar{W}_h^3 \coloneqq V_{h,3}^0(\textsf{p}-1,r)$$ we set  in the three-dimensional setting 
	\begin{alignat*}{4}	
	&\f{\Sigma}_h \coloneqq (\f{W}_{h}^2 \times \f{W}_{h}^2 \times \f{W}_{h}^2)^T,  \ \
	&&\f{U}_h \coloneqq ({W}_{h}^3 \times {W}_{h}^3 \times {W}_{h}^3)^T, \ \ 
	&&\f{P}_h \coloneqq (\bar{W}_h^3 \times \bar{W}_h^3 \times \bar{W}_h^3)^T.	
	\end{alignat*}
	If we write e.g. $(\f{W}_h^2 \times \f{W}_h^2 \times \f{W}_h^2)^T$ we mean that the rows of the matrix fields are elements of $\f{W}_h^2$. Further, we introduce the auxiliary spaces 
	\begin{align}
	\label{eq:add1}
	R_h \coloneqq ({W}_{h}^1 \times {W}_{h}^1 \times {W}_{h}^1)^T, \ R_{h,0} \coloneqq R_h \cap H_0^1(\o,\mathbb{M}), \ P_{h,0}  \coloneqq P_h \cap L_0^2(\o,\mathbb{R}^3).
	\end{align}
	For two dimensions the definition is easier, namely for $\textsf{p}> r+1\geq 1$ we choose
	\begin{align*}
	\Sigma_h \coloneqq (V_{h,2}^{1}(\textsf{p},r) \times V_{h,2}^1(\textsf{p},r))^T, \ U_h \coloneqq (V_{h,2}^{2}(\textsf{p},r) \times V_{h,2}^2(\textsf{p},r))^T, \ P_h \coloneqq V_{h,2}^0(\textsf{p}-1,r),
	\end{align*} 
	and  define the auxiliary space $R_h \coloneqq  \big(V_{h,2}^0(\textsf{p},r) \times V_{h,2}^0(\textsf{p},r) \big)^T$.
\end{definition}

Since we used spaces from the isogeometric de Rham sequences to define  $\Sigma_h, U_h, P_h$ one can introduce $L^2$-bounded cochain projections onto the mentioned spaces by means of the (row-wise) application of the projections from Sec. \ref{section:splines}. We denote them with \begin{align*}
\Pi_h^{\Sigma} \colon L^2(\o,\mathbb{M}) \rightarrow \Sigma_h, \ \ \Pi_h^U \colon L^2(\o,\mathbb{R}^n) \rightarrow U_h, \ \ \Pi_h^P \colon L^2(\o,\mathbb{R}^{s(n)}) \rightarrow P_h, \\ \Pi_h^R \colon L^2(\o,\mathbb{M}) \rightarrow R_h \ \textup{for} \ n=3, \ \ \Pi_h^R \colon L^2(\o,\mathbb{R}^2) \rightarrow R_h \ \textup{for}  \ n=2,
\end{align*}
and they lead to the commuting diagrams in Fig. \ref{Fig:5} ($2D$) and Fig. \ref{Fig:6} ($3D$).

	\begin{figure}[h!]
	\centering
	\begin{tikzpicture}

	\node at (10.1,4.5) {$P_h$};

     \draw[->] (10.1,3.4) to (10.1,3.2+1);
     \node at (5.4+4.4,1.65-1.6+3.7) {$\Pi_h^P $};

	\node at (2.15,1.4) {$\mathbb{R}^2 $};
	\node at (4,1.4) {$H^1(\o,\mathbb{R}^2) $};
	\node at (7.1,1.4) {${H}(\o,\d, \mathbb{M})$};
	\node at (10.1,1.4) {${L}^2(\o, \mathbb{R}^2)$};		
	\node at (12,1.4) {$\{0\} $};
	
	\node at (2.7,1.6) {$\iota $};
	\node at (5.4,1.65) {$\textup{curl} $};
	\node at (8.7,1.65) {$\nabla \cdot  $};		
	\node at (11.2,1.6) {$0$};
	
	\draw[->] (2.4,1.4) to (3,1.4);
	\draw[->] (4.55+0.4,1.4) to (5.5+0.4,1.4);
	\draw[->] (8.3,1.4) to (9.25,1.4);
	\draw[->] (11,1.4) to (11.6,1.4);

	\node at (2.2,3) {$\mathbb{R} $};
	\node at (4,3) {$H^1(\o,\mathbb{R}) $};
	\node at (7.1,3) {${H}(\o,\d, \mathbb{R}^2)$};
	\node at (10.1,3) {${L}^2(\o, \mathbb{R})$};		
	\node at (12,3) {$\{0\} $};
	
	\node at (2.7,3.2) {$\iota $};
	\node at (5.4,3.25) {$\textup{curl} $};
	\node at (8.7,3.25) {$\nabla \cdot  $};		
	\node at (11.2,3.2) {$0$};
	
	\draw[->] (2.4,3) to (3,3);
	\draw[->] (4.55+0.4,3) to (5.5+0.4,3);
	\draw[->] (8.3,3) to (9.25,3);
	\draw[->] (11,3) to (11.6,3);

	\node at (4,1.4-1.6) {$R_h$};
	\node at (7.1,1.4-1.6) {$\Sigma_h$};
	\node at (10.1,1.4-1.6) {$U_h$};

	\node at (5.4,1.65-1.6) {$\textup{curl} $};
	\node at (8.7,1.65-1.6) {$\nabla \cdot  $};

	\draw[->] (4.55+0.4,1.4-1.6) to (5.5+0.4,1.4-1.6);
	\draw[->] (8.3,1.4-1.6) to (9.25,1.4-1.6);

	\draw[<-] (10.1,3.4-3.2) to (10.1,3.2+1-3.2);	
	\draw[<-] (7.1,3.4-3.2) to (7.1,3.2+1-3.2);
	\draw[<-] (4,3.4-3.2) to (4,3.2+1-3.2);
	
	\node at (5.4+4.4,1.65-1.6+0.6) {$\Pi_h^U $};
	\node at (5.4+1.4,1.65-1.6+0.6) {$\Pi_h^{\Sigma} $};
	\node at (5.4-1.7,1.65-1.6+0.6) {$\Pi_h^R $};

	{\draw[->] (6.2+1.5,1.7) to (7.6+1.5,2.6);		
		\node at (8,2.35) {$\textup{Skew}$};}
	
	{	\draw[->] (3.96+0.5,1.7) to (6.1,2.6);
		\node at (5.05,2.35) {$\textup{I}$};	 }
	
	\end{tikzpicture}
	\caption{The relevant discrete spaces in $2D$.}
	\label{Fig:5}
\end{figure}

	\begin{figure}[h!]
	\centering
	\begin{tikzpicture}

	\node at (13.5,4.5) {$P_h$};

	\draw[->] (13.5,3.4) to (13.5,3.2+1);
	\node at (5.4+7.8,1.65-1.6+3.7) {$\Pi_h^P $};

	\node at (2.15,1.4) {$\mathbb{R}^3 $};
	\node at (4,1.4) {$H^1(\o,\mathbb{R}^3) $};
	\node at (7.1,1.4) {${H}(\o,\c, \mathbb{M})$};
	\node at (10.45,1.4) {${H}(\o,\d, \mathbb{R}^3)$};	
	\node at (13.5,1.4) {$L^2(\o,\mathbb{R}^3) $};	
	\node at (15.3,1.4) {$\{0\} $};
	
	\node at (2.7,1.6) {$\iota $};
	\node at (5.4,1.65) {$ \nabla $};
	\node at (8.7,1.65) {$\nabla \times  $};	
	\node at (12,1.65) {$\nabla \cdot $};	
	\node at (14.6,1.6) {$0$};
	
	\draw[->] (2.4,1.4) to (3,1.4);
	\draw[->] (4.55+0.4,1.4) to (5.5+0.4,1.4);
	\draw[->] (8.3,1.4) to (9.25,1.4);
	\draw[->] (11.65,1.4) to (12.6,1.4);
	\draw[->] (14.4,1.4) to (14.9,1.4);

	\node at (2.15,3) {$\mathbb{R}^3 $};
	\node at (4,3) {$H^1(\o,\mathbb{R}^3) $};
	\node at (7.1,3) {${H}(\o,\c, \mathbb{M})$};
	\node at (10.45,3) {${H}(\o,\d, \mathbb{R}^3)$};	
	\node at (13.5,3) {$L^2(\o,\mathbb{R}^3) $};	
	\node at (15.3,3) {$\{0\} $};
	
	\node at (2.7,3.2) {$\iota $};
	\node at (5.4,3.25) {$ \nabla $};
	\node at (8.7,3.25) {$\nabla \times  $};	
	\node at (12,3.25) {$\nabla \cdot $};	
	\node at (14.6,3.2) {$0$};
	
	\draw[->] (2.4,3) to (3,3);
	\draw[->] (4.55+0.4,3) to (5.5+0.4,3);
	\draw[->] (8.3,3) to (9.25,3);
	\draw[->] (11.65,3) to (12.6,3);
	\draw[->] (14.4,3) to (14.9,3);

	\node at (7.1,1.4-1.6) {$R_h$};
	\node at (10.45,1.4-1.6) {$\Sigma_h$};
	\node at (13.5,1.4-1.6) {$U_h$};	
	
	\draw[->] (8.3,3-3.2) to (9.25,3-3.2);
	\draw[->] (11.65,3-3.2) to (12.6,3-3.2);
	\node at (8.7,3.25-3.2) {$\nabla \times  $};	
	\node at (12,3.2-3.2) {$\nabla \cdot $};

	\draw[<-] (10.45,3.4-3.2) to (10.45,3.2+1-3.2);	
	\draw[<-] (13.5,3.4-3.2) to (13.5,3.2+1-3.2);
	\draw[<-] (7.1,3.4-3.2) to (7.1,3.2+1-3.2);
	
	\node at (5.4+7.8,1.65-1.6+0.6) {$\Pi_h^U $};
	\node at (5.4+1.4+3.35,1.65-1.6+0.6) {$\Pi_h^{\Sigma} $};
	\node at (5.4-1.7+3.1,1.65-1.6+0.6) {$\Pi_h^R $};

	{\draw[->] (6.2+5,1.7) to (7.6+5,2.6);		
		\node at (11.5,2.35) {$\textup{Skew}$};}
	
	{	\draw[->] (3.96+3.7,1.7) to (6.1+3.2,2.6);
		\node at (8.2,2.35) {$\Xi$};	 }
	
	\end{tikzpicture}
	\caption{The relevant spaces in $3D$.}
	\label{Fig:6}
\end{figure}

The approximation properties of latter projections  are easily derived with Lemma \ref{lemma_approx_spline_spaces_2D} and Lemma \ref{lemma:spline_approx_estimates_3D} and they are summarized in the next corollary.

\begin{corollary}
	\label{Corollary:Projections}
	Supposing smooth enough mappings $(\sigma,u_,q) \in L^2(\o,\mathbb{M}) \times L^2(\o,\mathbb{R}^n) \times L^2(\o,\mathbb{R}^{s(n)})$  and $w \in L^2(\o,\mathbb{M})$ ($n=3$), \ $w \in L^2(\o,\mathbb{R}^2)$ ($n=2$), \ respectively, then it holds  for $0 \leq r <\textsf{p}-1, \ 0 \leq s \leq \textsf{p}$ and proper constants $C_R, C_{\Sigma}, \ C_{U}, \ C_{P}$ in $2D$:
	\begin{align*}
		&\n{w-\Pi_h^{R}w}_{L^2} \leq C_{R} h^{s} \n{w}_{H^{s}}, \ \  \n{\sigma-\Pi_h^{\Sigma}\sigma}_{\d} \leq C_{\Sigma} h^{s} \n{\sigma}_{H^{s}(\d)}, \\ 
		&\n{u-\Pi_h^{U}u}_{L^2} \leq C_{U} h^{s} \n{u}_{H^{s}}, \ \ \n{q-\Pi_h^{P}q}_{L^2} \leq C_{P} h^{s} \n{q}_{H^{s}}.
	\end{align*}
	In $3D$ ($n=3$) we obtain analogous estimates, but in the inequalities  regarding $\Pi_h^R, \ \Pi_h^{\Sigma}, \ \Pi_h^U$ we allow for $0 \leq s \leq \textsf{p}+1$.
\end{corollary}
For the proof of the inf-sup stability in $3D$ we also need the space $R_{h,0}$ from \eqref{eq:add1} with  zero boundary conditions. For such spaces with boundary conditions there are also proper $L^2$-bounded projections. 
\begin{definition}[Projections with BC in $3D$]
	Let us write $\widehat{\Pi}_{h,0}^1 \colon \f{H}^1_0(\widehat{\o},\mathbb{R}^3) \rightarrow \widehat{W}_h^1 \cap \f{H}^1_0(\widehat{\o},\mathbb{R}^3)$ for the classical spline projection  with zero boundary values in the parametric domain, where $\widehat{W}_h^1 \coloneqq \mathcal{Y}_1(W_h^1) $.  Compare (4.6) in \cite{Buffa2011IsogeometricDD}, \cite[Sec. 3.4]{IGA3}, respectively. 
	Then define 
	\begin{align*}
	{\Pi}_{h,0}^1 \colon \f{H}_0^1({\o},\mathbb{R}^3) \rightarrow W_{h,0}^1 \coloneqq {W}_{h}^1 \cap H_0^1(\o,\mathbb{R}^3) \ , \  
	{\Pi}_{h,0}^1 \coloneqq \mathcal{Y}_1^{-1} \circ \widehat{\Pi}_{h,0}^1  \circ \mathcal{Y}_1.\end{align*}
If we have  matrix-valued mappings we write  $\Pi_{h,0}^R$ for the row-wise application of the above projection, i.e.  $\Pi_{h,0}^R \colon H_0^1(\o,\mathbb{M}) \rightarrow R_{h,0}$, \ with $\Pi_{h,0}^R = ({\Pi}_{h,0}^1,{\Pi}_{h,0}^1,{\Pi}_{h,0}^1)^T$.  For details on spline projections we refer to \cite{IGA3}.
\end{definition}

\begin{lemma}
	\label{Lemma:spline_projection_properties}
	It exists a  constant $C_{P1}$ independent of $h$ s.t. \begin{align*}
	\n{\f{v}-\Pi_{h,0}^1\f{v}}_{{L}^2} &\leq C_{P1} h \n{\f{v}}_{{H}^1},  \ \forall \f{v} \in \f{H}^1_0(\o,\mathbb{R}^3), \\
	\n{\f{v}-\Pi_{h,0}^1\f{v}}_{{H}^1} &\leq C_{P1}  \n{\f{v}}_{{H}^1}, \ \forall \f{v} \in \f{H}^1_0(\o,\mathbb{R}^3). 
	\end{align*}
	Analogous results are valid for $\Pi_{h,0}^R$.
\end{lemma}
\begin{proof}
	By the assumptions concerning the parametrization $\p{F}$ we have constants $C_1,C_2 >0$ s.t. $  C_1 \ \n{\f{v}}_{{H}^1(\o)} \leq  \n{\mathcal{Y}_1(\f{v})}_{{H}^1(\widehat{\o})} \leq C_2  \n{\f{v}}_{{H}^1(\o)}$ and $C_1 \ \n{\f{v}}_{{L}^2(\o)} \leq  \n{\mathcal{Y}_1(\f{v})}_{{L}^2(\widehat{\o})} \leq C_2  \n{\f{v}}_{{L}^2(\o)} $. And from standard IGA results (see \cite{IGA3}) we know that there is a constant $\tilde{C}$ with $\n{\widehat{\f{v}}-\widehat{\Pi}_{h,0}^1\widehat{\f{v}}}_{L^2} \leq \tilde{C} h \n{\widehat{\f{v}}}_{{H}^1} $ and $
	\n{\widehat{\f{v}}-\widehat{\Pi}_{h,0}^1\widehat{\f{v}}}_{{H}^1} \leq \tilde{C}  \n{\widehat{\f{v}}}_{{H}^1}$  for all $\widehat{\f{v}} \in \f{H}_0^1(\widehat{\o},\mathbb{R}^3)$. Thus, for $\f{v} \in \f{H}^1_0(\o,\mathbb{R}^3)$ we have the estimate chain
	\begin{align*}
	\n{\f{v}-\Pi_{h,0}^1\f{v}}_{{L}^2(\o)}&= \n{\f{v}-\mathcal{Y}_1^{-1} \circ \widehat{\Pi}_{h,0}^1  \circ \mathcal{Y}_1\f{v}}_{{L}^2(\o)} \leq \frac{1}{C_1} \n{\mathcal{Y}_1\f{v}-\widehat{\Pi}_{h,0}^1  \circ \mathcal{Y}_1\f{v}}_{{L}^2(\widehat{\o})} \\
	& \leq \frac{\tilde{C}}{C_1} h \n{\mathcal{Y}_1\f{v}}_{H^1(\widehat{\o})} \leq \frac{\tilde{C}C_2}{C_1} h \n{\f{v}}_{H^1(\o)}.
	\end{align*}
	Analogous estimates can be made for the ${H}^1$ norm. This finishes the proof for the projection $\Pi_{h,0}^1$. The estimates for $\Pi_{h,0}^R$ follow by similar steps.
\end{proof}

After we clarified the spaces for the application example of linear elasticity, it remains to prove the Brezzi stability conditions. We focus on this in the next section.

\subsubsection{Well-posedness}
Next we check the well-posedness of our finite-dimensional problem:
\emph{Find}	$({\sigma}_h,u_h,p_h) \in (\Sigma_h,U_h,P_h)$ s.t. \vspace*{-0.4cm}\begin{alignat}{5}
\label{weak_form_wek_symmetry_discrete}
&\langle {A} {\sigma}_h , {\tau}_h \rangle \ + \ &&\langle u_h,\nabla \cdot {\tau}_h  \rangle \  + \ && \langle   {p}_h, \textup{Skew} \, {\tau}_h \rangle  \color{black}  &&= \langle {u}_D, {\tau}_h \cdot \nu \rangle_{\Gamma}, \ \hspace{0.4cm} &&\forall {\tau}_h \in \Sigma_h, \nonumber\\
&\langle \nabla \cdot {\sigma}_h , {v}_h \rangle &&  && &&= \langle {f},{v}_h\rangle, \ \hspace{0.4cm} && \forall {v}_h \in U_h, \\
&\langle \textup{Skew} \, {\sigma}_h, {q}_h \rangle &&  &&  &&= 0,  \ \hspace{0.4cm} &&\forall {q}_h \in P_h. \nonumber
\end{alignat}
With Def. \ref{Defnition of discrete spaces} and Lemma \ref{Lemma:well-posedness_discrete_abstract} we see that the only missing part for a proof of the well-posedness in the discrete setting  is the inf-sup condition   \eqref{eq:aux_inf_sup_discrete} corresponding to the diagrams in Fig \ref{Fig:5} and Fig. \ref{Fig:6}. Since the proof of the two-dimensional case is much easier than the one for $3D$ and since for the former we use an analogous proof like in \cite{Arnold2015}, we consider the poofs for both cases 
separately.

\begin{theorem}[Well-posedness]
	\label{Lemma:well-posedness}
	For the choice of discrete spaces in Def. \ref{Defnition of discrete spaces}  there exists a  $ h_{\max}>0$ s.t. we obtain for   \eqref{weak_form_wek_symmetry_discrete} and for all $0 < h \leq h_{\max}$ a unique solution. Then, the corresponding  Brezzi stability conditions are fulfilled; cf. Def. \ref{Brezzi:inf-sup-condition_abstract}. 
\end{theorem}
\begin{proof}[Proof for $2D$:]
	As already mentioned, we look only at the  inf-sup condition \eqref{eq:aux_inf_sup_discrete} which reads for Fig. \ref{Fig:5}: There is a constant $C_{IS}>0$ with \begin{equation*}
	\underset{{q}_h \in P_h }{\inf} \ \underset{{w}_h \in R_h}{\sup}   \ \ \frac{ \langle  \nabla \cdot  {w}_h, {q}_h\rangle}{\n{{w}_h}_{H^1} \ \n{{q}_h}_{L^2}}  >C_{IS}.
	\end{equation*}
	One notes that $(R_h,P_h) $ is an isogeometric Taylor-Hood space pair. From \cite[Theorem 4.1]{BressanandSangalli} we get  the validity of the above inequality for a suitable $C_{IS}$, which is independent of $h$, as long as $h \leq h_{\max}$ for a proper $h_{\max}>0$.  This finishes the first proof.
	
\end{proof}
\begin{proof}[Proof for $3D$:]
	Again it is enough to look at the following inf-sup condition	
	\begin{equation}
	\label{eq:inf_sup_3D}
	\underset{{v}_h \in P_h }{\inf} \ \underset{{w}_h \in R_h}{\sup}   \ \ \frac{ \langle  \nabla \cdot (\Xi {w}_h), {v}_h\rangle}{\n{{w}_h}_{\c} \ \n{{v}_h}_{L^2}} \geq C   >0,
	\end{equation}
	which implies the property (S2) in the Brezzi conditions, to prove inf-sup stability. Unfortunately, latter condition is not trivial to check and we follow the proof of the Taylor-Hood stability in \cite{Bressan} that is based on the so-called macroelement technique. Before we specify this technique, we need to adapt the preliminary steps made in the mentioned paper to our setting. Basic idea for preparations is to first reduce the   original inf-sup condition to a different one, that is derived utilizing Verfürth's trick. Applying this trick, we will consider an even stronger inf-sup condition involving the $H^1$-norm in the denominator. The detailed steps for the proof of the inf-sup conditions are outlined in the subsequent sections which result in Lemma \ref{Lemma:inf_sup_stability}. This lemma finally yields a $h_{\max}>0$ and a $C>0$ s.t. \eqref{eq:inf_sup_3D} is satisfied $\forall h \leq h_{\max}$.
\end{proof}
	\subsubsection{Verfürth's trick}
	Here and until Theorem \ref{Theorem:convergence estimates} we are in the $3D$ setting, i.e. $n=3$.\\
	We want to emphasize that the proof structure and steps are inspired by the stability proof of the Taylor-Hood space pair in \cite{Bressan} and \cite{BressanandSangalli}. Nevertheless, since the spaces we use differ and since in our setting matrix fields appear, we write down the important steps in a detailed fashion. We begin with a new auxiliary inf-sup condition which involves boundary conditions. \\
	First, one notes the fact that $$ H^1_0(\widehat{\o},\mathbb{M}) \cap (S_{\textsf{p},\textsf{p},\textsf{p}}^{r,r,r})^{3 \times 3} \subset \mathcal{Y}_1({R}_{h,0}), $$ with row-wise acting $\mathcal{Y}_1$. Furthermore,  it is $P_h \subset H^1(\o,\mathbb{R}^3)$.
	 Since $\Xi$ defines a $H^1$-bounded bijection  ${H}^1_0(\o,\mathbb{M}) \rightarrow {H}^1_0(\o,\mathbb{M})$ one has \begin{align*}
	\underset{\f{q} \in \f{L}_0^2(\o,\mathbb{R}^3)}{\inf} \ \ \underset{\f{\tau} \in H_0^1(\o,\mathbb{M})}{\sup}   \ \ \frac{\langle \nabla \cdot (\Xi \f{\tau}), \f{q}\rangle}{\n{\Xi \f{\tau}}_{H^1} \ \n{\f{q}}_{L^2}} =  \underset{\f{q} \in \f{L}_0^2(\o)}{\inf} \ \ \underset{\f{\tau} \in H_0^1(\o,\mathbb{M})}{\sup}   \ \ \frac{\langle \nabla \cdot  \f{\tau}, \f{q}\rangle}{\n{\f{\tau}}_{H^1} \ \n{\f{q}}_{L^2}} \geq C_{IS1} >0,
	\end{align*} 
	where the last inequality sign is clear due to the classical result of Girault and Raviart \cite{Girault}. Clearly, we obtain then \begin{equation}
	\underset{\f{q} \in \f{L}_0^2(\o,\mathbb{R}^3)}{\inf} \ \ \underset{\f{\tau} \in H_0^1(\o,\mathbb{M})}{\sup}   \ \ \frac{\langle \nabla \cdot (\Xi \f{\tau}), \f{q}\rangle}{\n{\f{\tau}}_{H^1} \ \n{\f{q}}_{L^2}} \geq C_{IS2} >0,
	\end{equation}
	where $ C_{IS2} = \frac{C_{IS1}}{C_c}$ is determined by the inequality $\n{\Xi\f{\tau}}_{H^1} \leq C_c \n{\f{\tau}}_{H^1}$.

	 Now let $\f{q} \in P_{h,0} \backslash \{0\}$ arbitrary, but fixed. Then we find a $\overline{\f{\tau}} \in H_0^1(\o,\mathbb{M})$ s.t. \begin{align}
	\label{eq:sec1_eq_1}
	\langle \nabla \cdot (\Xi \overline{\f{\tau}}), \f{q}\rangle &\geq C_{IS2} \n{\f{q}}_{{L}^2}^2 ,\\
	\n{\overline{\f{\tau}}}_{H^1} &= \n{\f{q}}_{L^2}. \nonumber
	\end{align} Integration by parts and \eqref{eq:sec1_eq_1}  lead to
	\begin{align}
	\label{eq:Verfürth_1}
	\langle \nabla \cdot (\Xi \Pi_{h,0}^R\overline{\f{\tau}}), \f{q}\rangle &= \langle \nabla \cdot (\Xi \overline{\f{\tau}}), \f{q}\rangle+ \langle \nabla \cdot (\Xi \Pi_{h,0}^R\overline{\f{\tau}}- \Xi \overline{\f{\tau}}), \f{q}\rangle  \nonumber \\
	& \geq C_{IS2} \n{\f{q}}_{{L}^2}^2 + \langle (\Xi \Pi_{h,0}^R\overline{\f{\tau}}- \Xi \overline{\f{\tau}}),  \nabla \f{q}\rangle.
	\end{align}
	And then we see for the last term in \eqref{eq:Verfürth_1}:
	\begin{align*}
	\langle (\Xi \Pi_{h,0}^R\overline{\f{\tau}}- \Xi \overline{\f{\tau}}),  \nabla \f{q}\rangle  &\leq \sum_{K \in \mathcal{T}_h} | \langle  \Xi \Pi_{h,0}^R\overline{\f{\tau}}- \Xi \overline{\f{\tau}} ,  \nabla \f{q}  \rangle_{L^2(K)} | \\
	& \leq \sum_{K \in \mathcal{T}_h} \n{\Xi \Pi_{h,0}^R\overline{\f{\tau}}- \Xi \overline{\f{\tau}}}_{L^2(K)} \n{\nabla \f{q}}_{L^2(K)} \\
	& \leq C_{\Xi} \sum_{K \in \mathcal{T}_h} h_K^{-1} \n{\Pi_{h,0}^R\overline{\f{\tau}}-  \overline{\f{\tau}}}_{L^2(K)} h_K\n{\nabla \f{q}}_{L^2(K)}\\
	& \leq C_{\Xi} \Big( \sum_{K \in \mathcal{T}_h} h_K^{-2} \n{\Pi_{h,0}^R\overline{\f{\tau}}-  \overline{\f{\tau}}}^2_{L^2(K)}\Big)^{1/2}  \Big( \sum_{K \in \mathcal{T}_h} h_K^2 \n{\nabla \f{q}}^2_{L^2(K)}\Big)^{1/2} .
	\end{align*}
	In view of Assumption \ref{Assumption:regular_triangul} and Lemma \ref{Lemma:spline_projection_properties} there is  a constant $C_{V1}< \infty$ with 
	\begin{align*}
	\langle (\Xi \Pi_{h,0}^R\overline{\f{\tau}}- \Xi \overline{\f{\tau}}),  \nabla \f{q}\rangle & \leq C_{V1} \n{\overline{\f{\tau}}}_{{H}^1(\o)}  \Big( \sum_{K \in \mathcal{T}_h} h_K^2 \n{\nabla \f{q}}_{L^2(K)}^2\Big)^{1/2} \\
	& \leq C_{V1} \n{\f{q}}_{{L}^2(\o)} \n{\f{q}}_{P,h} ,
	\end{align*}
	where $\n{\f{q}}_{P,h} \coloneqq \Big( \sum_{K \in \mathcal{T}_h} h_K^2 \n{\nabla \f{q}}_{L^2(K)}^2\Big)^{1/2}.$		
	\begin{remark}
		Indeed $\n{\cdot}_{P,h}$ defines a norm on $P_{h,0} \subset \f{L}_0^2(\o, \mathbb{R}^3)$.
	\end{remark}
The above estimate and \eqref{eq:Verfürth_1} yield directly the inequality 
\begin{align}
\langle \nabla \cdot (\Xi \Pi_{h,0}^R\overline{\f{\tau}}), \f{q}\rangle 
& \geq C_{IS2} \n{\f{q}}_{{L}^2}^2 - C_{V1} \n{\f{q}}_{{L}^2} \n{\f{q}}_{P,h} .
\end{align}
Consequently, setting $C_c \coloneqq 1+C_{P1}$ it is $ \n{{\Pi}_{h,0}^R \overline{\f{\tau}}}_{H^1} \leq C_c \n{ \overline{\f{\tau}}}_{H^1} $ (see Lemma \ref{Lemma:spline_projection_properties}) and
\begin{align}
\label{eq:inf_sup_2}
\underset{\f{\tau} \in R_{h,0}}{\sup}   \ \ \frac{\langle \nabla \cdot (\Xi \f{\tau}), \f{q}\rangle}{\n{\f{\tau}}_{H^1} } &\geq  \frac{1}{C_c \n{\overline{\f{\tau}}}_{H^1}} \big( C_{IS2} \n{\f{q}}_{{L}^2}^2 - C_{V1} \n{\f{q}}_{{L}^2(\o)} \n{\f{q}}_{P,h} \big) \nonumber \\
& =  \frac{C_{IS2}}{C_c} \n{\f{q}}_{{L}^2} - \frac{C_{V1}}{C_c}  \n{\f{q}}_{P,h}.
\end{align}
The last relation holds for all  $\f{q} \in P_{h,0}$.
Now, consider the inf-sup condition 
\begin{align}
\label{eq:inf_sup_cond}
\exists C_{IS}>0  \ \textup{s.t.} \ \forall h>0:  
\underset{\f{q} \in P_{h,0}}{\inf} \ \ \underset{\f{\tau} \in R_{h,0}}{\sup}   \ \ \frac{\langle \nabla \cdot (\Xi \f{\tau}), \f{q}\rangle}{\n{\f{\tau}}_{H^1} \n{\f{q}}_{{L}^2}}  \geq  C_{IS}.
\end{align}
The Verfürth trick (\cite{Verfuert}) can now be used to reduce latter inf-sup condition \eqref{eq:inf_sup_cond} to the new condition 
\begin{align}
\label{eq:inf_sup_cond3}
\exists C_{V}>0  \ \textup{s.t.} \ \forall h>0:  
\underset{\f{\tau} \in R_{h,0}}{\sup}   \ \ \frac{\langle \nabla \cdot (\Xi \f{\tau}), \f{q}\rangle}{\n{\f{\tau}}_{H^1} }  \geq  C_{V} \n{\f{q}}_{P,h}.
\end{align} 

Namely, using \eqref{eq:inf_sup_cond3} and \eqref{eq:inf_sup_2} we have
\begin{align*}
\underset{\f{\tau} \in R_{h,0}}{\sup}   \ \ \frac{\langle \nabla \cdot (\Xi \f{\tau}), \f{q}\rangle}{\n{\f{\tau}}_{H^1} }  &\geq \frac{C_{IS2}}{C_c} \n{\f{q}}_{{L}^2} - \frac{C_{V1}}{C_c}  \n{\f{q}}_{P,h}  \\
&\geq \frac{C_{IS2}}{C_c} \n{\f{q}}_{{L}^2} - \frac{C_{V1}}{C_cC_V}    \underset{\f{\tau} \in R_{h,0}}{\sup}   \ \ \frac{\langle \nabla \cdot (\Xi \f{\tau}), \f{q}\rangle}{\n{\f{\tau}}_{H^1} }.
\end{align*}
Hence \begin{align*}
\underset{\f{\tau} \in R_{h,0}}{\sup}   \ \ \frac{\langle \nabla \cdot (\Xi \f{\tau}), \f{q}\rangle}{\n{\f{\tau}}_{H^1} } \geq C_{IS } \n{\f{q}}_{{L}^2},
\end{align*}
with $C_{IS}= \frac{C_{IS2}}{C_c} \Big(1+ \frac{C_{V1}}{C_cC_V} \Big)^{-1}= \frac{C_{IS2}}{C_c} \Big( \frac{C_cC_V + C_{V1}}{C_cC_V} \Big)^{-1}= \frac{C_{IS2} C_V}{C_cC_{V}+C_{V1}}$.

Applying the Poincar\'e inequality  $\n{\f{\tau}}_{H^1} \leq (1+C_{Po}) |\f{\tau}|_{H^1}$ we can  simplify \eqref{eq:inf_sup_cond3} to 
\begin{align}
\label{eq:inf_sup_cond4}
\exists C_{V}^{'}>0,  \ \textup{s.t.} \ \forall h>0:  
\underset{\f{\tau} \in R_{h,0}}{\sup}   \ \ \frac{\langle \nabla \cdot (\Xi \f{\tau}), \f{q}\rangle}{|\f{\tau}|_{H^1} }  \geq  C_{V}^{'} \n{\f{q}}_{P,h}.
\end{align} 
Above $ |\cdot|_{H^1}$ stands for the standard Sobolev seminorm.

Thus it is justified to consider \eqref{eq:inf_sup_cond4} instead of  \eqref{eq:inf_sup_cond}.

\subsubsection{Macroelement technique}

Analogous to \cite{Bressan,BressanandSangalli}
we use the macroelement technique to  prove the condition  \eqref{eq:inf_sup_cond4}; compare \cite[Section 3.3]{Bressan}. Therefore we first define  the set of macroelement domains $\mathcal{M}_h^L$ corresponding to the triangulation $\mathcal{T}_h$. 
\begin{definition}[Macroelement domains]
	For $L \in \mathbb{N}_{>0}$ the set of macroelement domains $\mathcal{M}_h^L $ in a $n$-dimensional mesh $\mathcal{T}_h$ is defined as follows: If $L=1$ we have 
	\begin{align*}
	{\mathcal{M}}_h^1 &\coloneqq \{ M_K^1 \ | \ K \in \mathcal{T}_h , \ \# \{ N \in \mathcal{T}_h \ | N \subset M_K^1 \}=3^n  \}, \ \ \textup{where for }\ \ K \in \mathcal{T}_h  \ \ \textup{it is} \\
M_K^1 &\coloneqq \bigcup_{N \in \mathcal{T}_h \land \overline{N}  \cap \overline{K} \neq \emptyset} \overline{N} .  
	\end{align*}
	And recursively, for $L>1$
	\begin{align*}
	{\mathcal{M}}_h^L &\coloneqq \{ M_K^L \ | \ K \in \mathcal{T}_h , \ \# \{ N \in \mathcal{T}_h \ | N \subset M_K^L \}=(2L+1)^{n}  \}, \ \ \textup{where for }\ \ K \in \mathcal{T}_h  \ \ \textup{it is} \\
M_K^L &\coloneqq \bigcup_{N \in \mathcal{T}_h \land \overline{N}  \cap \overline{M_K^{L-1}} \neq \emptyset} \overline{N} .  
	\end{align*}
	We write $\#Q$ for the cardinality of a set $Q$.
\end{definition}

The $\mathcal{M}_h^L$ satisfy three important properties:
\begin{itemize}
	\item $\forall \mathcal{T}_h, \ \forall K \in \mathcal{T}_h, \ \exists  M \in \mathcal{M}_h^L $ s.t. $K \subset M$.
	\item $\forall \mathcal{T}_h, \ \forall K \in \mathcal{T}_h, $ there are at most $C_{\textup{overlap}} = (2L+1)^{n}$ many macroelement domains $M$ in $\mathcal{M}_h^L$ s.t. $K \subset M$.
	\item $\forall \mathcal{T}_h, \ \forall M \in \mathcal{M}_h^L, $ there are at most $C_{\textup{elem}} = (2L+1)^{n} $ different mesh elements in $M$.
\end{itemize}
An illustration of macroelements as collection of mesh elements is sketched in Fig. \ref{Fig:7} for $n=2$. Basic idea of the technique of macroelements is to reduce the global inf-sup condition  \eqref{eq:inf_sup_cond4} to a local alternative on the macroelements. Due to this localization step and a regular mesh assumption  one can then prove the new condition more easily by means of considerations on the simple parametric spline spaces and parametric meshes. 
In the following we drop the index $L$ and assume $L$ to be a fixed natural number. Later, we will come back to this issue and clarify what $L$ actually is.

\begin{figure}
	\centering
	\begin{tikzpicture}[scale=0.8]	
	
			\foreach \x in {-1,-0.5,0,0.5,1,1.5,2,2.5,3,3.5}{		
		\draw[scale=0.9] (\x,0) -- (\x,4+0.5);
		\draw[scale=0.9] (-1,\x+1,0) -- (3.5,\x+1);
	}

\draw[scale=0.9,green, very thick] (0,1)  rectangle (2.5,3.5);

\draw[->,scale=0.9] (-1.25,0)--(4.25,0);
\draw[->,scale=0.9] (-1,-0.25)--(-1,5.25);
\draw[scale=0.9] (-1.25,4.5)--(-1,4.5);
\draw[scale=0.9] (3.5,-0.25)--(3.5,0);
\node[scale=0.9,below] at (3.2,-0.2)  {$1$};
\node[below] at (-0.9,-0.2)  {$0$};
\node[below] at (-1.3,0.3)  {$0$};
\node[below] at (-1.3,4.35)  {$1$};

\node[below] at (1,4.8)  {$\widehat{\mathcal{T}}_h$};

    \node (myfirstpic) at (9,2) {\includegraphics[width=0.29\linewidth]{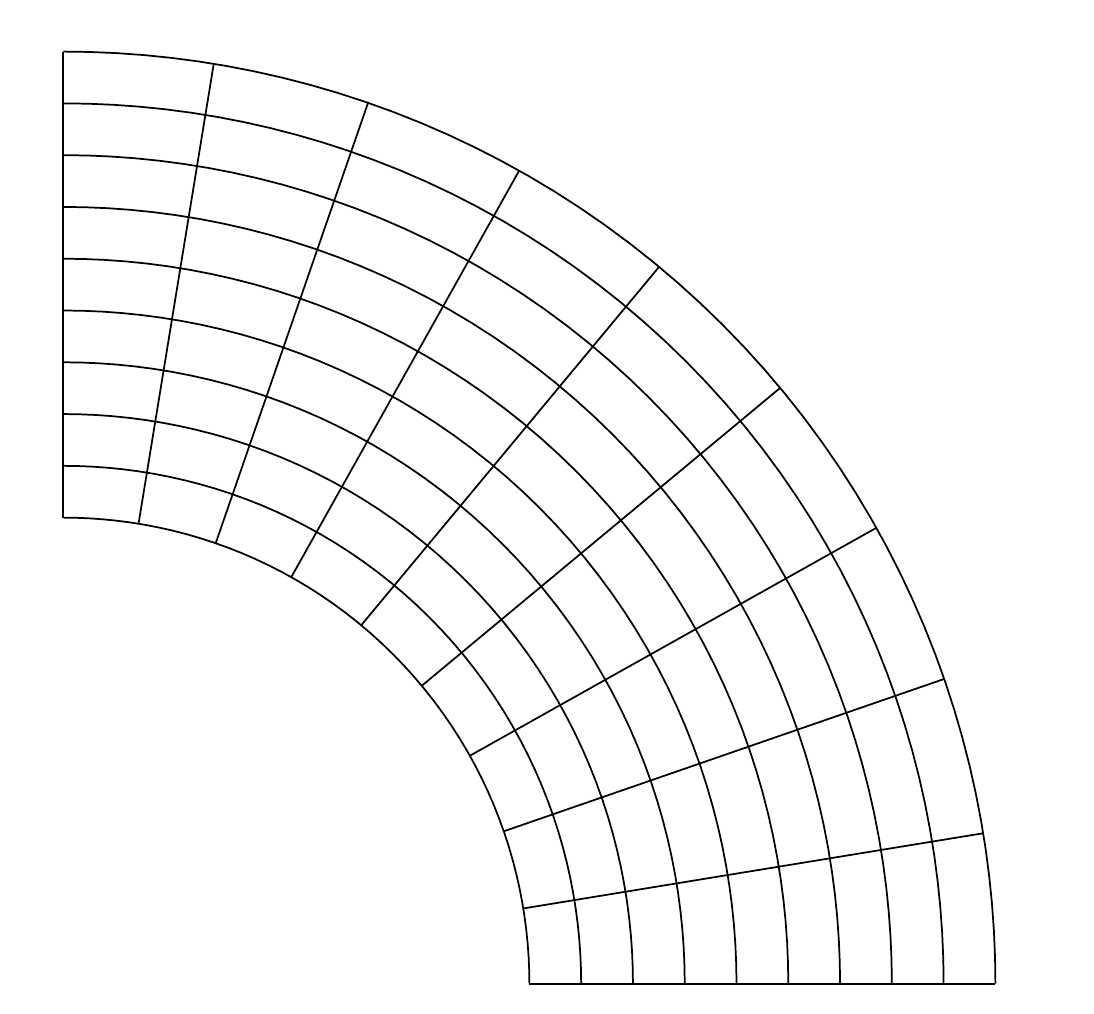}};
	\draw[scale=0.9] (-1,0)  rectangle (3.5,4.5);
	\draw[fill=blue,opacity=0.3,scale=0.9] (0.5,1.5)  rectangle (2,3);
		\draw[fill=red, opacity=0.7,scale=0.9] (1,2)  rectangle (1.5,2.5);
    \draw[scale=0.9, thick] (0.5,1.5)  rectangle (2,3);
	\draw[scale=0.9,fill=green, opacity = 0.1] (0,1)  rectangle (2.5,3.5);

	\draw [thick,->, out=30,in=150] (4,3) to (5.5,3);	
	\node at (5.25-0.5,3.5)  {$\p{F}$};

		\draw[ fill=blue,opacity=0.3, scale=2.65,shift={(2.355,-0.272)}] (30:1.33333)--(30:1.666666) 
	arc(30:60:1.66666)--(60:1.333333) arc(60:30:1.33333);
	\draw[ thick,opacity=0.7, scale=2.65,shift={(2.355,-0.272)}] (30:1.33333)--(30:1.666666) 
	arc(30:60:1.66666)--(60:1.333333) arc(60:30:1.33333);
	
	\draw[ fill=red,opacity=0.7, scale=2.65,shift={(2.355,-0.272)}] (40:1.33333+1/9)--(40:1.666666-1/9) 
	arc(40:50:1.66-1/9)--(50:1.333333+1/9) arc(50:40:1.33333+1/9);

		\draw[fill=green,opacity=0.1, scale=2.58,shift={(2.355+0.09,-0.272+0.022)}] (19:1.33333-1/9)--(19:1.666666+1/9) 
	arc(19:71:1.66666+1/9)--(71:1.333333-1/9) arc(71:19:1.33333-1/9);

		\draw[green, very thick,opacity=1, scale=2.58,shift={(2.355+0.09,-0.272+0.022)}] (19:1.33333-1/9)--(19:1.666666+1/9) 
	arc(19:71:1.66666+1/9)--(71:1.333333-1/9) arc(71:19:1.33333-1/9);

	\draw[->] (7-0.5+0.25,-0.5+0.25)--(8.5-0.5+0.25,-0.5+0.25);
	\draw[->] (6.5+0.25,-0.5+0.25)--(6.5+0.25,1+0.25);
	\node at (8,-0.5)  {$x$};
	\node at (6.5,1)  {$y$};

	\node[below] at (8,5)  {$\mathcal{T}_h$};
	\draw[->,dashed] (8.5,3.5) -- (11,4.3);
	\draw[->,dashed] (9.2,2.2) -- (11.4,2.5);
	\draw[->,dashed] (9.9,1.4) -- (12,1);
	\node[color=green!50!black] at (11.6,4.4)  {$M_K^2$};
	\node[color=red!80!black] at (11.8,2.6)  {$K$};
	\node[color=blue!50!black] at (12.55,0.9)  {$M_K^1$};
		\end{tikzpicture}
		\caption{Illustration of the macroelements $M_K^1$ and $M_K^2$ for the red highlighted mesh element $K \in \mathcal{T}_h$. For decreasing $h$ the macroelement shapes can be approximated by affine transformed squares. This observation together with the possibility of a covering of $\o$ by use of the macroelements yields the basic idea for the reduction of a global inf-sup condition to a localized version.}
		\label{Fig:7}
\end{figure}
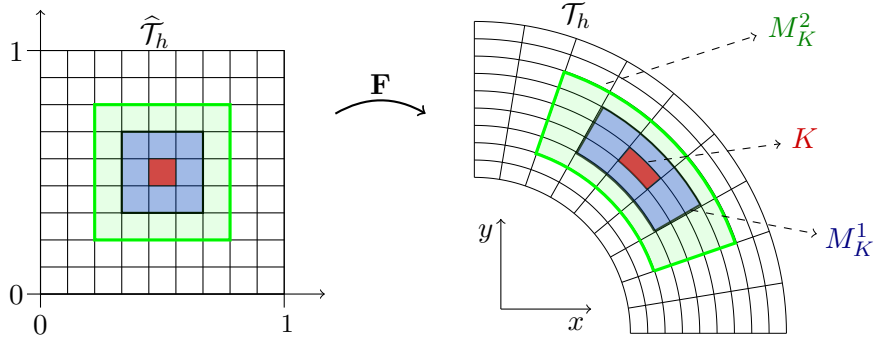
Macroelement domains can be defined analogously in the parametric domain and in this case we use the notation $\widehat{\mathcal{M}}_h$ as well as  $\widehat{M}_K $ for the parametric elements.
For each $M \in \mathcal{M}_h$ we can define the spaces 
\begin{align}
\t{V}_{M} &\coloneqq \{ \f{\tau}_{| M} \ | \ \f{\tau} \in R_{h,0}, \ \textup{supp}(\f{\tau}) \subset M   \}, \\
\mathcal{P}_{M} &\coloneqq \{ \f{v}_{| M} \ | \ \f{v} \in P_{h,0}, \ \int_{M}\f{v} d \f{x}=0    \}
\end{align}
and interpret mappings in latter spaces as mappings on $\o$  through a straight-forward zero extension. We equip  $\mathcal{P}_{M}$ with the norm \begin{equation*}
\n{\f{q}}_{\mathcal{P}_{M}} \coloneqq \Big( \sum_{K \subset M} h_K^2 \n{\nabla \f{q}}_{L^2(K)}^2\Big)^{1/2}
\end{equation*}
and note that $$\Pi_{\mathcal{P}_{M}}\f{q} \coloneqq \f{q}_{|M} - \frac{1}{|M|} \int_{M} \f{q} d \f{x} $$ defines a projection from $P_{h,0}$  onto $\mathcal{P}_{M}$. The integral is meant in a component-wise manner.

As already mentioned, the idea of the macroelement technique is to reduce the global inf-sup condition \eqref{eq:inf_sup_cond4}  to a local version on macroelements.
To be more precise, assume that
\begin{align}
\label{eq:inf_sup_macro}
\exists C_{\textup{macro}}>0, \ \ \textup{s.t.}  \ \ \forall h, \ \forall M \in \mathcal{M}_h: \ \  \ \underset{\f{q} \in \mathcal{P}_{M}}{\inf}  \ \underset{\f{\tau} \in \t{V}_M}{\sup}   \ \ \frac{\langle  \Xi \f{\tau}, \nabla \f{q}\rangle}{|\f{\tau}|_{H^1} \n{\f{q}}_{\mathcal{P}_M} }  \geq  C_{\textup{macro}}.
\end{align}
Then this  inf-sup stability on the macroelements leads to  stability  in the sense of \eqref{eq:inf_sup_cond4}; cf. \cite[Section 3.3]{Bressan}.  This can be seen as follows. Let $\f{q}  \in P_{h,0} $ arbitrary, but fixed. Further let $\f{\tau}_M \in \t{V}_M$ chosen such that \begin{align*}
\langle  \Xi \f{\tau}_M, \nabla \Pi_{\mathcal{P}_{M}}\f{q}\rangle  &\geq C_{\textup{macro}} \n{\Pi_{\mathcal{P}_{M}}\f{q}}_{\mathcal{P}_{M}}^2 \\
| \f{\tau}_M|_{H^1(\o)} &= \n{\Pi_{\mathcal{P}_{M}} \f{q}}_{\mathcal{P}_{M}}.
\end{align*}
On the one hand, one obtains
\begin{align*}
\underset{\f{\tau} \in R_{h,0}}{\sup}   \ \ {\langle  \Xi \f{\tau}, \nabla \f{q}\rangle}  &\geq \langle \sum_{M \in \mathcal{M}_h}  \Xi \f{\tau}_M , \nabla q \rangle =  \sum_{M \in \mathcal{M}_h} \langle  \Xi \f{\tau}_M , \nabla q \rangle  \\
& \geq \sum_{M \in \mathcal{M}_h} C_{\textup{macro}} \n{\Pi_{\mathcal{P}_{M}}\f{q}}_{\mathcal{P}_{M}}^2 \\
&=  \sum_{M \in \mathcal{M}_h} C_{\textup{macro}} \Big( \sum_{K \subset M} h_K^2 \n{\nabla \Pi_{\mathcal{P}_{M}}\f{q}}_{L^2(K)}^2\Big)  \\
& \geq C_{\textup{macro}} \sum_{K \in \mathcal{T}_h} \Big(  h_K^2 \n{\nabla\f{q}}_{L^2(K)}^2\Big) \\
&=C_{\textup{macro}} \n{\f{q}}_{P,h}^2.
\end{align*}
On the other hand, one can estimate
\begin{align*}
| \sum_{M \in \mathcal{M}_h}  \f{\tau}_M |^2_{H^1(\o)} & =   \sum_{K \in \mathcal{T}_h}| \sum_{M \in \mathcal{M}_h}  \f{\tau}_M |^2_{H^1(K)} \\
& \leq \sum_{K \in \mathcal{T}_h}  \Big(\sum_{ \underset{K \subset M}{M \in \mathcal{M}_h} }  |\f{\tau}_M| _{H^1(K)}  \Big)^2 \\
& \leq C_{\textup{overlap}}\sum_{K \in \mathcal{T}_h} \sum_{ \underset{K \subset M}{M \in \mathcal{M}_h} }  |\f{\tau}_M|^2 _{H^1(K)}  \\
& \leq C_{\textup{overlap}} \sum_{ {M \in \mathcal{M}_h} }  |\f{\tau}_M|^2 _{H^1(\o)}  \\
& \leq C_{\textup{overlap}} \sum_{ {M \in \mathcal{M}_h} }  \n{\Pi_{\mathcal{P}_{M}} \f{q}}_{\mathcal{P}_{M}}^2 \\
& \leq  C_{\textup{overlap}}\sum_{M \in \mathcal{M}_h} \sum_{ K \subset M }  h_K^2 \n{\nabla\f{q}}_{L^2(K)}^2  \\
& \leq   C_{\textup{overlap}}^2  \n{\f{q}}_{P,h}^2.
\end{align*}

Then, it is easy to see that the both inequality chains from above lead to wanted result \eqref{eq:inf_sup_cond4}. One can set $C_V^{'} = C_{\textup{macro}} C_{\textup{overlap}}^{-1}$.

Thus to show  the inf-sup stability \eqref{eq:inf_sup_cond} it  remains to prove  \eqref{eq:inf_sup_macro}.

\subsubsection{Proof of the auxiliary inf-sup condition  \eqref{eq:inf_sup_macro}}
Now we apply a result from \cite{BressanandSangalli}, where the inf-sup stability of a similar space pair is proven.
In this section we exploit several times the \emph{Einstein summation convention}.

\begin{lemma}
	There is a constant $C_{\textup{macro}}>0$ and a mesh size $h_{\textup{max}}>0$ s.t. for all $\mathcal{T}_h$ and $  \forall M \in \mathcal{M}_h$ it holds
	
	\begin{align}
	\underset{q \in {P}_{M}}{\inf} \ \  \underset{\f{v} \in \f{V}_M}{\sup}   \ \ \frac{\langle  \f{v}, \nabla q \rangle}{|\f{v}|_{H^1} \n{{q}}_{{P}_{M}} }  \geq  C_{\textup{macro}}, \ \ \textup{if}  \ h \leq h_{\max},
	\end{align}
	where 
	\begin{align}
	\f{V}_{M} &\coloneqq \{ \f{v}_{| M} \ | \ \f{v} \in \big(V_{h,3}^0(\textsf{p},r) \big)^3, \ \textup{supp}(\f{v}) \subset M   \}, \\
	P_{M} &\coloneqq \{ {q}_{| M} \ | \ q \in \bar{{W}}_{h}^3 = V_{h,3}^0(\textsf{p}-1,r), \ \int_{M}q dx=0  \}
	\end{align}
	and $\n{{q}}_{{P}_{M}} $ is analogously defined like $\n{{q}}_{\mathcal{P}_{M}} $, i.e. $\n{\f{q}}_{{P}_{M}} \coloneqq \Big( \small\sum_{K \subset M} h_K^2 \n{\nabla \f{q}}_{L^2(K)}^2\Big)^{1/2}$.
\end{lemma}

\begin{proof}
	Since we only consider spline basis functions this follows by Section 4.3 and Theorem 4.1 in \cite{BressanandSangalli} with NURBS weight function $\omega=1$. The mentioned reference shows that the macroelements with  $L = \textsf{p}-1$ are suitable.
\end{proof}
\begin{remark}
	Our definition of macroelements depends on $L$. However, the results in \cite[Theorem 4.1]{BressanandSangalli} imply that the choice $L=\textsf{p}-1$ leads to the inf-sup inequality of the last lemma. That is why we assume $\mathcal{M}_h = \mathcal{M}_h^{\textsf{p}-1}$ in the rest of the article.
\end{remark}
A row-wise application of the last lemma  yields the  existence of a  constant $C_{\textup{macro}}^{'}>0$ s.t. for all  macroelements $M \in \mathcal{M}_h$ and $\f{q}  \in  \f\mathcal{P}_{M} $ we find a $\f{\tau}_M \in \{ \f{\tau} \in (S_{\textsf{p},\textsf{p},\textsf{p}}^{r,r,r} \circ \p{F}^{-1})^{3\times 3} \ |  \  \textup{supp}(\f{\tau}) \subset  M \}$ with \begin{align}
\label{eq:macro_2}
\langle\f{\tau}_M, \nabla \f{q} \rangle   \geq C_{\textup{macro}}^{'} |\f{\tau}_M|_{H^1(M)}  \n{\f{q}}_{\mathcal{P}_{M}},
\end{align}
if $h \leq h_{\max}$. One notes that   $C_{\textup{macro}}^{'}$ is independent from $h , \ \f{q}$ and the macroelement. Next we choose for each macroelement  one $\f{x}_M \in M $, e.g. the push-forward of the center in the corresponding parametric macroelement $\widehat{M}$. Then define $\p{J}_{\f{x}_M} \coloneqq \p{J}(\f{x}_M)$ and set \begin{align*}
\tilde{\f{\tau}}_{M} \coloneqq \big[\Xi^{-1}(\f{\tau}_M)\big] \f{J}_{\f{x}_M}  \p{J}^{-1} \in \t{V}_M.
\end{align*} 
Indeed, we have $ \overline{{\f{\tau}}}_M \coloneqq  \big[\Xi^{-1}(\f{\tau}_M) \big] \p{J}_{\f{x}_M} \in (S_{\textsf{p},\textsf{p},\textsf{p}}^{r,r,r} \circ \p{F}^{-1})^{3\times 3}$ and the rows of $\tilde{\f{\tau}}_{M}$ are elements of $ \p{J}^{-T} (S_{\textsf{p},\textsf{p},\textsf{p}}^{r,r,r} \circ \p{F}^{-1})^{3} \in {W}_h^1$. Then, one gets
\begin{align}
\label{eq:aux_line1}
\langle \Xi\tilde{\f{\tau}}_{M}, \nabla \f{q} \rangle &=  	\langle \Xi \big[\overline{\f{\tau}}_{M} \p{J}_{\f{x}_M}^{-1}\big] , \nabla \f{q} \rangle + 	\langle \Xi \big[\overline{\f{\tau}}_{M} \big( \p{J}^{-1} - \p{J}_{\f{x}_M}^{-1} \big) \big], \nabla \f{q} \rangle \nonumber  \\
&= 	\langle {\f{\tau}}_{M}  , \nabla \f{q} \rangle + 	\langle \Xi \big[\overline{\f{\tau}}_{M} \big( \p{J}^{-1} - \p{J}_{\f{x}_M}^{-1} \big) \big], \nabla \f{q} \rangle 
\end{align}
The second term in line \eqref{eq:aux_line1} can be estimated as follows.
Introduce the auxiliary matrix $\tilde{\p{J}} \coloneqq \p{J}_{\f{x}_M} \big( \p{J}^{-1} - \p{J}_{\f{x}_M}^{-1} \big).$
We set $\f{\sigma}_M \coloneqq \Xi \big[ [\Xi^{-1}(\f{\tau}_M) ] \tilde{\p{J}}  \big]$ and have
\begin{align*}
\langle \f{\sigma}_M, \nabla \f{q} \rangle_{L^2(M)} &= 	\langle \widehat{\f{\sigma}}_M  |\textup{det}(\p{J})| ,  \widehat{\nabla} \widehat{\f{q}}\p{J}^{-1} \rangle_{L^2(\widehat{M})} ,
\end{align*}
where $\widehat{\f{\sigma}}_M  \coloneqq {\f{\sigma}}_M  \circ \p{F}$, \ $\widehat{q} = q \circ \p{F}$. With $\widehat{\f{\tau}}_M =\f{\tau}_M \circ \p{F}$ we can find  bounded and continuous coefficient functions $c_{ij}^{kl}=c_{ij}^{kl}(\tilde{\p{J}},\textup{det}(\p{J}),\p{J}^{-1},\Xi,\Xi^{-1})$ s.t.
\begin{align}
\label{eq:line_1}
\langle \f{\sigma}_M, \nabla \f{q} \rangle_{L^2(M)} &= \int_{\widehat{M}} (\widehat{\f{\tau}}_M)_{kl}  \ c_{ij}^{kl}   \ \h_j \widehat{\f{q}}_id \widehat{\f{x}}  \nonumber \\
&\leq  C\n{c_{ij}^{kl}}_{L^{\infty}(\widehat{M})} \n{\widehat{\f{\tau}}_M}_{L^2(\widehat{M})}    \n{\widehat{\nabla}\widehat{\f{q}}}_{L^2(\widehat{M})} \nonumber \\
& \leq  C\tilde{C} \n{c_{ij}^{kl}}_{L^{\infty}(\widehat{M})}  h |\widehat{\f{\tau}}_M|_{H^1(\widehat{M})}  \n{\widehat{\nabla}\widehat{\f{q}}}_{L^2(\widehat{M})} \nonumber
\\ & \leq \bar{C} \n{c_{ij}^{kl}}_{L^{\infty}(\widehat{M})} |\widehat{\f{\tau}}_M|_{H^1(\widehat{M})} \Big( \sum_{ K \subset \widehat{M} } h \n{\widehat{\nabla}\widehat{\f{q}}}_{L^2(K)}\Big) \nonumber \\
& \leq \bar{C} \sqrt{C_{\textup{elem}}}  \n{c_{ij}^{kl}}_{L^{\infty}(\widehat{M})} |{\widehat{\f{\tau}}}_{{M}}|_{H^1({\widehat{M}})} \Big( \sum_{ K \subset \widehat{M} } h^2 \n{\widehat{\nabla}\widehat{\f{q}}}_{L^2(K)}^2\Big)^{1/2} 
\end{align}
To obtain the second inequality sign above we use the estimate for the Poincare constant $C_{Po}(D)$ of rectangular domains $D$:
\begin{equation*}
C_{Po}(D) \leq {\textup{diam}(D)} \ \  \Rightarrow \ \  C_{Po}(\widehat{M}) \leq \tilde{C} h, \ \textup{for a proper } \tilde{C}, \ \ \  \textup{see \cite{Braess}}.
\end{equation*}

Analogously to \cite{Bressan},  one introduces the auxiliary mapping $\widehat{\f{q}}_a  \coloneqq \widehat{\f{q}} - \frac{1}{|\widehat{M}|} \int_{\widehat{M}} \widehat{\f{q}} d \widehat{\f{x}}$.  Obviously, we have $\widehat{\nabla} \widehat{\f{q}} = \widehat{\nabla}\widehat{\f{q}}_a$. Besides, with the norm equivalence statement in line (3.37) of \cite{Bressan} we obtain for some constant $C_{H^1}$  independent of the mesh size
\begin{equation}
\label{eq:line_2}
|\widehat{\f{\tau}}_M|_{H^1(\widehat{M})} \leq C_{H^1} |{\f{\tau}}_M|_{H^1({M})}.
\end{equation}
Further, by the regularity assumption on $\p{F}$, there is for each $\varepsilon >0$ a mesh size $h_{\varepsilon}>0$ s.t. for all $h \leq h_{\varepsilon}$ we get \begin{align*}
|\tilde{\p{J}}|= \big|\p{J}_{\f{x}_M} \big(\p{J}^{-1}(\widehat{\f{x}})-\p{J}^{-1}_{\f{x}_M}\big)\big|&< \varepsilon,  \ \ \ \forall M \in \mathcal{M}_h, \forall \widehat{\f{x}} \in \widehat{M}.	
\end{align*}
Hence, since  $\f{\sigma}_M = \Xi \big[ [\Xi^{-1}(\f{\tau}_M) ] \tilde{\p{J}}  \big]$ there is a constant $C_{\textup{para}}< \infty$ which is independent of $h$ and the macroelement such that \begin{equation} \label{eq:line_3}
|c_{ij}^{kl}| < C_{\textup{para}} \varepsilon \ \ \textup{on} \ \widehat{M},
\end{equation}
for $h$ small enough.
Combining  \eqref{eq:line_2} and \eqref{eq:line_3} we find a constant $C_{\textup{para}}^{'} < \infty$ with 
\begin{align}
\label{eq:line_4}
|\langle \f{\sigma}_M, \nabla \f{q} \rangle_{L^2(M)}| & \leq C_{\textup{para}}^{'} \varepsilon   |{\f{\tau}}_M|_{H^1({M})}   \Big( \sum_{ K \subset \widehat{M} } h^2 \n{\widehat{\nabla}\widehat{\f{q}}_a}_{L^2(K)}^2\Big)^{1/2} .
\end{align}
Applying the inequality chain (3.47) in \cite{Bressan} we obtain a constant $C_{\textup{est}}$ s.t.
\begin{equation}
\label{eq:line_5}
\Big( \sum_{ K \subset \widehat{M} } h^2 \n{\widehat{\nabla}\widehat{\f{q}}_a}_{L^2(K)}^2\Big)^{1/2} \leq C_{\textup{est}}  \n{\f{q}}_{\mathcal{P}_M} .
\end{equation}
In view of  \eqref{eq:line_4} and \eqref{eq:line_5} it is
\begin{align}
	\label{eq:estimate_tilde_1}
\langle \Xi\tilde{\f{\tau}}_{M}, \nabla \f{q} \rangle \geq C_{\textup{macro}}^{'} |\f{\tau}_M|_{H^1(M)}  \n{\f{q}}_{\mathcal{P}_M} - C_{\textup{para}}^{'}   \varepsilon   C_{\textup{est}} 	|{\f{\tau}}_M|_{H^1({M})}  \n{  {\f{q}}}_{\mathcal{P}_M}.
\end{align}
Hence, if we choose $h \leq h_{\max} $ for some suitable $h_{\max}>0$ we get  $  C_{\textup{macro}}^{''} = C_{\textup{macro}}^{'} - C_{\textup{para}}^{'}   \varepsilon   C_{\textup{est}}  >0 $.
And since $\p{J}$ is $C^1$ smooth by assumption, we obtain for some $C_{\textup{est},1}, C_{\textup{est},2}$ and with the Poincar\'e inequality:
\begin{align*}
|\tilde{\f{\tau}}_M|_{H^1(M)} &=  \bigg| \big[\Xi^{-1}(\f{\tau}_M)\big] \p{J}_{\f{x}_M}  \p{J}^{-1}  \bigg|_{H^1(M)} \leq  C_{\textup{est},2} \n{ {\f{\tau}}_M}_{H^1(M)} \\
& \leq C_{\textup{est},2} (1+C_{Po}(\o)) |\f{\tau}_M|_{H^1(M)}
\end{align*}

One notes $M \subset \o$, i.e. the Poincar\'{e} constant can be chosen independently from $M$. \\
With the last  estimate and \eqref{eq:estimate_tilde_1} we can conclude now 
\begin{align}
\langle \Xi\tilde{\f{\tau}}_{M}, \nabla \f{q} \rangle &\geq C_{\textup{macro}}^{''} |{\f{\tau}}_M|_{H^1({M})}  \n{  {\f{q}}}_{\mathcal{P}_M} \geq  \frac{C_{\textup{macro}}^{''}}{C_{\textup{est},2} (1+C_{Po}(\o))} | \tilde{\f{\tau}}_M|_{H^1(M)} \n{  {\f{q}}}_{\mathcal{P}_M}.
\end{align}

In other words with $C_{\textup{macro}} = \frac{C_{\textup{macro}}^{''}}{C_{\textup{est},2} (1+C_{Po}(\o))}>0$ and for $h  \leq h_{\textup{max}}$ we have shown the property \eqref{eq:inf_sup_macro}.

\subsubsection{Inf-sup condition for spaces without boundary conditions}

From the last section together with  the Verfürth trick we get the validity of  \eqref{eq:inf_sup_cond} if $h \leq h_{\max}$.  However, there we assumed special boundary conditions and the property of vanishing mean value for the variable $q$. Thus, we want to prove the next lemma.
\begin{lemma}
	\label{Lemma:inf_sup_stability}
	There is a mesh size $h_{\max}>0$ and a constant $C_{IS}^{'}>0$ s.t. 
	\begin{align}
	\label{eq:inf_sup_cond_without_BC}
	\ \forall 0< h \leq h_{\max}:  
	\underset{\f{q} \in \f{P}_h}{\inf} \ \underset{\f{\tau} \in R_h \cap H^1(\o,\mathbb{M})}{\sup}   \ \ \frac{\langle \nabla \cdot (\Xi \f{\tau}), \f{q}\rangle}{\n{\f{\tau}}_{H^1} \n{\f{q}}_{{L}^2}}  \geq  C_{IS}^{'}.
	\end{align}
\end{lemma} 
Before proving this lemma we first look at a useful auxiliary result.

\begin{lemma}
	\label{Lemma:aux_1}
	It exists a $\tilde{h}>0$ such that   $ \forall h \leq \tilde{h}$, we find $\f{M}_h^i  \in R_h \cap H^1(\o,\mathbb{M}), \ i = 1,2,3$ with 
	\begin{align*}
	\n{\f{M}_h^i}_{H^1} &\leq C_i < \infty , \\
	\n{\nabla \cdot (\Xi\f{M}_h^i) - \f{e}_i}_{L^2} &\leq \min\{ \frac{C_{IS}}{2 \sqrt{C_{\o}}}, \frac{1}{5 \sqrt{C_{\o}}}\} \eqqcolon C_{L,1}, \ \ \forall i, \ \ C_{\o} \coloneqq 1/|\o|. 
	\end{align*}
	Above the $C_i$ are independent of $h$ and $\f{e}_i$ denote the canonical basis vectors in $\mathbb{R}^3$. 
\end{lemma} 

\begin{proof}
In this proof we write $\Pi_h = (\Pi_h^1,\Pi_h^1,\Pi_h^1)^T$, i.e. we have a row-wise application of $\Pi_h^1$, where $\Pi_h^{1} \coloneqq \Pi_{h,3}^1(\textsf{p}+1,r+1) $ is the spline projection from Sec. \ref{section:splines} onto $V_{h,3}^1(\textsf{p}+1,r+1)$ with degree $\textsf{p}+1$ and regularity $r+1$. In particular it holds $$\Pi_h \tau \in R_h \cap H^1(\o,\mathbb{M}), \ \ \forall \tau \in H^1(\o,\mathbb{M}).$$ Now	introduce the matrix $\f{M}_h^1 \coloneqq \Pi_h \bigg[  \Xi^{-1}\bigg($ \small $\begin{pmatrix}
	x_1 & 0 & 0  \\
	0 & 0 & 0 \\
	0 & 0 & 0
	\end{pmatrix}$ \normalsize $ \bigg)   \bigg] \in R_h$	
	and $\f{X}_1 \coloneqq $  \small $\begin{pmatrix}
	x_1 & 0 & 0  \\
	0 & 0 & 0 \\
	0 & 0 & 0
	\end{pmatrix}$.
	Exploiting the approximation properties of the splines, cf. \cite[Theorem 5.3 and Remark 5.1]{Buffa2011IsogeometricDD}, one obtains
	\begin{align*}
	\n{\nabla \cdot \Xi\f{M}_h^1 - \f{e}_1}_{L^2}  &\leq \n{\Xi \f{M}_h^1 - \Xi \Xi^{-1} \f{X}_1 }_{H^1} \leq C_{\Xi} \n{ \f{M}_h^1 -  \Xi^{-1} \f{X}_1 }_{H^1} \\
	& \leq C_{\Xi} \n{ \Pi_h \Xi^{-1} \f{X}_1 -  \Xi^{-1} \f{X}_1 }_{H^1} \\
	& \leq C_{\Xi} C_{R} h \n{\Xi^{-1} \f{X}_1}_{H^2} 
	\end{align*}
	fo some constant $C_R$.
	And by the continuity of the projection operators it is
	\begin{align*}
	\n{ \f{M}_h^1 }_{H^1} = 	\n{ \Pi_h \Xi^{-1} \f{X}_1 }_{H^1} \leq (1+C_R) \n{\Xi^{-1} \f{X}_1}_{H^1} \eqqcolon C_1.
	\end{align*}
	Since $\f{X}_1$ is fixed we get directly the assertion for $i=1$. The cases $i=2, \ i=3$ can be treated similarly.
\end{proof}

\begin{proof}[Proof of Lemma \ref{Lemma:inf_sup_stability}]
	Let $\f{q} \in \f{P}_h$ arbitrary, but fixed. Further we write $\tilde{\f{q}} \coloneqq \f{q} - \bar{\f{q}},$ where \begin{equation*}
	\bar{\f{q}}_i \coloneqq \frac{1}{|\o|} \int_{\o} \f{q}_i d \f{x} = \textup{const.}
	\end{equation*}
	Then we define \begin{equation*}
	\overline{\f{M}}_h^i \coloneqq \bar{\f{q}}_i \ \f{M}_h^i, \ \ \  \ 
	\end{equation*}
	with the $\f{M}_h^i$ from Lemma  \ref{Lemma:aux_1}.
	By the previous section we know that there is a $\tilde{\f{\tau}}  \in R_h \cap \f{H}_0^1(\o,\mathbb{M})$ with \begin{align*}
	\langle \nabla \cdot(\Xi \tilde{\f{\tau}}) , \tilde{\f{q}} \rangle &\geq C_{IS}  \n{\tilde{\f{q}}}_{L^2}^2 ,\\
	\n{\tilde{\f{\tau}}}_{H^1} & = \n{\tilde{\f{q}}}_{L^2}.
	\end{align*}
	Set $ \f{\tau} \coloneqq \tilde{\f{\tau}} + \sum_{ i}  \overline{\f{M}}_h^i$. This means	
	\begin{align*}
	\langle \nabla \cdot  ( \Xi {\f{\tau}}) , {\f{q}} \rangle & = \langle  \nabla \cdot( \Xi \tilde{\f{\tau}} ), \tilde{\f{q}} \rangle +\underbrace{ \sum_{ i} \langle  \nabla \cdot (\Xi \overline{\f{M}}_h^i) , \tilde{\f{q}} \rangle}_{\eqqcolon (\star,1)} \\
	&+ \underbrace{\langle  \nabla \cdot( \Xi \tilde{\f{\tau}} ), \bar{\f{q}} \rangle}_{\eqqcolon (\star,2)} + \underbrace{\sum_{ i} \langle  \nabla \cdot (\Xi \overline{\f{M}}_h^i) , \bar{\f{q}} \rangle}_{\eqqcolon (\star,3)}.
	\end{align*}
	Using Lemma \ref{Lemma:aux_1} the three last terms can be estimated as follows:
	\begin{align*}
	|(\star,1)|  \leq \sum_{ i} |\langle  \nabla \cdot (\Xi \overline{\f{M}}_h^i) , \tilde{\f{q}} \rangle| &\leq \sum_{ i} |\langle  \nabla \cdot (\Xi \overline{\f{M}}_h^i)- \bar{\f{q}}_i \f{e}_i , \tilde{\f{q}} \rangle| + |\langle  \bar{\f{q}}_i \f{e}_i , \tilde{\f{q}} \rangle| \\
	& \leq  \sum_{ i} |\bar{\f{q}}_i| |\langle  \nabla \cdot (\Xi {\f{M}}_h^i)- \f{e}_i , \tilde{\f{q}} \rangle| + 0  \\
	& \leq \sum_{ i} C_{L,1} |\bar{\f{q}}_i| \n{\tilde{\f{q}}}_{L^2} \\
	&  \leq \sum_{ i} C_{L,1} \sqrt{C_{\o}} \n{\bar{\f{q}}_i}_{L^2} \n{\tilde{\f{q}}}_{L^2}  \\
	&  \leq  2C_{L,1}  \sqrt{C_{\o}} \n{\bar{\f{q}}}_{L^2} \n{\tilde{\f{q}}}_{L^2}\\	 
	& \leq C_{L,1}  \sqrt{C_{\o}}  \big( \n{\bar{\f{q}}}_{L^2}^2 + \n{\tilde{\f{q}}}_{L^2}^2 \big).\\	
	(\star,2)  = \langle   \Xi \tilde{\f{\tau}} , \nabla \bar{\f{q}} \rangle &=0.\\ (\star,3)  = \sum_{ i}  \langle  \nabla \cdot (\Xi \overline{\f{M}}_h^i) , \bar{\f{q}} \rangle   	& = 	\sum_{ i,j}  \langle  \nabla \cdot (\Xi \overline{\f{M}}_h^i) , \bar{\f{q}}_j \f{e}_j \rangle   \\
	& = 	\sum_{ i,j}  \langle  \nabla \cdot (\Xi \overline{\f{M}}_h^i)-\bar{\f{q}}_i \f{e}_i , \bar{\f{q}}_j \f{e}_j \rangle   + 	\sum_{ i,j}  \langle  \bar{\f{q}}_i \f{e}_i , \bar{\f{q}}_j \f{e}_j \rangle  
	\\	& \geq \sum_{i} \n{\bar{\f{q}}_i}_{L^2}^2- \sum_{i,j} \sqrt{C_{\o}} C_{L,1} \ \n{\bar{\f{q}}_i}_{L^2} \n{\bar{\f{q}}_j}_{L^2} \\
	& \geq \sum_{i} \n{\bar{\f{q}}_i}_{L^2}^2- 3 \sqrt{C_{\o}} C_{L,1} \ \n{\bar{\f{q}}_i}_{L^2}^2 \\
	& \geq (1- 3\sqrt{C_{\o}} C_{L,1}) \n{\bar{\f{q}}}_{L^2}^2 .
	\end{align*}
	Consequently, we have
	\begin{align}
		\label{eq:non-BC_estimate_1}
	\langle \nabla \cdot  ( \Xi {\tau}) , {\f{q}} \rangle & \geq C_{IS} \n{\tilde{\f{q}}}_{L^2}^2 + (1- 3\sqrt{C_{\o}} C_{L,1}) \n{\bar{\f{q}}}_{L^2}^2  - C_{L,1}  \sqrt{C_{\o}}  \big( \n{\bar{\f{q}}}_{L^2}^2 + \n{\tilde{\f{q}}}_{L^2}^2 \big) \nonumber \\
	& =  \big( C_{IS}-C_{L,1}  \sqrt{C_{\o}} \big)\n{\tilde{\f{q}}}_{L^2}^2 + \big(1- 4 C_{L,1}  \sqrt{C_{\o}} ) \n{\bar{\f{q}}}_{L^2}^2 \nonumber\\
	& \geq C_{IS,2} (\n{\tilde{\f{q}}}_{L^2}+ \n{\bar{\f{q}}}_{L^2})^2, 
	\end{align}
	for a suitable $C_{IS,2}>0$ if $C_{L,1}$ is small enough, see Lemma \ref{Lemma:aux_1}.\\
	And on the other hand we can estimate 
	\begin{align}
			\label{eq:non-BC_estimate_2}
	\n{\f{\tau}}_{H^1} & \leq \n{\tilde{\f{\tau}}}_{H^1} + \sum_{ i}  \n{\overline{\f{M}}_h^i}_{H^1} \nonumber \\
	&  \leq  \n{\tilde{\f{q}}}_{L^2} + \sum_{ i} |\bar{\f{q}}_i| \n{{\f{M}}_h^i}_{H^1} \nonumber \\
	& \leq \n{\tilde{\f{q}}}_{L^2} + \sum_{ i} \sqrt{C_{\o}} C_i \n{\bar{\f{q}}}_{L^2} \nonumber \\
	& \leq C_{IS,3} (\n{\tilde{\f{q}}}_{L^2} + \n{\bar{\f{q}}}_{L^2}), 
	\end{align}
where $C_{IS,3} \coloneqq \max\{1,3\sqrt{C_{\o}}C_i \}$.
	Finally, we can combine \eqref{eq:non-BC_estimate_1} and \eqref{eq:non-BC_estimate_2} to see 
	\begin{align*}
	\frac{\langle \nabla \cdot  ( \Xi {\f{\tau}}) , {\f{q}} \rangle}{\n{\f{\tau}}_{H^1} } & \geq \frac{ C_{IS,2}}{ C_{IS,3}} (\n{\tilde{\f{q}}}_{L^2} + \n{\bar{\f{q}}}_{L^2})  \geq \frac{ C_{IS,2}}{ C_{IS,3}} ( \n{{\f{q}}}_{L^2}).
	\end{align*}
	The arbitrariness of $\f{q}$ finishes the proof.
	
\end{proof}

With Lemma \ref{Lemma:inf_sup_stability} it follows the well-posedness of the discrete method also for three dimensions. Thereby, we obtain the next error estimate which corresponds to Corollary \ref{Corollary:estimate} in the abstract setting.

\begin{theorem}
	\label{Theorem:convergence estimates}
	Let $\textsf{p}  \geq 2, \ \textsf{p}-1>r \geq 0, \  \textsf{p} \geq s \geq 0$.
	And let $(\sigma,u,p)$ be the exact solution of the continuous problem \eqref{weak_form_wek_symmetry_contiuous} and $(\sigma_h,u_h,p_h)$ the approximate solutions obtained utilizing the discrete spaces from Def. \ref{Defnition of discrete spaces}.
	Then, assuming enough regularity for $(\sigma,u,p)$, it holds for a constant $C_{conv}$ independent of $h$:\begin{align*}
	\norm{[\f{\sigma}-\f{\sigma}_h,\f{u}-\f{u}_h,\f{p}-\f{p}_h]}_\mathcal{B} &\leq C_{conv} h^{s} \ \big( \norm{\f{\sigma}}_{H^s(\d)}  + \norm{\f{u}}_{H^s} + \norm{\f{p}}_{H^s} \big),
	\end{align*} 
	where $\norm{[\tau,\f{v},\f{q}]}_\mathcal{B}^2 = \norm{\tau}_{\d}^2  + \norm{\f{v}}_{L^2}^2 + \norm{\f{q}}_{L^2}^2.$
	For the spaces in the $3D$ setting we further have the improved estimate:
	$$\n{\nabla \cdot \f{\sigma}-\nabla \cdot \f{\sigma}_h}_{L^2} \leq C_{conv} h^{\textsf{p}+1} \n{\nabla \cdot \f{\sigma}}_{H^{\textsf{p}+1}}.$$
\end{theorem}
\begin{proof}
	We look at the proof in $3D$. The estimate for $2D$ follows by similar steps. \\
	By the well-posedness, cf. Theorem \ref{Lemma:well-posedness}, we have a quasi-optimal estimate like in Corollary \ref{Corollary:estimate}. Thus due to projection properties in Corollary \ref{Corollary:Projections} one obtains easily the first estimate. More precisely, there is a constant $C_{QO}$ with
	\begin{align*}
	\norm{[\f{\sigma}-\f{\sigma}_h,\f{u}-\f{u}_h,\f{p}-\f{p}_h]}_\mathcal{B}  &\leq C_{QO} \norm{[\f{\sigma}-\Pi_{h}^{\Sigma} \f{\sigma},\f{u}-\Pi_{h}^{U}\f{u},\f{p}-\Pi_{h}^P\f{p}]}_\mathcal{B}  \\
& \leq  C_{QO} \big( \norm{\f{\sigma}-\Pi_{h}^{\Sigma} \f{\sigma}}_{\d}  + \norm{\f{u}-\Pi_{h}^{U}\f{u}}_{L^2} + \norm{\f{p}-\Pi_{h}^P\f{p}}_{L^2}  \big) \\
	& \leq C h^{s} \ \big( \norm{\f{\sigma}}_{H^s(\d)}  + \norm{\f{u}}_{H^s} + \norm{\f{p}}_{H^s} \big),
	\end{align*}
where $C$ in the last line depends on $C_{QO}$ and the  projections.
	Th second estimate of the assertion can be obtained by a standard approach. We observe the fact $ \langle \nabla \cdot \sigma - \nabla \cdot \sigma_h , w_h \rangle=0, \ \forall w_h \in U_h$ and that $\nabla \cdot \tau_h \in U_h, \forall \tau_h \in \Sigma_h$. This means we have the equality chain
	\begin{align*}
	\n{\nabla \cdot \sigma-\nabla \cdot \sigma_h}^2_{L^2} &= \langle \nabla \cdot \sigma-\nabla \cdot \sigma_h  , \nabla \cdot \sigma-\nabla \cdot \sigma_h \rangle \\
	&= \langle \nabla \cdot \sigma-\Pi_h^U \nabla \cdot \sigma, \nabla \cdot \sigma-\nabla \cdot \sigma_h \rangle + \langle \Pi_h^U\nabla \cdot \sigma-\nabla \cdot \sigma_h , \nabla \cdot \sigma-\nabla \cdot \sigma_h \rangle \\
	& \leq \n{\nabla \cdot \sigma- \Pi_h^U \nabla \cdot \sigma}_{L^2} \ \n{ \nabla \cdot \sigma - \nabla \cdot \sigma_h}_{L^2}+0
	\end{align*}
	Thus, again in view of Corollary  \ref{Corollary:Projections}, we get \newline  \hspace{1cm} \ \ \ $ \hspace{1cm} \n{\nabla \cdot \sigma -\nabla \cdot \sigma_h}_{L^2} \leq \n{\nabla \cdot \sigma- \Pi_h^U \nabla \cdot \sigma}_{L^2} \leq C h^{\textsf{p}+1}  \n{\nabla \cdot \sigma}_{H^{\textsf{p}+1}}$. Since we use $V_{h,3}^2(\textsf{p}+1,r)$ to define $\Sigma_h$ the last estimate is clear. \\
	The statement for $2D$ can be proven similarly and we omit this part here.
\end{proof}

With the last theorem and with Theorem \ref{Lemma:well-posedness} we can summarize that the   discretization  of the mixed formulation of linear elasticity with weakly  imposed symmetry utilizing  the proposed method is possible and from a theoretical point of view we should obtain arbitrarily high convergence orders, provided regular enough exact solutions. Indeed, we showed the well-posedness only for $h \leq h_{\max}$  for some proper $h_{\max}>0$. But in all the numerical experiments we carried out, we do not see problems for large $h$. Maybe a closer study of the proofs in order to specify $h_{\max}$ would be advisable. Another point one should mention are the assumed boundary conditions. In the sections above we considered the easy case of pure displacement boundary conditions. Clearly, in applications the case of mixed boundary conditions involving traction forces is more interesting. This is why we   give some remarks concerning generalizations of our approach including such mixed boundary conditions in the next section.

\subsection{Brief remarks on generalizations and stability}

\subsubsection{Mixed boundary conditions}

Up to now we studied the simple case of pure displacement boundary conditions specified by $u_D \in H^{1/2}(\partial \o,\mathbb{R}^n)$. However another important type of boundary conditions are the so-called traction forces. More precisely, one has  a disjoint partition of the boundary  $ \overline{\Gamma_t}   \cup \overline{\Gamma_D} = \partial \o, \ \Gamma_t \cap \Gamma_D = \emptyset$ together with a suitable mapping $t_n $ and the requirement $$\sigma \cdot \nu = t_n  \ \textup{on} \   \Gamma_t  \ \ \ \textup{and} \ \ \ u = u_D \ \textup{on} \ \Gamma_D.$$
If the traction forces are prescribed by a given auxiliary stress mapping $\tilde{\sigma}$, i.e. $\tilde{\sigma} \cdot \nu = t_n$, we can reduce the mixed boundary conditions to the case $t_n =0$ by considering $\sigma-\tilde{\sigma}$. Latter condition can be implemented in the discrete setting very easily by adapting  $\Sigma_h$ to $$\Sigma_{h,\Gamma_t} \coloneqq \{\tau \in \Sigma_h \ | \ \tau \cdot \nu=0 \ \textup{on} \  \Gamma_t \}.$$ Of course, the mixed boundary conditions lead also to a small change on the right-hand side of \eqref{weak_form_wek_symmetry_contiuous}. For smooth enough $u_D$ one considers the integral only over $\Gamma_D$ instead of $\langle u_D,  \tau \cdot \nu \rangle_{\Gamma}$. Finally, we remark that in the discrete setting we assume for reasons of simplification that $\Gamma_t$ and $\Gamma_D$ are the union of full mesh boundary elements.

\subsubsection{Multi-patch geometries}
If one considers more complicated geometries, it might be reasonable to describe the overall domain $\o$ by means of a collection of  patches $\o_k, \ k=1, \dots,N, \ N \in \mathbb{N} $. To be more precise, if we are in the field of multi-patch IGA, we assume $$ \overline{\o} = \bigcup_{k=1}^N\overline{ \o_k} \ \  \ \ , \ \p{F}_k \colon [0,1]^n \rightarrow \overline{\o_k }\ \ ,$$
for a suitable number $N$. On the $k$-th patch we can define the discrete spaces $\Sigma_h^{(k)},U_h^{(k)}, \  $ etc. as described in Sec. \ref{Sec:Chioce of spces}.  The global spaces are then straightforwardly defined as $$\Sigma_h \coloneqq \{\sigma \in H(\o,\d,\mathbb{M}) \ | \ \sigma_{|\o_k} \in \Sigma_h^{(k)} \}, \ U_h = \{u \in L^2(\o,\mathbb{R}^n) \ | u_{|\o_k} \in U_h^{(k)} \}, \ \ \textup{etc.} \ \ \ \ \ .$$
The implementation of the coupling conditions for the stress spaces can then be achieved rather easily by exploiting the divergence preserving property of the transformation $\mathcal{Y}_2$. One further notes the fact that also in the multi-patch setting $\nabla\cdot \Sigma_h \subset U_h$. Numerical tests suggest the well-posedness of the mixed formulation for multi-patch parametrizations, although we can not give a proof. Besides,  a combination of mixed boundary conditions and multi-patch IGA is  possible.

\subsubsection{Stability in the (nearly) incompressible regime}

Within the scope of an isotropic and homogeneous elastic material we already mentioned  the two Lam\'{e} coefficients $\lambda,\mu$ determining the compliance tensor $A$. A special case  are (nearly) incompressible materials that correspond to large $\lambda$ values, $\lambda \rightarrow \infty$, respectively. An important motivation to consider mixed formulations is to obtain reasonable approximate solutions also for this regime. Hence, we want to show here that we can achieve stability for $\lambda \rightarrow \infty$ with our discretization ansatz, at least in case of B-spline parametrizations.
To show the latter, we follow the argumentation of Arnold et al. in \cite{Arnold2015}  and we  use  \cite{Braess,Schoeberl}. Furthermore, we restrict ourselves mainly to the three-dimensional case, since the argumentation for $2D$ is in principle the same.  \\
In Theorem \ref{Theorem:convergence estimates}, we see an error estimate that depends on some constant $C_{conv}$. This constant in turn depends on various items,  namely on $ \o, \p{F}, \p{J}$, etc. But, to obtain a stable convergence behavior w.r.t. $\lambda$ we have to  check if we can define $C_{conv}$ independent of $\lambda$. Since in the Brezzi stability requirements in Def. \ref{Brezzi:inf-sup-condition_abstract} we only get a $\lambda$-dependence for (S1) it is enough to consider the coercivity estimate constant $C_{S1}$ and the continuity constant $C_{\mathfrak{A}}$ of the bilinear form $\langle A \cdot , \cdot \rangle $ to see the stability in the incompressible limit. Due to
\begin{align*}
\langle A \sigma , \tau \rangle \leq  \frac{1}{\mu} \n{\sigma}_{L^2} \n{\tau}_{L^2}, \ \  \forall \sigma, \forall\tau \in H(\o,\d,\mathbb{M}) ,
\end{align*}
it remains to look at $C_{S1}$ from (S1). As explained  in \cite{Arnold2015} we get stability for the case of homogeneous Dirichlet boundary condition if $I \in \Sigma_h$, where $I$ stands for the identity tensor. First we have a look at the three-dimensional situation. In view of the definition of $\Sigma_h$ it is easy to see that
$$I \in \Sigma_h  \ \ \Leftrightarrow \ \ \textup{adj}(\p{J})^k \in \big(S_{\textsf{p}+1,\textsf{p},\textsf{p}}^{r,r-1,r-1} \times S_{\textsf{p},\textsf{p}+1,\textsf{p}}^{r-1,r,r-1} \times S_{\textsf{p},\textsf{p},\textsf{p}+1}^{r-1,r-1,r}\big)^T, \forall k, \ $$
where $\textup{adj}(\p{J})^k$ is the $k$-th column of the adjugate matrix $\textup{adj}(\p{J}) = \textup{det}(\p{J}) \p{J}^{-1}$  of $\p{J}$. By the definition of the adjugate matrix as the transpose of the cofactor matrix, it is easy to see that for $\p{F} \in (S_{q,q,q}^{s,s,s})^3, \p{F}=(F_1,F_2,F_3)^T$ with $q> s \geq 1$ we obtain
$$\textup{adj}(\p{J})^{k} \in  \big(S_{2q,2q-1,2q-1}^{s,s-1,s-1} \times S_{2q-1,2q,2q-1}^{s-1,s,s-1} \times S_{2q-1,2q-1,2q}^{s-1,s-1,s} \big).$$ For example it is $$\textup{adj}(\p{J})_{11} = J_{22}J_{33}-J_{23}J_{32} = 
\underbrace{\h_2F_{2}}_{\in S_{q,q-1,q}^{s,s-1,s}}\ \underbrace{\h_3F_{3}}_{\in S_{q,q,q-1}^{s,s,s-1}} - \underbrace{\h_3F_{2}}_{\in S_{q,q,q-1}^{s,s,s-1}} \ \underbrace{\h_2F_{3}}_{\in S_{q,q-1,q}^{s,s-1,s}} \in S_{2q,2q-1,2q-1}^{s,s-1,s-1}.$$ Similar thoughts lead to the other entries. Thus $ \mathcal{Y}^{-1}_2(\textup{adj}(\p{J})^k) \in V_{h,3}^2(2q,s)$ for $k=1,2,3$. Consequently, if we choose the discrete spaces such that $2q \leq \textsf{p}+1$ and $r \leq s$, we have indeed $I \in \Sigma_h$.  In other words, if one wants stability for large $\lambda$ it might be advisable    to apply a  $\textsf{p}$-refinement step. \\
In the $2D$ case the condition $I \in \Sigma_h$ is  easier to see since then $$\textup{adj}(\p{J}) = \begin{bmatrix}
\h_{2}F_2 & -\h_2F_1 \\
-\h_1F_2 & \h_1F_1
\end{bmatrix}.$$ For  $\p{F} \in (S_{q,q}^{s,s})^2$  one has $I \in \Sigma_h$ if $\textsf{p} \geq q, \ s \geq r$. In particular, the isogeometric paradigm fits to the stability condition.\\
The situation for non-zero $u_D$ can be handled by first solving the original problem on the space $\Sigma_h^{tr} \coloneqq \Sigma_h \cap \{ \tau \ | \  \int_{\o} \textup{tr}(\tau) dx=0\}$ and then adding  $c I$ to the obtained solution $ \sigma_h$, where the constant $c$ is determined by the first line of the mixed formulation, by the equation $$\langle A (\sigma_h + c I),  I \rangle  = \langle u_D, I \cdot \nu \rangle_{\Gamma},$$
respectively.\\
It is natural to ask, whether a generalization to NURBS parametrizations with non-trivial weight functions is possible. Although numerical experiments indicate a generalization possibility, we do not have a proof available.\\
Another point would be the incompressibility stability for mixed boundary conditions which we address next. 
Below we follow the steps in \cite{Braess} and   \cite{Schoeberl}. Let $\sigma_h \in \Sigma_{h,\Gamma_t}$ arbitrary but fixed with $ \emptyset \neq \Gamma_t \subsetneq \partial \o$. Let $\Gamma_D \coloneqq \partial \o \backslash \Gamma_t$. Then, e.g. by the results in \cite{Pauly_2022_1}, there exists a $w \in H^1_{\Gamma_D}(\o,\mathbb{R}^n)$ s.t. $$\nabla \cdot w = \textup{tr}(\sigma_h), \ \ \textup{and} \ \ \n{w}_{H^1} \leq C \n{\textup{tr}(\sigma_h)}_{L^2},$$ where $C$ is a constant depending on $\o, \Gamma_t$ and $\Gamma_D$.
With $\textup{dev}(\sigma_h)$ denoting the deviatoric part\footnote{We have the relation  $\sigma = \frac{1}{n}\textup{tr}(\sigma)I+\textup{dev}(\sigma)$.} of $\sigma_h$, we   obtain 
\begin{align*}
\frac{1}{n}\n{\textup{tr}(\sigma_h)}^2_{L^2} &=  \frac{1}{n}\langle \textup{tr}(\sigma_h) , \nabla \cdot w \rangle = \langle \sigma_h-\textup{dev}(\sigma_h) , \nabla w \rangle \\
&= -\langle \nabla \cdot \sigma_h ,  w \rangle - \langle \textup{dev}(\sigma_h) , \nabla w \rangle \\
& \leq \n{\nabla \cdot \sigma_h}_{L^2} \n{w}_{L^2} + \n{\textup{dev}(\sigma_h)}_{L^2} \n{w}_{H^1}\\
& \leq C \big(\n{\nabla \cdot \sigma_h}_{L^2} \n{w}_{L^2} + \n{\textup{dev}(\sigma_h)}_{L^2} \n{\textup{tr}(\sigma_h)}_{L^2} \big).
\end{align*}
In view of the condition (S1) we can assume now that $\sigma_h \in kB_h \coloneqq  \{ \f{\tau} \in \Sigma_h \ | \  \langle  \nabla \cdot  \f{\tau}_h, \f{v}_h \rangle + \langle  \textup{Skew} \, \f{\tau}_h, \f{q}_h \rangle  =0,  \ \forall (\f{v}_h , \f{q}_h) \in {U}_h \times P_h \}$, which implies $\nabla \cdot \sigma_h = 0$. Hence we get with the estimate from above $$\n{\textup{tr}(\sigma_h)}_{L^2}  \leq C \n{\textup{dev}(\sigma_h)}_{L^2},$$
with a $C$ independent of $\lambda$.
With Remark 5.3 (Chapter VI.) in \cite{Braess} we finally obtain
\begin{align*}
\langle A \sigma_h , \sigma_h \rangle &\geq \frac{1}{2\mu} \n{\textup{dev}(\sigma_h)}^2_{L^2} \geq  \frac{1}{4\mu} \n{\textup{dev}(\sigma_h)}^2_{L^2}  + \frac{1}{4\mu} \n{\textup{dev}(\sigma_h)}^2_{L^2}    \\ &\geq \frac{1}{4\mu}\n{\textup{dev}(\sigma_h)}^2_{L^2}  + \frac{1}{4C^2\mu} \n{\textup{tr}(\sigma_h)}^2_{L^2}  \geq \tilde{C} \Big(\n{\textup{dev}(\sigma_h)}_{L^2}  + \n{\textup{tr}(\sigma_h)}_{L^2}  \Big)^2 \geq \tilde{C} \n{\sigma_h}^2_{L^2} \\
& = \tilde{C} \n{\sigma_h}_{\d}^2 ,
\end{align*}
with $\tilde{C} \coloneqq \min\{ \frac{1}{8\mu}, \ \frac{1}{8C^2\mu}\}.$ In particular, there is a estimate for the coercivity condition (S1) which does not dependent on $\lambda$. This suggests a stable convergence behavior in the incompressible regime also for traction boundary conditions.

\subsection{Numerical examples}
\label{Sec:Numerical_examples}
In this section we discuss the discrete method for linear elasticity introduced in the previous sections by means of several numerical tests. We look at the two-dimensional as well as at the $3D$ setting. For the computations and implementation of the method we used MATLAB together with the GeoPDEs package \cite{geopdes,geopdes3.0,MATLAB:2020}.  Below, we say the mesh has mesh size $h$ if there are $1/h$ many subdivisions in each coordinate direction, also in each patch in case of a multi-patch example. Due to the smoothness of the applied parametrizations, this notion of a mesh size is equivalent to the definition in Sec. \ref{section:splines}.

\subsubsection{Deformed square}
In the first example we consider the convergence behavior for a two-dimensional domain with a curved boundary which has the smooth parametrization $\p{F}(\zeta_1,\zeta_2)= (\zeta_1,\zeta_2-\zeta_1^2+\zeta_1)$; see Fig. \ref{Fig:Num_2} (a). Furthermore we assume homogeneous boundary conditions $u_D = 0 $ and choose the source function $f$ in the mixed formulation s.t. the exact displacement solution is $u_1 = g, u_2=-g$ with $$\boxed{g \circ \p{F}(\zeta_1,\zeta_2) = \sin(\pi \zeta_1) \sin(\pi \zeta_2)}.$$ Here we use the Lam\'{e} coefficients $\lambda=2,\mu=1$.   Due to this regular problem and the convergence statement in Theorem  \ref{Theorem:convergence estimates}, we can expect for degree $\textsf{p} \geq 2$ that the errors $\n{\sigma-\sigma_h}_{\d}, \ \n{u-u_h}_{L^2},$ $ \ \n{p-p_h}_{L^2}$ are of order $\mathcal{O}(h^\textsf{p})$. And indeed, if we look at the errors in Fig. \ref{Fig:Num_1} one observes the mentioned convergence order. In the subsequent Fig. \ref{Fig:Num_2} (b)-(c) we further show exemplarily the approximate solution for the $x$-displacement and the first stress component $\sigma_{11}$ for $\textsf{p}=6, \ r=4$ and $h=1/3$. Latter two plots illustrate that we obtain reasonable results also for higher polynomial degrees combined with high regularity  B-splines ($r=4$). This usage of globally smooth basis functions is in general hard for classical FEM approaches. The $2D$ geometry from above is utilized further to show  more aspects of our method. On the one hand, we  verify the stability for large $\lambda$ and, on the other hand the generalization to the multi-patch framework  is exemplified. For the former we apply the stability test from \cite{Rettung} at the extreme case $\lambda = \infty, \mu = 1$. This leads again to a regular source function implying the exact solution $$u_1 = (\cos(2\pi x)-1)\sin(2\pi y), \ \ u_2 = (1-\cos(2\pi y))\sin(2\pi x).$$ In this situation we apply mixed boundary conditions, where the displacement boundary  part $\Gamma_D$ is just the left edge ($x=0$) of the computational domain and on the rest of $\partial \o$ we have traction forces, cf. Fig \ref{Fig:Num_2} (a). If we then consider the resulting  errors in Fig. \ref{Fig:Num_3} we note a stable decrease that fits to the theoretical approximations.  In particular we do not see a blowing up in the errors.  

Next, in view of Fig. \ref{Fig:Num_7} (a) we divide now $\o$ in nine patches $\o_1, \dots , \o_9$ and repeat the first convergence experiment in the compressible regime, namely $\lambda=2, \ \mu=1$. Fig. \ref{Fig:Num_6} demonstrates the approximation properties in the multi-patch setting. One notes that the underlying parametrization is continuous, but  not $C^1$.   

\begin{figure}
	\begin{minipage}{0.33\textwidth}
		\begin{tikzpicture}[scale=1]
		\begin{loglogaxis} [
		        title={(a) \hspace*{0.5cm} $\n{\sigma-\sigma_h}_{\d}$ \hspace*{0.7cm}},
		width = 5.5cm,
		height=5.5cm,
		xlabel={Mesh size h},
		xmin=0.038, xmax=0.6,
		ymin=10^-5, ymax=0.9*10^1,
		xtick={1/20,1/10,0.25,0.5},
		log x ticks with fixed point,
		ytick={10^-5,10^-4,10^-3,10^-2,10^-1,10^0},
		legend pos=south east,
		xmajorgrids=true,
		ymajorgrids=true,
		grid style=dashed,
		legend columns=1,
		]
		\addplot
		coordinates {
			(1/2,  8.4378696 )
			(1/4, 	 2.3346271)
			(1/6, 1.0502569 )
			(1/8, 0.593162)
			(1/10, 	0.38032679  )
			(1/12, 	0.26438118)
			(1/14, 	0.19435722)
			(1/16, 	0.14886367)
			(1/18, 	 0.11765278)
			(1/20, 	0.095317469 )
		};
		\addlegendentry{\textsf{p}=2},
		\addplot
		coordinates {
			(1/2, 1.6638296 )
			(1/4, 	0.20239057 )
			(1/6, 0.060097256)
			(1/8, 0.025381002)
			(1/10, 	0.013002308  )
			(1/12, 	0.0075268762)
			(1/14, 	0.0047408935)
			(1/16, 	 0.0031764443)
			(1/18, 	0.002231122)
			(1/20, 	0.0016265956)
		};
		\addlegendentry{\textsf{p}=3},
		\addplot[color=black,mark=triangle*,
		mark options={solid, fill=black,scale=1.3}]
		coordinates {
			(1/2, 0.15517396)
			(1/4, 	 0.011775443)
			(1/6, 0.0023740182)
			(1/8, 0.00075613872)
			(1/10, 	0.00031063999)
			(1/12, 	0.00015004754)
			(1/14, 	8.1069907e-05)
			(1/16, 	4.7551326e-05)
			(1/18, 	2.9698744e-05)
			(1/20, 1.9491304e-05)
		};,
		\addlegendentry{\textsf{p}=4},
		\addplot[color=blue, mark size=1.8pt,mark=none, dashed]
		coordinates {
			(1/2, 22.6*0.5^2)
			(1/20, 22.6*0.05^2)
		};
		\addplot[color=red, mark size=1.8pt,mark=none, dashed]
		coordinates {
			(1/2, 6.4*0.5^3)
			(1/20, 6.4*0.05^3)};
		\addplot[color=black, mark size=1.8pt,mark=none, dashed]
		coordinates {
			(1/2, 6.1*0.5^4)
			(1/22, 6.1/22/22/22/22)};
		\end{loglogaxis}
		\end{tikzpicture}
	\end{minipage}
	\hspace*{-0.45cm}
	\begin{minipage}{0.33\textwidth}
		\begin{tikzpicture}[scale=1]
		\begin{loglogaxis} [
	  title={(b) \hspace*{0.5cm} $\n{u-u_h}_{L^2}$ \hspace*{0.7cm}},
		width = 5.5cm,
		height=5.5cm,
		xlabel={Mesh size h},
		xmin=0.038, xmax=0.6,
		ymin=0.7*10^-7, ymax=1.6*10^-1,
		xtick={1/20,1/10,0.25,0.5},
	log x ticks with fixed point,
	ytick={10^-7,10^-6,10^-5,10^-4,10^-3,10^-2},
	legend pos=south east,
		legend pos=south east,
		xmajorgrids=true,
		ymajorgrids=true,
		grid style=dashed,
		legend columns=1,
		]
			\addplot
	coordinates {
		(1/2,  0.093756549 )
		(1/4, 	 0.023379552)
		(1/6, 0.010308744 )
		(1/8, 0.0057779089)
		(1/10, 	0.0036911881  )
		(1/12, 	0.002560695)
		(1/14, 	0.0018801287)
		(1/16,  0.0014388664)
		(1/18, 	0.0011365483)
		(1/20, 	0.00092040857 )
	};
		\addlegendentry{\textsf{p}=2},
			\addplot
	coordinates {
		(1/2, 0.012713267 )
		(1/4, 	0.0015462286 )
		(1/6,  0.0004550646)
		(1/8, 0.00019148245)
		(1/10, 	9.7912176e-05  )
		(1/12, 	5.6620262e-05)
		(1/14, 	3.5639454e-05)
		(1/16, 	 2.3868316e-05)
		(1/18, 1.6759883e-05)
		(1/20, 	1.2216051e-05)
	};
		\addplot[color=black,mark=triangle*,
	mark options={solid, fill=black,scale=1.3}]
	coordinates {
	(1/2,  0.0012078353)
	(1/4, 	 7.5927791e-05)
	(1/6, 1.4968406e-05)
	(1/8, 4.7296384e-06)
	(1/10, 	1.9357876e-06)
	(1/12, 	9.3311951e-07)
	(1/14, 	5.035312e-07)
	(1/16, 	2.9510446e-07)
	(1/18, 	1.8420765e-07)
	(1/20, 1.2084692e-07)
};,
		\addplot[color=blue, mark size=1.8pt,mark=none, dashed]
	coordinates {
		(1/2, 0.22*0.5^2)
		(1/20, 0.22*0.05^2)
	};
	\addplot[color=red, mark size=1.8pt,mark=none, dashed]
	coordinates {
		(1/2, 0.045*0.5^3)
		(1/20, 0.045*0.05^3)};
	\addplot[color=black, mark size=1.8pt,mark=none, dashed]
	coordinates {
		(1/2, 0.04*0.5^4)
		(1/22, 0.04/22/22/22/22)};
		\legend{,,,\small $\mathcal{O}(h^2)$,\small $\mathcal{O}(h^3)$,\small $\mathcal{O}(h^4)$}	
		\end{loglogaxis}
		\end{tikzpicture}
	\end{minipage}
	\hspace*{-0.45cm}
	\begin{minipage}{0.33\textwidth}
		\begin{tikzpicture}[scale=1]
		\begin{loglogaxis} [
	  title={(c) \hspace*{0.5cm} $\n{p-p_h}_{L^2}$ \hspace*{0.7cm} },
		width = 5.5cm,
		height=5.5cm,
		xlabel={Mesh size h},
		xmin=0.038, xmax=0.6,
		ymin=1.7*10^-7, ymax=5*10^-1,
		xtick={1/20,1/10,0.25,0.5},
		log x ticks with fixed point,
		ytick={10^-7,10^-6,10^-5,10^-4,10^-3,10^-2},
		legend pos=south east,
		xmajorgrids=true,
		ymajorgrids=true,
		grid style=dashed,
		legend columns=1,
		]
		\addplot
		coordinates {
			(1/2, 0.24554224)
			(1/4, 	 0.052132683)
			(1/6, 0.021880047)
			(1/8,0.012032821)
			(1/10, 	0.0076185294 )
			(1/12, 0.0052595331)
			(1/14, 	0.0038504145)
			(1/16,  0.002941181)
			(1/18, 	0.0023202332)
			(1/20, 0.0018772758 )
		};
		\addplot
		coordinates {
			(1/2,0.034989007 )
			(1/4, 	0.0050593522 )
			(1/6,  0.0015850231)
			(1/8, 0.00068644401)
			(1/10, 	0.00035664023  )
			(1/12, 0.00020828654)
			(1/14, 	0.00013198087)
			(1/16, 	8.8810364e-05)
			(1/18, 6.2582082e-05)
			(1/20, 4.5739967e-05)
		};	
		\addplot[color=black,mark=triangle*,
		mark options={solid, fill=black,scale=1.3}]
		coordinates {
			(1/2,  0.0041171369)
			(1/4, 	0.00022868278)
			(1/6,4.3320474e-05)
			(1/8, 1.3451217e-05)
			(1/10, 5.4561696e-06)
			(1/12, 2.6163484e-06)
			(1/14, 	1.4071509e-06)
			(1/16, 	8.2283099e-07)
			(1/18, 5.127959e-07)
			(1/20, 3.360124e-07)
		};,
		\addplot[color=blue, mark size=1.8pt,mark=none, dashed]
		coordinates {
			(1/2, 0.48*0.5^2)
			(1/20, 0.48*0.05^2)
		};
		\addplot[color=red, mark size=1.8pt,mark=none, dashed]
		coordinates {
			(1/2, 0.182*0.5^3)
			(1/20, 0.182*0.05^3)};
		\addplot[color=black, mark size=1.8pt,mark=none, dashed]
		coordinates {
			(1/2, 0.1*0.5^4)
			(1/22, 0.1/22/22/22/22)};
		\end{loglogaxis}
		\end{tikzpicture}
	\end{minipage}
\caption{Deformed square: Here we plot the different errors w.r.t. $h$ and for the degrees $\textsf{p}=2,3,4$. For all cases we set $r=0$. Dashed lines with slopes $2,3,4$ are added to verify the convergence rates.}
\label{Fig:Num_1}
\end{figure}
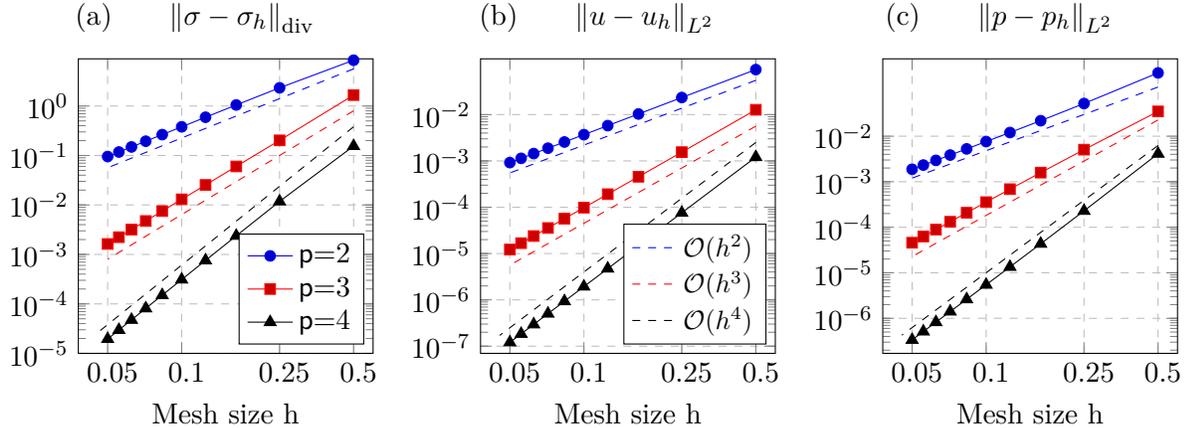

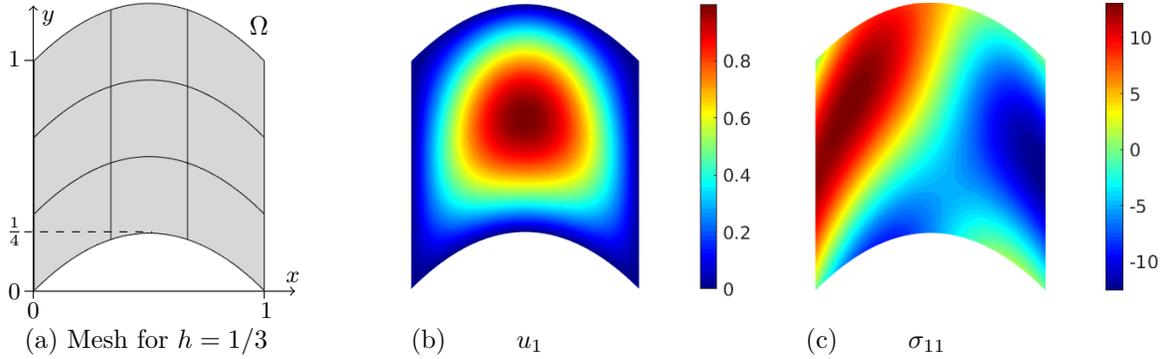
\begin{figure}[h!]
	\centering\begin{tikzpicture}
	\node (eins) at (-7.51,0) {\includegraphics[width=0.21\linewidth]{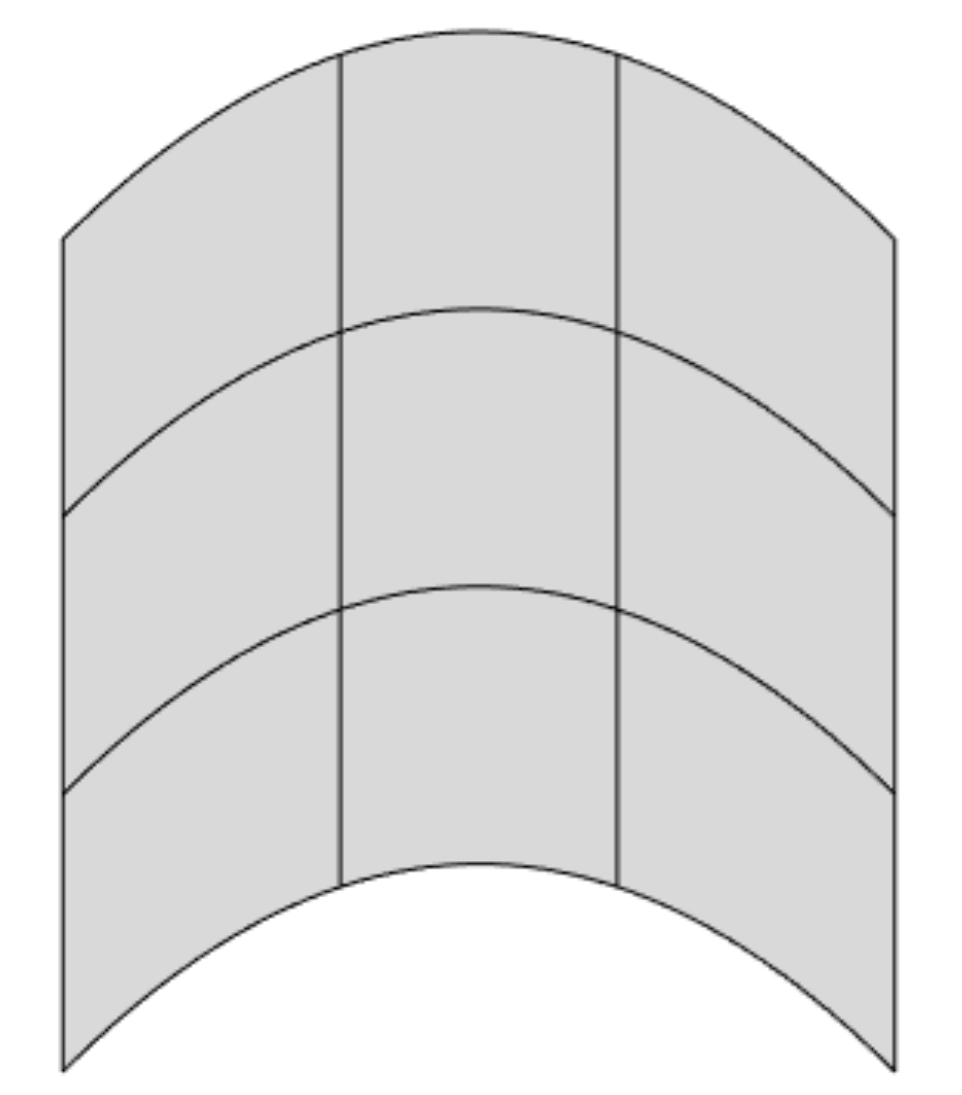}};
	\draw[->,shift={(-0.07,0.06)}] (-9.1,-2) -- (-5.5,-2);
		\draw[shift={(-0.07,-0.07)}] (-9.1,1.2) -- (-8.98,1.2);
	\draw[->,shift={(0.05,0.06)}] (-9.1,-2.1) -- (-9.1,1.8);
	\node[shift={(0,0)}] at (-9.3,1.15) {\footnotesize $1$};
		\node[shift={(0,0.5)}] at (-8.85,1.2) {\footnotesize $y$};
		\node[shift={(0,0)}] at (-9.3,-1.94) {\footnotesize $0$};
		\node[shift={(0,0.35)}] at (-9.3,-1.5) {\footnotesize $\frac{1}{4}$};
		\node[shift={(-0.6,0.15)}] at (-5,-1.9) {\footnotesize $x$};
		\node[shift={(0,0)}] at (-9.05,-2.2) {\footnotesize $0$};
		\node[shift={(0,0)}] at (-5.95,-2.2) {\footnotesize $1$};
			\draw[dashed,shift={(-0.07,0.35)}] (-9.1,-1.5) -- (-7.4,-1.5);		
			\node[shift={(2.8,0.4)}] at (-8.85,1.2) {\small $\o$};
		\draw[shift={(0.02,0.06)}] (-6,-2.1) -- (-6,-2);
	\node at (-7.55,-2.6) {\small (a) Mesh for $h=1/3$};
		\node at (-3.15,-2.6) {\small (b) \hspace{0.7cm} $u_1$};
		\node at (2.15,-2.6) {\small (c) \hspace{0.7cm} $\sigma_{11}$};
	\node[shift={(0,-0.08)}] (eins) at (-6.84+0.09+5,0)	{\includegraphics[width=4.55cm]{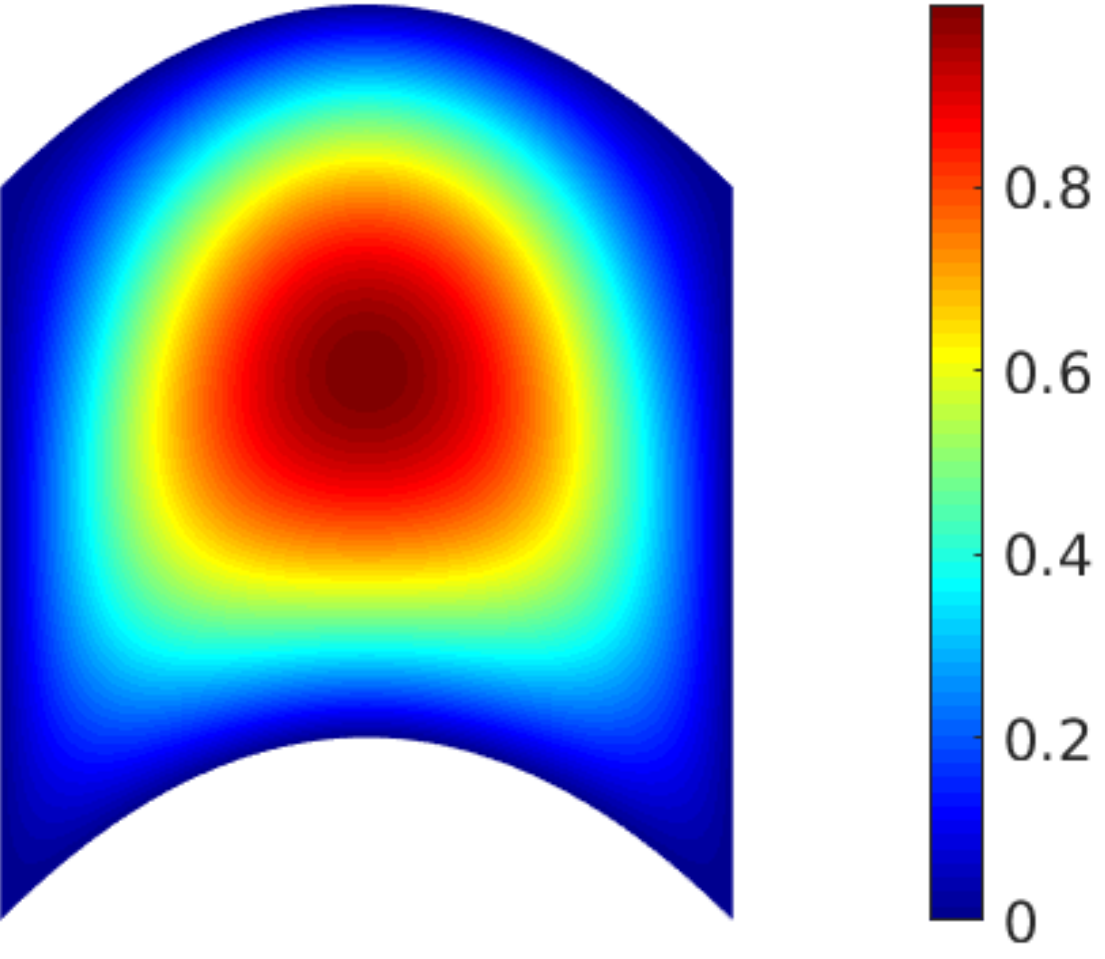}};
	
	\node[shift={(0,-0.033)}] (eins) at (-6.88+10.5,0)	{\includegraphics[width=4.55cm]{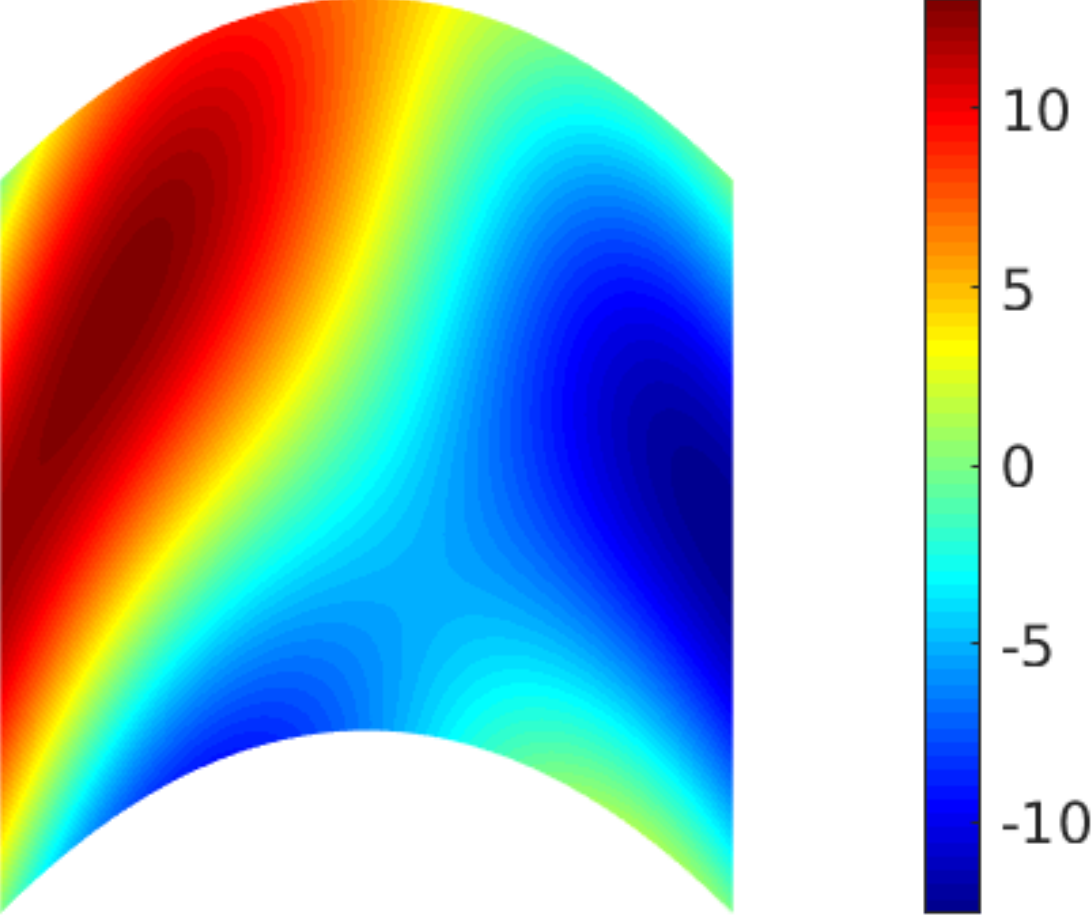}};
	\end{tikzpicture}
	\caption{On the left we see the geometry for the first test example. In the middle and on the right we see displacement and stress approximate solutions with degree $\textsf{p}=6$ and $r=4$ for the coarse mesh on the left.  Within IGA we can implement easily basis functions with high degree and high regularity. }
	\label{Fig:Num_2}
\end{figure}

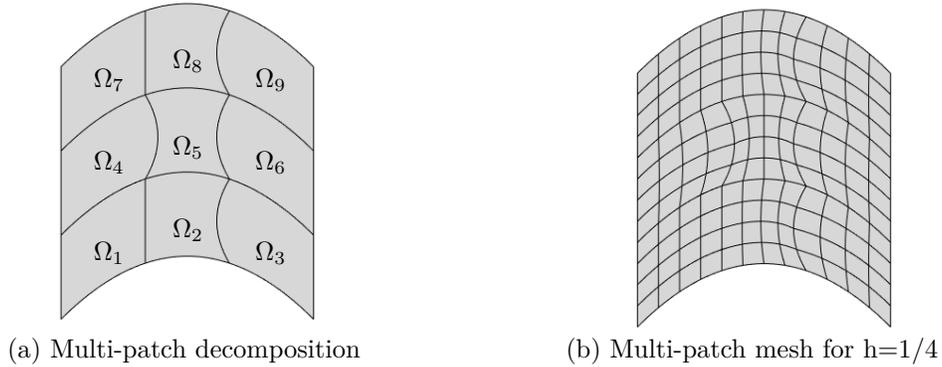
\begin{figure}[h!]
	\centering\begin{tikzpicture}
		\node (eins) at (-7.5,0) {\includegraphics[width=0.21\linewidth]{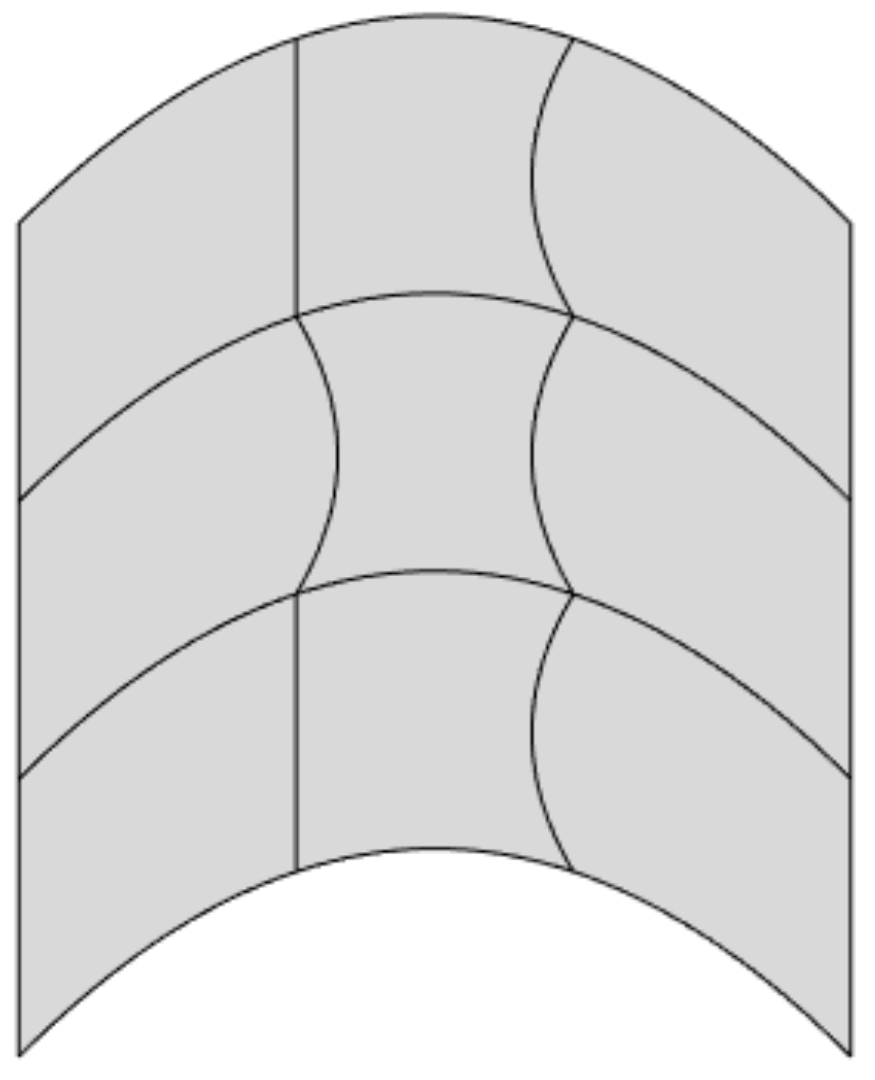}};
		\node at (-7.55,-2.5) {\small (a) Multi-patch decomposition};
		\node at (-0,-2.5) {\small (b) Multi-patch mesh for h=1/4};
		\node[shift={(0,-0.08)}] (eins) at (-6.84+7,0)	{\includegraphics[width=3.55cm]{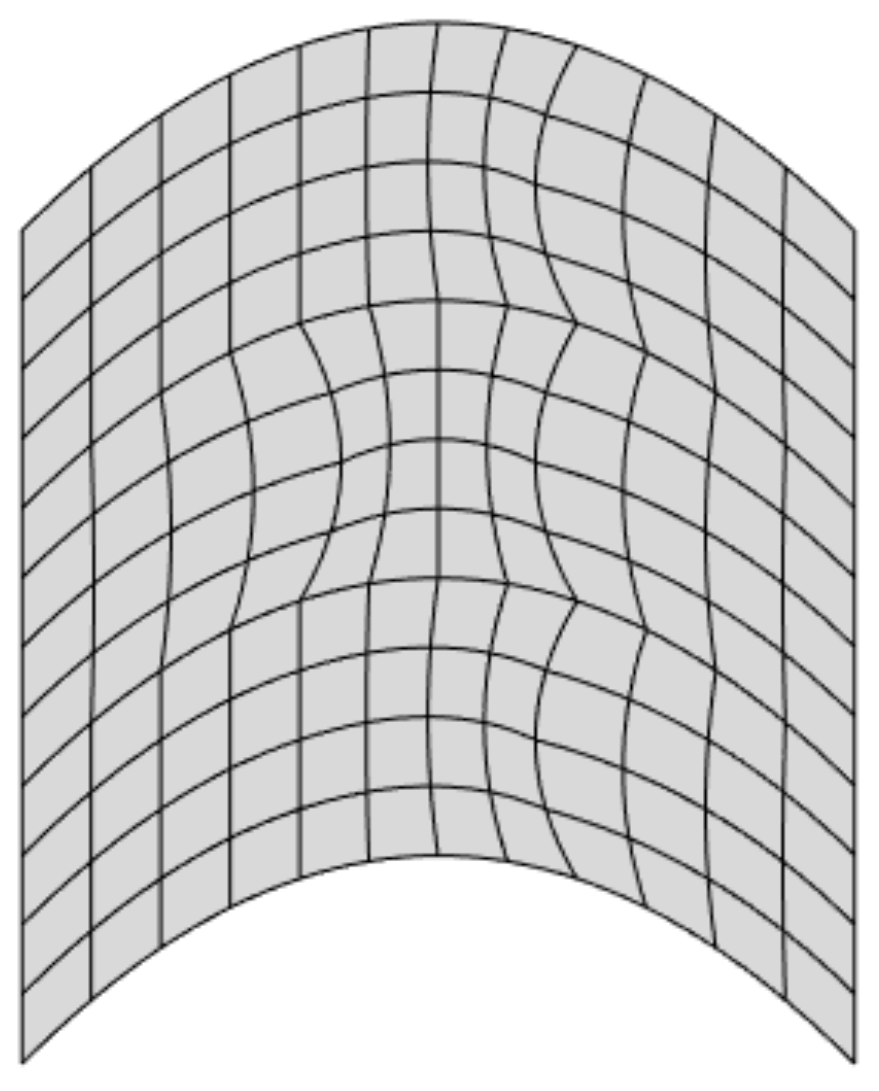}};
		\node at (-8.55,-1.2) {\small $\o_1$};
		\node at (-7.5,-0.9) {\small $\o_2$};
		\node at (-6.4,-1.2) {\small $\o_3$};
		\node at (-6.4,-0) {\small $\o_6$};
		\node at (-6.4,1.1) {\small $\o_9$};
		\node at (-7.5,0.2) {\small $\o_5$};
		\node at (-7.5,1.35) {\small $\o_8$};
		\node at (-8.55,0) {\small $\o_4$};
		\node at (-8.55,1.1) {\small $\o_7$};
	\end{tikzpicture}
	\caption{The multi-patch decomposition of the  first geometry example is pictured on the left. If we set $h=1/4$, i.e. 4 subdivision in each direction and in each patch, we get the mesh on the right.}
	\label{Fig:Num_7}
\end{figure}

\begin{figure}
	\begin{minipage}{0.33\textwidth}
		\begin{tikzpicture}[scale=1]
			\begin{loglogaxis} [
				title={(a) \hspace*{0.5cm} $\n{\sigma-\sigma_h}_{\d}$ \hspace*{0.7cm}},
				width = 5.5cm,
				height=5.5cm,
				xlabel={Mesh size h},
				xmin=0.038, xmax=0.6,
				ymin=3*10^-4, ymax=6*10^1,
				xtick={1/20,1/10,0.25,0.5},
				log x ticks with fixed point,
				ytick={10^-4,10^-3,10^-2,10^-1,10^0,10^1},
				legend pos=south east,
				xmajorgrids=true,
				ymajorgrids=true,
				grid style=dashed,
				legend columns=1,
				]
				\addplot
				coordinates {
					(1/2,  31.206416 )
					(1/4, 	 10.011948)
					(1/6, 4.6524481 )
					(1/8, 2.6575762)
					(1/10, 	1.7129303 )
					(1/12, 1.1940965)
					(1/14, 0.8793162)
					(1/16, 	0.67423039)
					(1/18, 	 0.53326887)
					(1/20, 0.43226262 )
				};
				\addlegendentry{\textsf{p}=2},
				\addplot
				coordinates {
					(1/2, 12.141817 )
					(1/4, 2.0140054 )
					(1/6, 0.63689488)
					(1/8, 0.27543609)
					(1/10,  0.14275764 )
					(1/12, 	0.083195152)
					(1/14, 	0.052623581)
					(1/16, 	 0.035359467)
					(1/18, 	0.0248871)
					(1/20, 	0.018171351)
				};
			
			'    3.9018987'
			'   0.35730121'
			'  0.076614515'
			'  0.024957971'
			'  0.010361794'
			' 0.0050337386'
			' 0.0027290851'
			' 0.0016043115'
			' 0.0010035218'
			'0.00065932964'
			
				\addlegendentry{\textsf{p}=3},
				\addplot[color=black,mark=triangle*,
				mark options={solid, fill=black,scale=1.3}]
				coordinates {
					(1/2,  3.9018987)
					(1/4, 	 0.35730121)
					(1/6, 0.076614515)
					(1/8, 0.024957971)
					(1/10, 	0.010361794)
					(1/12, 	0.0050337386)
					(1/14, 	0.0027290851)
					(1/16, 	0.0016043115)
					(1/18, 	0.0010035218)
					(1/20, 0.00065932964)
				};,
				\addlegendentry{\textsf{p}=4},
				\addplot[color=blue, mark size=1.8pt,mark=none, dashed]
				coordinates {
					(1/2, 102.6*0.5^2)
					(1/20, 102.6*0.05^2)
				};
				\addplot[color=red, mark size=1.8pt,mark=none, dashed]
				coordinates {
					(1/2, 70.4*0.5^3)
					(1/20, 70.4*0.05^3)};
				\addplot[color=black, mark size=1.8pt,mark=none, dashed]
				coordinates {
					(1/2, 60.1*0.5^4)
					(1/20, 60.1/20/20/20/20)};
			\end{loglogaxis}
		\end{tikzpicture}
	\end{minipage}
	\hspace*{-0.45cm}
	\begin{minipage}{0.33\textwidth}
		\begin{tikzpicture}[scale=1]
			\begin{loglogaxis} [
				title={(b) \hspace*{0.5cm} $\n{u-u_h}_{L^2}$ \hspace*{0.7cm}},
				width = 5.5cm,
				height=5.5cm,
				xlabel={Mesh size h},
			xmin=0.038, xmax=0.6,
			ymin=0.3*10^-5, ymax=1.6*10^0,
			xtick={1/20,1/10,0.25,0.5},
			log x ticks with fixed point,
			ytick={10^-7,10^-6,10^-5,10^-4,10^-3,10^-2,10^-1,10^0},
				legend pos=south east,
				xmajorgrids=true,
				ymajorgrids=true,
				grid style=dashed,
				legend columns=1,
				]
	      \addplot
			coordinates {
				(1/2,  0.54644753 )
				(1/4, 	 0.15357128)
				(1/6, 0.066994935 )
				(1/8, 0.037679593)
				(1/10,  0.024165792  )
				(1/12, 	 0.016811321)
				(1/14, 	0.012366772)
				(1/16,  0.009476804)
				(1/18, 	0.0074927025)
				(1/20, 	0.0060720032 )
			};
				\addlegendentry{\textsf{p}=2},
		\addplot
			coordinates {
				(1/2,  0.17417774)
				(1/4, 	 0.026282882 )
				(1/6,  0.0081933844)
				(1/8, 0.0035223289)
				(1/10, 	0.0018191277  )
				(1/12, 	0.0010575798)
				(1/14, 0.00066778847)
				(1/16, 	 0.00044812287)
				(1/18, 0.00031508556)
				(1/20, 	0.00022987763)
			};
					\addplot[color=black,mark=triangle*,
			mark options={solid, fill=black,scale=1.3}]
			coordinates {
				(1/2,  0.050325728)
				(1/4, 	0.0044077673)
				(1/6, 0.00093905052)
				(1/8, 0.00030535296)
				(1/10, 	0.00012670264)
				(1/12, 	6.1540363e-05)
				(1/14, 	3.3362525e-05)
				(1/16, 	1.9612034e-05)
				(1/18, 	1.2267621e-05)
				(1/20, 8.0600668e-06)
			};,
					\addplot[color=blue, mark size=1.8pt,mark=none, dashed]
				coordinates {
					(1/2, 1.5*0.5^2)
					(1/20, 1.5*0.05^2)
				};
				\addplot[color=red, mark size=1.8pt,mark=none, dashed]
				coordinates {
					(1/2, 0.8*0.5^3)
					(1/20, 0.8*0.05^3)};
				\addplot[color=black, mark size=1.8pt,mark=none, dashed]
				coordinates {
					(1/2, 0.65*0.5^4)
					(1/20, 0.65/20/20/20/20)};
				\legend{,,,\small $\mathcal{O}(h^2)$,\small $\mathcal{O}(h^3)$,\small $\mathcal{O}(h^4)$}	
			\end{loglogaxis}
		\end{tikzpicture}
	\end{minipage}
	\hspace*{-0.45cm}
	\begin{minipage}{0.33\textwidth}
		\begin{tikzpicture}[scale=1]
			\begin{loglogaxis} [
				title={(c) \hspace*{0.5cm} $\n{p-p_h}_{L^2}$ \hspace*{0.7cm} },
				width = 5.5cm,
				height=5.5cm,
				xlabel={Mesh size h},
				xmin=0.038, xmax=0.6,
				ymin=0.9*10^-5, ymax=5*10^0,
				xtick={1/20,1/10,0.25,0.5},
				log x ticks with fixed point,
				ytick={10^-7,10^-6,10^-5,10^-4,10^-3,10^-2,10^-1,10^0},
				legend pos=south east,
				xmajorgrids=true,
				ymajorgrids=true,
				grid style=dashed,
				legend columns=1,
				]								
				\addplot
				coordinates {
					(1/2, 2.6149952 )
					(1/4,0.77520502)
					(1/6,  0.32498624)
					(1/8, 0.17265464)
					(1/10, 0.10664215 )
					(1/12, 	0.07242433)
					(1/14, 	0.052439282)
					(1/16,  0.039750653)
					(1/18, 0.031186393)
					(1/20, 	0.025130116 )
				};									
				\addplot
				coordinates {
					(1/2, 0.70131291)
					(1/4, 	0.12312285 )
					(1/6, 0.043310884)
					(1/8, 0.020272546)
					(1/10, 	0.011096631  )
					(1/12, 0.006718895)
					(1/14, 0.0043693469)
					(1/16, 0.0029974931)
					(1/18, 0.0021437688)
					(1/20, 0.0015851928)
				};						
				\addplot[color=black,mark=triangle*,
				mark options={solid, fill=black,scale=1.3}]
				coordinates {
					(1/2,  0.23012734)
					(1/4,0.022357446)
					(1/6,0.0047695733)
					(1/8,  0.001496635)
					(1/10, 	0.00060335807)
					(1/12, 	0.00028703938)
					(1/14, 	0.00015332843)
					(1/16, 	8.9173906e-05)
					(1/18, 	5.5339073e-05)
					(1/20, 3.6141102e-05)
				};,
				\addplot[color=blue, mark size=1.8pt,mark=none, dashed]
				coordinates {
					(1/2, 6.3*0.5^2)
					(1/20, 6.3*0.05^2)
				};
				\addplot[color=red, mark size=1.8pt,mark=none, dashed]
				coordinates {
					(1/2, 5.8*0.5^3)
					(1/20, 5.8*0.05^3)};
				\addplot[color=black, mark size=1.8pt,mark=none, dashed]
				coordinates {
					(1/2, 3.2*0.5^4)
					(1/20, 3.2/20/20/20/20)};
			\end{loglogaxis}
		\end{tikzpicture}
	\end{minipage}
	\caption{We display the errors for the incompressibility test with geometry Fig. \ref{Fig:Num_2} (a). The results indicate the stability of the method even for $\lambda=\infty$.}
	\label{Fig:Num_3}
\end{figure}
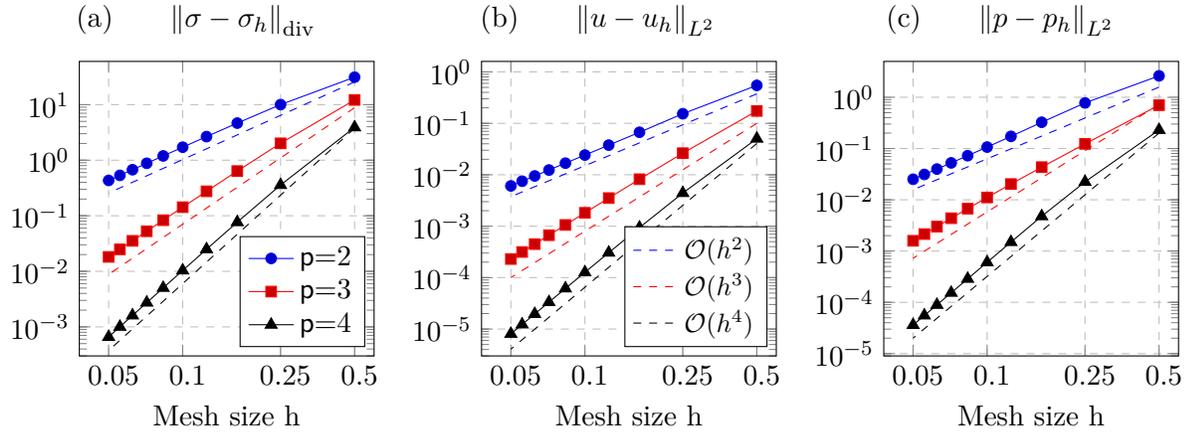

\begin{figure}
	\begin{minipage}{0.33\textwidth}
		\begin{tikzpicture}[scale=1]
		\begin{loglogaxis} [
		title={(a) \hspace*{0.5cm} $\n{\sigma-\sigma_h}_{\d}$ \hspace*{0.7cm}},
		width = 5.5cm,
		height=5.5cm,
		xlabel={Mesh size h},
		xmin=0.1, xmax=1,
		ymin=1.2*10^-5, ymax=1.6*10^1,
		xtick={1/20,1/8,0.25,0.5},
		log x ticks with fixed point,
		ytick={10^-5,10^-4,10^-3,10^-2,10^-1,10^0},
		legend pos=south east,
		xmajorgrids=true,
		ymajorgrids=true,
		grid style=dashed,
		legend columns=1,
		]
		\addplot
		coordinates {
			(1,  5.1142 )
			(1/2, 	  1.3529)
			(1/3, 0.60991 )
			(1/4,  0.34488)
			(1/5, 0.22128 )
			(1/6, 	0.15388)
			(1/7, 	0.11315)
			(1/8, 	0.086674)
		};
		\addlegendentry{\textsf{p}=2},
		\addplot
		coordinates {
			(1, 0.92335 )
			(1/2, 	 0.13147)
			(1/3,  0.040377)
			(1/4,   0.017269)
			(1/5, 0.0088996)
			(1/6, 	0.0051687)
			(1/7,0.003262)
			(1/8, 0.0021884)
		};
		\addlegendentry{\textsf{p}=3},
		\addplot[color=black,mark=triangle*,
		mark options={solid, fill=black,scale=1.3}]
		coordinates {				
			(1,   0.16379 )
			(1/2, 	  0.013183)
			(1/3,0.0027442 )
			(1/4,  0.00088359)
			(1/5, 0.00036479 )
			(1/6, 	0.00017667)
			(1/7, 	9.5604e-05)
			(1/8, 	5.6133e-05)
		};,
		\addlegendentry{\textsf{p}=4},
		\addplot[color=blue, mark size=1.8pt,mark=none, dashed]
		coordinates {
			(1, 3.5*1)
			(1/8, 3.5/8/8)
		};
		\addplot[color=red, mark size=1.8pt,mark=none, dashed]
		coordinates {
			(1, 0.6*1)
			(1/8, 0.6/8/8/8)};
		\addplot[color=black, mark size=1.8pt,mark=none, dashed]
		coordinates {
			(1, 0.119*1)
			(1/8, 0.119/8/8/8/8)};
		\end{loglogaxis}
		\end{tikzpicture}
	\end{minipage}
	\hspace*{-0.45cm}
	\begin{minipage}{0.33\textwidth}
		\begin{tikzpicture}[scale=1]
		\begin{loglogaxis} [				
		title={(b) \hspace*{0.5cm} $\n{u-u_h}_{L^2}$ \hspace*{0.7cm}},
		width = 5.5cm,
		height=5.5cm,
		xlabel={Mesh size h},
	xmin=0.1, xmax=1,
	ymin=2*10^-7, ymax=1.9*10^-1,
	xtick={1/20,1/8,0.25,0.5},
	log x ticks with fixed point,
	ytick={10^-6,10^-5,10^-4,10^-3,10^-2,10^-1,10^0},
		legend pos=south east,
		xmajorgrids=true,
		ymajorgrids=true,
		grid style=dashed,
		legend columns=1,
		]
		\addplot
	coordinates {
		(1,  0.067694 )
		(1/2, 0.017163)
		(1/3, 0.0077154 )
		(1/4,  0.0043578)
		(1/5, 0.0027944)
		(1/6, 0.0019426)
		(1/7, 	0.0014282)
		(1/8, 	0.0010939)
	};
		\addlegendentry{\textsf{p}=2},
			\addplot
		coordinates {
			(1, 0.012561 )
			(1/2,0.0019101)
			(1/3,  0.00058121)
			(1/4,   0.00024747)
			(1/5, 0.00012725)
			(1/6, 	7.3808e-05)
			(1/7,4.6545e-05)
			(1/8,  3.121e-05)
		};
		\addlegendentry{\textsf{p}=3},
			\addplot[color=black,mark=triangle*,
		mark options={solid, fill=black,scale=1.3}]
		coordinates {				
			(1,    0.0027803)
			(1/2, 	  0.00018329)
			(1/3,3.6895e-05 )
			(1/4, 1.1759e-05)
			(1/5, 4.8331e-06 )
			(1/6, 2.3353e-06)
			(1/7, 	1.262e-06)
			(1/8, 	7.4033e-07)
		};,
		\addlegendentry{\textsf{p}=4},
			\addplot[color=blue, mark size=1.8pt,mark=none, dashed]
		coordinates {
			(1, 0.045*1)
			(1/8, 0.045/8/8)
		};
		\addplot[color=red, mark size=1.8pt,mark=none, dashed]
		coordinates {
			(1, 0.0088*1)
			(1/8, 0.0088/8/8/8)};
		\addplot[color=black, mark size=1.8pt,mark=none, dashed]
		coordinates {
			(1, 0.00178*1)
			(1/8, 0.00178/8/8/8/8)};
		\legend{,,,\small $\mathcal{O}(h^2)$,\small $\mathcal{O}(h^3)$,\small $\mathcal{O}(h^4)$}	
		\end{loglogaxis}
		\end{tikzpicture}
	\end{minipage}
	\hspace*{-0.45cm}
	\begin{minipage}{0.33\textwidth}
		\begin{tikzpicture}[scale=1]
		\begin{loglogaxis} [				
		title={(c) \hspace*{0.5cm} $\n{p-p_h}_{L^2}$ \hspace*{0.7cm}},
		width = 5.5cm,
		height=5.5cm,
		xlabel={Mesh size h},
		xmin=0.1, xmax=1,
		ymin=2*10^-7, ymax=1.9*10^-1,
		xtick={1/20,1/8,0.25,0.5},
		log x ticks with fixed point,
		ytick={10^-6,10^-5,10^-4,10^-3,10^-2,10^-1,10^0},
		legend pos=south east,
		xmajorgrids=true,
		ymajorgrids=true,
		grid style=dashed,
		legend columns=1,
		]
		\addplot
		coordinates {
			(1, 0.10979708 )
			(1/2, 0.027121131)
			(1/3, 0.011549124 )
			(1/4,  0.0063692889)
			(1/5, 0.0040290762)
			(1/6, 0.0027774042)
			(1/7, 0.0020301847)
			(1/8, 0.0015486298)
		};
		\addplot
		coordinates {
			(1, 0.015629065 )
			(1/2,0.0023187183)
			(1/3, 0.00073642592)
			(1/4,   0.00032272851)
			(1/5, 0.00016925038)
			(1/6,  9.956608e-05)
			(1/7,6.3448032e-05)
			(1/8, 4.2886197e-05)
		};
		\addplot[color=black,mark=triangle*,
		mark options={solid, fill=black,scale=1.3}]
		coordinates {				
			(1,     0.0017639411)
			(1/2, 	 0.00014784948)
			(1/3,2.8994263e-05)
			(1/4,9.0567808e-06)
			(1/5,3.6653564e-06 )
			(1/6,1.7503598e-06)
			(1/7, 9.3734327e-07)
			(1/8, 	5.4592541e-07)
		};,
		\addplot[color=blue, mark size=1.8pt,mark=none, dashed]
		coordinates {
			(1, 0.06*1)
			(1/8, 0.06/8/8)
		};
		\addplot[color=red, mark size=1.8pt,mark=none, dashed]
		coordinates {
			(1, 0.011*1)
			(1/8, 0.011/8/8/8)};
		\addplot[color=black, mark size=1.8pt,mark=none, dashed]
		coordinates {
			(1, 0.00128*1)
			(1/8, 0.00128/8/8/8/8)};	
		\end{loglogaxis}
		\end{tikzpicture}
	\end{minipage}
	\caption{$2D$ Multi-patch example: Errors w.r.t. the mesh size and for the degrees $\textsf{p}=2,3,4$ and $r=0$.}
	\label{Fig:Num_6}
\end{figure}
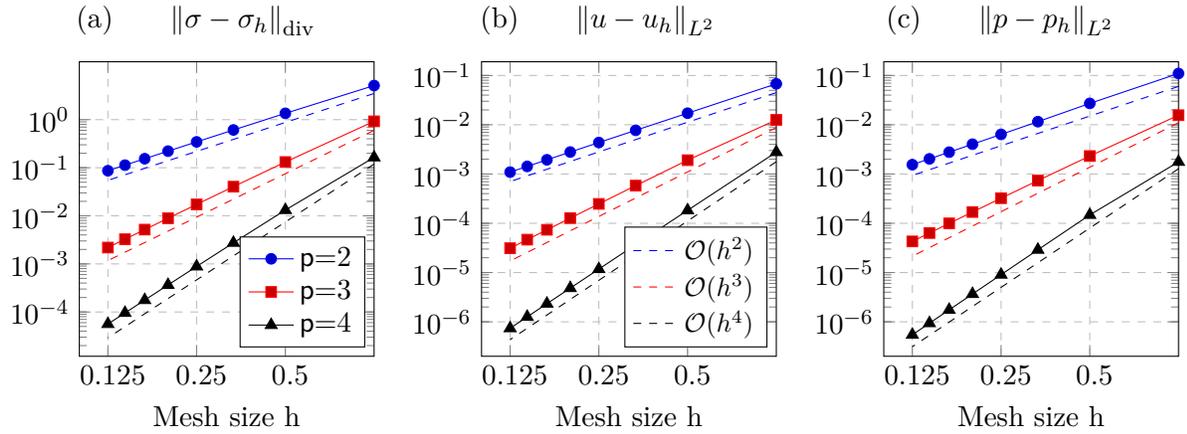

\subsubsection{Cook's membrane}
A classical benchmark problem within the scope of (nearly) incompressible elasticity theory is the \emph{Cook membrane}. It is represented by a rectangular domain that is fixed at the left side and is loaded by a uniformly distributed shear force $t_n=(0,1/16)^T$ at the right boundary; see Fig. \ref{Fig:Num_5} (a) and cf. \cite{Test_cook}. For numerical methods that suffer from locking, one often can observe  far too small approximate displacements. For comparison purposes we apply our discrete scheme with $\lambda = \infty, \  \mu=0.375$ analogous to the example in \cite{Test_cook}. In latter article a reference value for the displacement at the point $P$ in Fig.  \ref{Fig:Num_5} (a) is computed and we study how our approximations tend to these reference values. In Fig.  \ref{Fig:Num_5} (b)-(c) we see the results for the special case $\textsf{p}=r+2$ and $\textsf{p}=2,3,4$. Since for all degrees the differences between reference and approximate   values diminish,  we can conclude here the absence of locking effects.

\begin{figure}
	\begin{minipage}{0.35\textwidth}
		\hspace*{-0.2cm} \vspace*{0.1cm}
		\begin{tikzpicture}[scale=1]
		\node (eins) at (-2.7,4.16) {\includegraphics[width=0.61\linewidth]{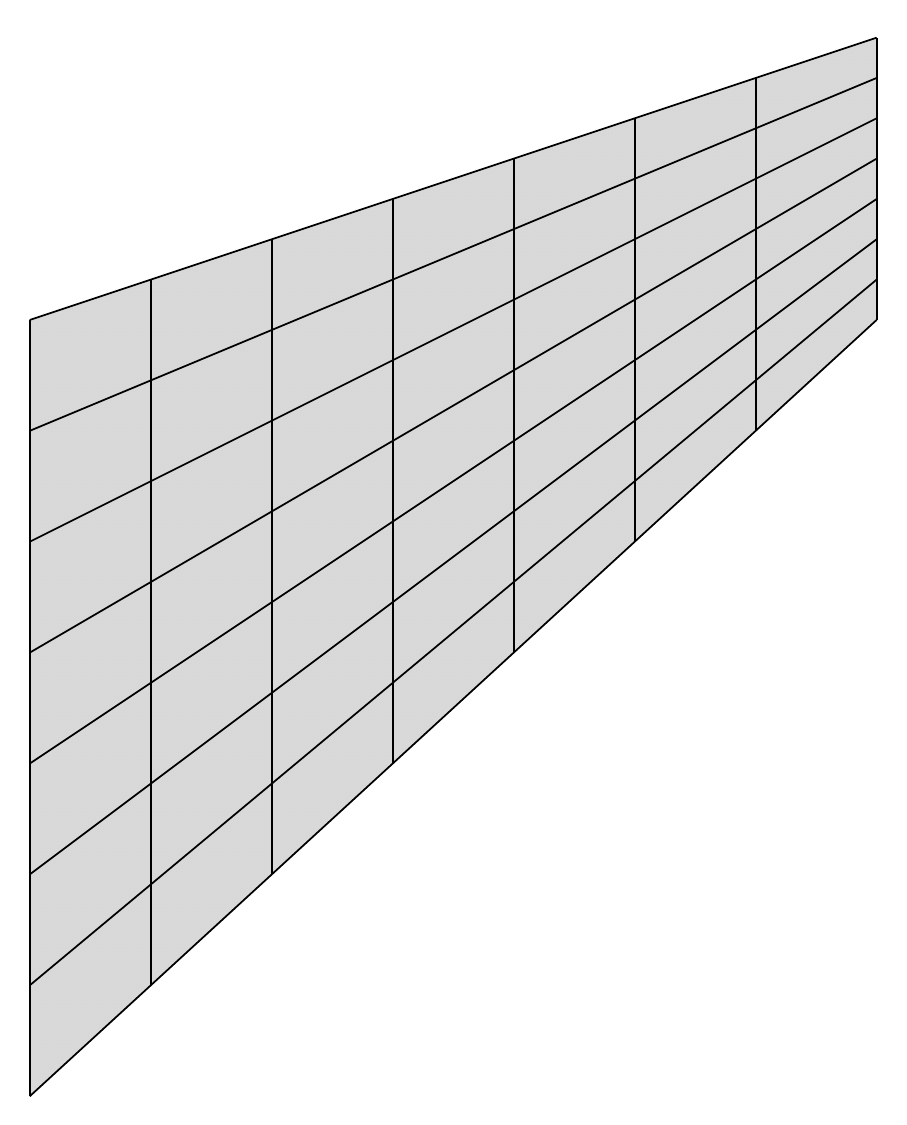}};
		\draw[|-|,scale=1.28] (-0.6,4) -- (-0.6,4.88);
		\draw[|-|,scale=1.28] (-3.6,0+1.6) -- (-3.6,2.4+1.6);
		\draw[|-|,scale=1.28] (-0.6,0+1.6) -- (-0.6,2.4+1.6);
		\draw[|-|,scale=1.28] (-3.425,1.4) -- (-0.79,1.4);		
		{	\foreach \x in {1.6,1.7,1.8,1.9,2,2.1,2.2,2.3,2.4,2.5,2.6,2.7,2.8,2.9,3,3.1,3.2,3.3,3.4,3.5,3.6,3.7,3.8,3.9}
			\draw[scale=1.28] (-3.5,\x ) to (-3.4 ,\x +0.1);}
		\draw[->,very thick,scale=1.28] (-0.7,4.3) -- (-0.7,5.3);
		\draw[fill=red,scale=1.28] (-0.8,4.43) circle(0.06);
		\node[red,scale=1.28]  at (-1.2,5.33) {$P$};
		\node[left,scale=1.28]  at (-4.5,3.6) {\footnotesize $44$};
		\node[left,scale=1.28]  at (-0.1,3.6) {\footnotesize $44$};
		\node[left,scale=1.28]  at (-0.1,5.65) {\footnotesize $16$};
		\node[left,scale=1.28]  at (-2.3,2) {\footnotesize $48$};
		\node[left,scale=1.28]  at (-0.9,6.6) {\small $t_n$};
		\node[left,scale=1]  at (-3.9,6.5) { (a) };
		\end{tikzpicture}
	\end{minipage}
	\hspace*{-0.4cm}
	\begin{minipage}{0.33\textwidth}
		\begin{tikzpicture}[scale=1]
		\begin{semilogxaxis} [
		title={(b) \hspace*{0.0cm} $y$-displacement at $P$ \hspace*{0.2cm}},
		width = 5.5cm,
		height=5.5cm,
		xlabel={Number of DoF},
		xmin=20, xmax=12000,
		ymin=15.5, ymax=22,
		xtick={10^1,10^2,10^3,10^4},
		xticklabel style={/pgf/number format/sci},
		ytick={14,16,18,20,22},
		legend pos=north east,
		xmajorgrids=true,
		ymajorgrids=true,
		grid style=dashed,
		legend columns=1,
		]				 
		\addplot[color=blue,mark=*,
		mark options={solid, fill=blue,scale=0.8}]
		coordinates {			
			(24,   21.0654 )
			(393, 		16.9836)
			(1212, 16.6973 )
			(2481, 16.5989)
			(4200, 16.5535 )
			(6369, 16.5269)
			(8988, 16.5112)
		};
		\addlegendentry{\textsf{p}=2},
		\addplot[color=red,mark=square*,
		mark options={solid, fill=red,scale=0.8}]
		coordinates {			
			(57,  18.4602 )
			(504, 	16.6112)
			(1401, 	16.5217)
			(2748, 	16.4952)
			(4545, 16.4833 )
			(6792, 16.4767)
			(9489, 16.4724)
		};
		\addplot[color=black,mark=triangle*,
		mark options={solid, fill=black,scale=0.8}]
		coordinates {			
			(104,17.1732 )
			(629, 	 16.5297)
			(1604,  16.4907 )
			(3029,  16.4785)
			(4904, 16.4723 )
			(7229, 16.4685)
			(10004, 16.4658)
		};,
		\addplot[color=brown, mark size=1.8pt,mark=none, dashed,line width=1.1]
		coordinates {
			(30, 16.442)
			(12000, 16.4420)
		};
		\legend{$\textsf{p}=2$,$\textsf{p}=3$,$\textsf{p}=4$,\small ref. value,,}	
		\end{semilogxaxis}
		\end{tikzpicture}
	\end{minipage}
	\hspace*{-0.2cm}
	\vspace*{0.28cm}
	\begin{minipage}{0.33\textwidth}
		\begin{tikzpicture}[scale=1]
		\begin{semilogxaxis} [
		title={(c) \hspace*{0.0cm} $x$-displacement at $P$ \hspace*{0.2cm}},
		width = 5.5cm,
		height=5.5cm,
		xlabel={Number of DoF},
		xmin=20, xmax=12000,
		ymin=-10.1, ymax=-6.8,
		xtick={10^1,10^2,10^3,10^4},
		xticklabel style={/pgf/number format/sci},
		ytick={-10,-9,-8,-7},
		legend pos=north east,
		xmajorgrids=true,
		ymajorgrids=true,
		grid style=dashed,
		legend columns=1,
		]		 
		\addplot[color=blue,mark=*,
		mark options={solid, fill=blue,scale=0.8}]
		coordinates {			
			(24,   -9.7186 )
			(393, -7.5050)
			(1212, -7.3813 )
			(2481, -7.3202)
			(4200, -7.3031 )
			(6369, -7.2880)
			(8988, -7.2830)
		};
		
		\addplot[color=red,mark=square*,
		mark options={solid, fill=red,scale=0.8}]
		coordinates {			
			(57,  -8.6432 )
			(504, 	-7.3319)
			(1401, 	-7.2865)
			(2748, 	-7.2767)
			(4545, -7.2719 )
			(6792, -7.2690)
			(9489, -7.2668)
		};
		\addplot[color=black,mark=triangle*,
		mark options={solid, fill=black,scale=0.8}]
		coordinates {			
			(104,-7.6030 )
			(629, -7.2887)
			(1604,  -7.2769 )
			(3029,  -7.2715)
			(4904, -7.2678)
			(7229, -7.2653)
			(10004, -7.2636)
		};,
		\addplot[color=brown, mark size=1.8pt,mark=none, dashed,line width=1.1]
		coordinates {
			(30, -7.2480)
			(12000, -7.2480)
		};
		\end{semilogxaxis}
		\end{tikzpicture}
	\end{minipage}
	\caption{For the Cook membrane we get good displacement values in $P$ although we have set $\lambda=\infty$ as well as $r=\textsf{p}-2$.}
	\label{Fig:Num_5}
\end{figure}
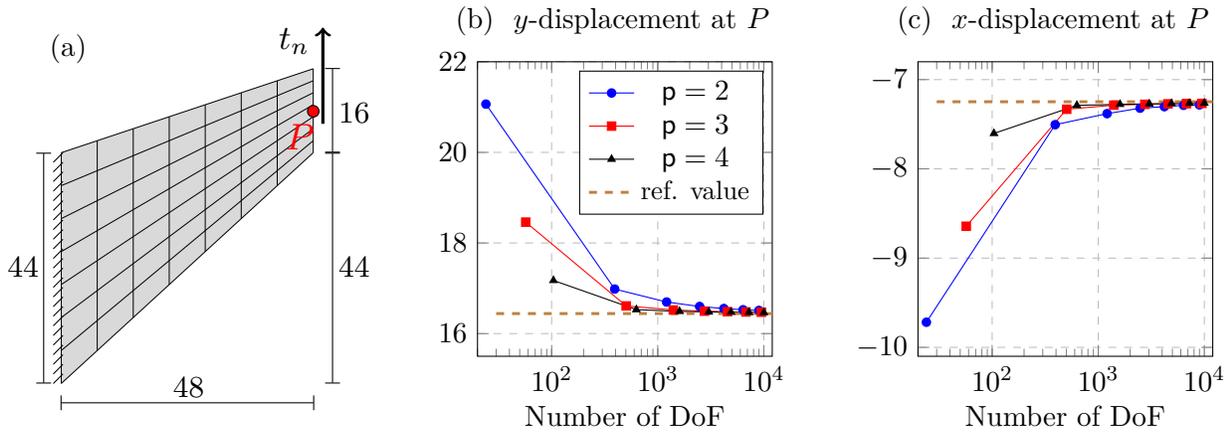

\subsubsection{$3D$ Ring}

Next we look at a convergence test example in $3D$, where the geometry and the mesh for $h=1/6$ are given in Fig. \ref{Fig:Num_8} (a). Again we define the source term $f$ such that the exact solution has zero displacement boundary values and the exact solution is $$u_1 = \frac{g}{2}, \ u_2=g, \ u_3 = -\frac{g}{2}, \ \ \   \textup{with} \ \ \ g \circ \p{F}(\zeta_1,\zeta_2,\zeta_3) = \sin(\pi \zeta_1) \sin(\pi \zeta_2) \sin(\pi \zeta_3).$$ Above, $\p{F} \colon (0,1)^3 \rightarrow \o$ denotes the regular parametrization of $\o$.  The weak formulation with weakly imposed symmetry incorporates besides the displacement also the stresses and the Lagrange multiplier $p \in L^2(\o,\mathbb{R}^3)$. Consequently,  if we have finer meshes or higher degrees  one runs into trouble due to a large number of degrees of freedom (DoF). Especially the  solving of the underlying large linear systems causes a lot of  computational effort. That is why we applied the iterative solver \emph{minres()} to solve the system of equations in the $3D$ case. The mentioned solver is pre-implemented in MATLAB and as relative tolerance we chose $5\cdot 10^{-8}$, i.e. $$ \frac{\n{M \bar{x} - b}}{\n{b}} \leq 5\cdot 10^{-8}  \ \footnote[2]{$\norm{\cdot}$ stands for the Euclidean norm.},$$
with $M, \ b$ denoting the system matrix and right-hand side and $\bar{x}$ is the approximate solution obtained by \emph{minres()}. 
Since the norms of relative and also of the absolute residuals are much less than the errors we computed in the convergence test, we think that the  plots are still meaningful. In Fig. \ref{Fig:Num_10} we see the error decays for $\textsf{p}=2, r=0$ and  $\textsf{p}=3,r=1$ and the decreasing basically fits to the theoretically predicted order $\mathcal{O}(h^\textsf{p})$. However, the convergence order of the displacement field seems to be one higher than the one from the theory. This indicates that the estimates from the sections above might be too pessimistic. There still are open questions whether the underlying inf-sup condition can be proven with less strict requirements.  On the other hand,  the refined error estimate $\n{\nabla \cdot \sigma- \nabla \cdot \sigma_h}_{L^2} \in \mathcal{O}(h^{\textsf{p}+1})$ can be realized well in Fig. \ref{Fig:Num_8}. Furthermore, the approximations for $u_1$ and $\sigma_{11}$ on the slice $z=0.5$ through the ring domain are pictured in Fig.   \ref{Fig:Num_12}. 
In addition to the ring example we show the feasibility of multi-patch geometries in $3D$ by considering a spherical domain of radius $R=1$. Such a shape needs a decomposition in several patches to avoid parametrization singularities. More precisely, we use the $7$-patch decomposition illustrated in Fig. \ref{Fig:Num_11} (a) that is provided by the GeoPDEs software \cite{geopdes3.0}. In order to have a reference solution, we adapted the right-hand source $f$ to   get the exact displacements
$$u_1 = \cos(0.5 \pi \rho^2), \ \ \ u_2 = -u_1, \ \ \  u_3 = 2 u_1,$$
where $\rho$ denotes here the radial coordinate. Again we set $\lambda=2, \ \mu=1$.
For the mesh  in Fig. \ref{Fig:Num_11} (a) and $p=2,r=0$ we obtain an error $\n{u-u_h}_{L^2} \approx 0.00468$ and the first approximate displacement component  is  displayed in the mentioned figure, too. One notes the advantage of the IGA ansatz that allows for the exact geometry representation.

\begin{figure}[h!]
	\hspace*{0.5cm}
	\begin{tikzpicture}
	\node (zwei) at (0,0) {\includegraphics[width=0.32\linewidth]{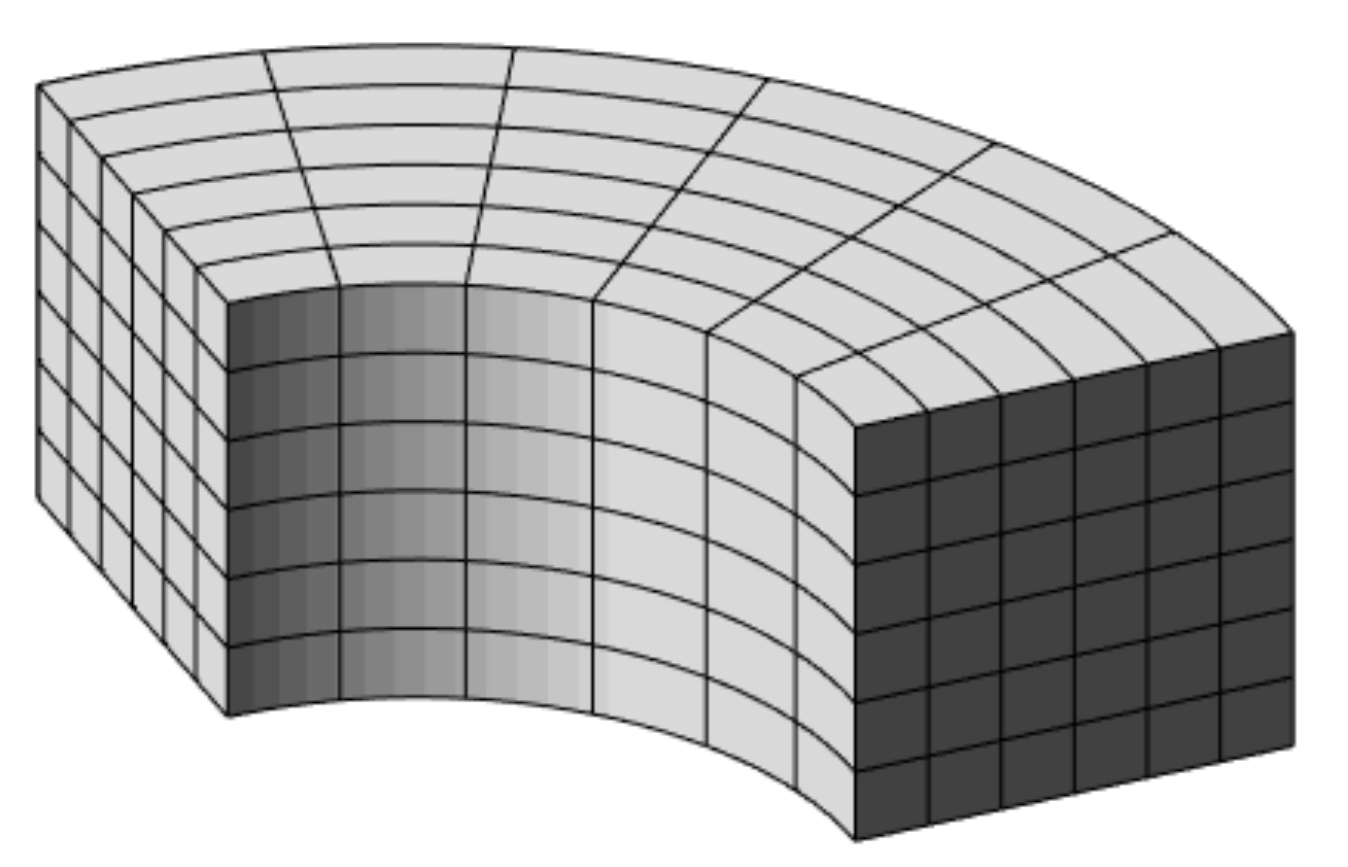}};
	\draw[scale=1,dashed,thick] (-1,-2-0.05) -- (-1+0.48*3.5,-2+0.48*0.8-0.48*0.05-0.01);
	\draw[scale=1,|-|] (-1+0.05,-2-0.05-0.2) -- (-1+0.49*3.5+0.05,-2+0.49*0.8-0.49*0.05-0.01-0.2);
	\draw[scale=1, thick,dashed] (-1,-2-0.05) -- (-1-0.39*2,-2+0.39*2.3-0.39*0.05-0.02);
	\draw[scale=1, |-|] (-1-0.15,-2-0.05-0.1) -- (-1-0.79*2-0.15,-2+0.79*2.3-0.79*0.05-0.02-0.1);
	
	\draw[scale=1, |-|] (2.7,-1.28) -- (2.7,0.38);
	\node at (0,-2.3) {\small $1$}; 
	\node at (-2.1,-1.4) {\small $2$};
	\node at (2.9,-0.45) {\small $1$};	
	\node at (-3.3,-2.3) {\small (a)};	
	
	\draw[->,shift={(-0.7,-1.7)}] (-3,0) --(-2,1.2/5);
	\draw[->,shift={(-0.7,-1.7)}] (-3,0) --(-3,1);
	\draw[->,shift={(-0.7,-1.7)}] (-3,0) --(-3-3.5/6,4/6);
	\node at (-3,-1.35) {\small $x$};
	\node at (-4.3,-1.35) {\small $y$};
	\node at (-3.55,-0.9) {\small $z$};
	\end{tikzpicture}
	\hspace*{0.8cm}
	\begin{minipage}{0.35\textwidth}
		\vspace*{-3.6cm}
		\begin{tikzpicture}[scale=1]
		\begin{loglogaxis} [
		title={$\n{\nabla \cdot \sigma - \nabla \cdot \sigma_h}_{L^2}$},
		width = 5.5cm,
		height=5.5cm,
		xlabel={Mesh size h},
		xmin=0.05, xmax=0.5,
		ymin=2*10^-4, ymax=0.5*10^1,
		xtick={1/16,1/8,0.25,0.5},
		log x ticks with fixed point,
		ytick={10^-3,10^-2,10^-1,10^0},
		legend pos= outer north east,
		xmajorgrids=true,
		ymajorgrids=true,
		grid style=dashed,
		legend columns=1,
		]
		\addplot
		coordinates {
			(1/2, 1.5490403)
			(1/4, 	0.1798428)
			(1/6, 0.053894847)
			(1/8, 0.022826837)
			(1/10, 0.011708598 )
			(1/12,0.0067824899)
			(1/14, 0.0042737332)
		};
		\addplot
		coordinates {
			(1/2, 0.256817825)
			(1/4, 0.0214768347)
			(1/6, 	0.00390787752)
			(1/8, 	0.0011883617)
			(1/10,0.000490565369)
		};
		\addplot[color=blue, mark size=1.8pt,mark=none, dashed]
		coordinates {
			(1/2, 7.3*0.5^3)
			(1/14, 7.3/14/14/14)
		};
		\addplot[color=red, mark size=1.8pt,mark=none, dashed]
		coordinates {
			(1/2, 2.9*0.5^4)
			(1/10, 2.9*0.1^4)};
		\legend{\textsf{p}=2,\textsf{p}=3,\small $\mathcal{O}(h^3)$,\small $\mathcal{O}(h^4)$}	
		\end{loglogaxis}
		\node at (0,4.45) {\small (b)};
		\end{tikzpicture}
	\end{minipage}
	\caption{A mesh for the $3D$ convergence example ($h=1/6$) is shown on the left, whereas on the right we see the corresponding errors for the divergence. For $\textsf{p}=2$ it is $r=0$ and for $\textsf{p}=3$ we set $r=1$.  }
	\label{Fig:Num_8}
\end{figure}
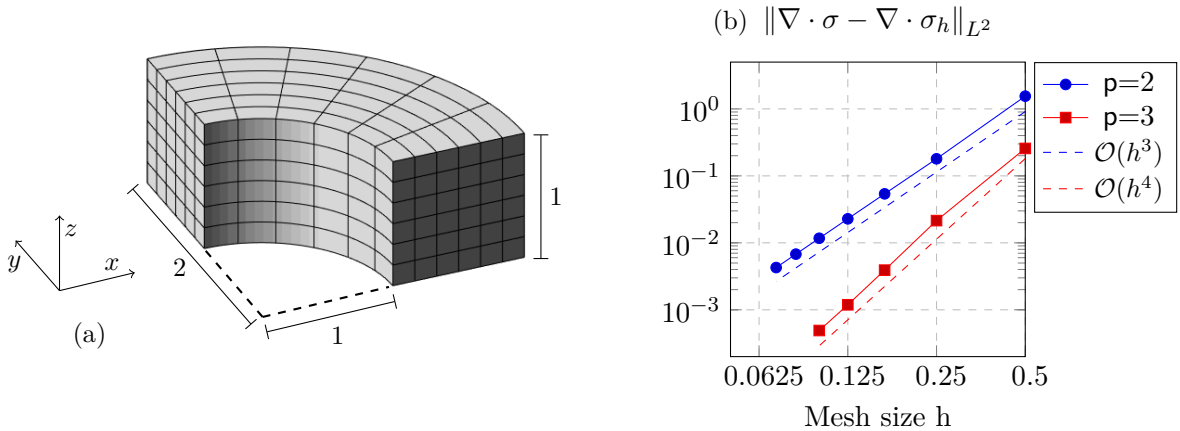

\begin{figure}[h!]
	\centering\begin{tikzpicture}
	\node (eins) at (-7.5,0) {\includegraphics[width=0.3\linewidth]{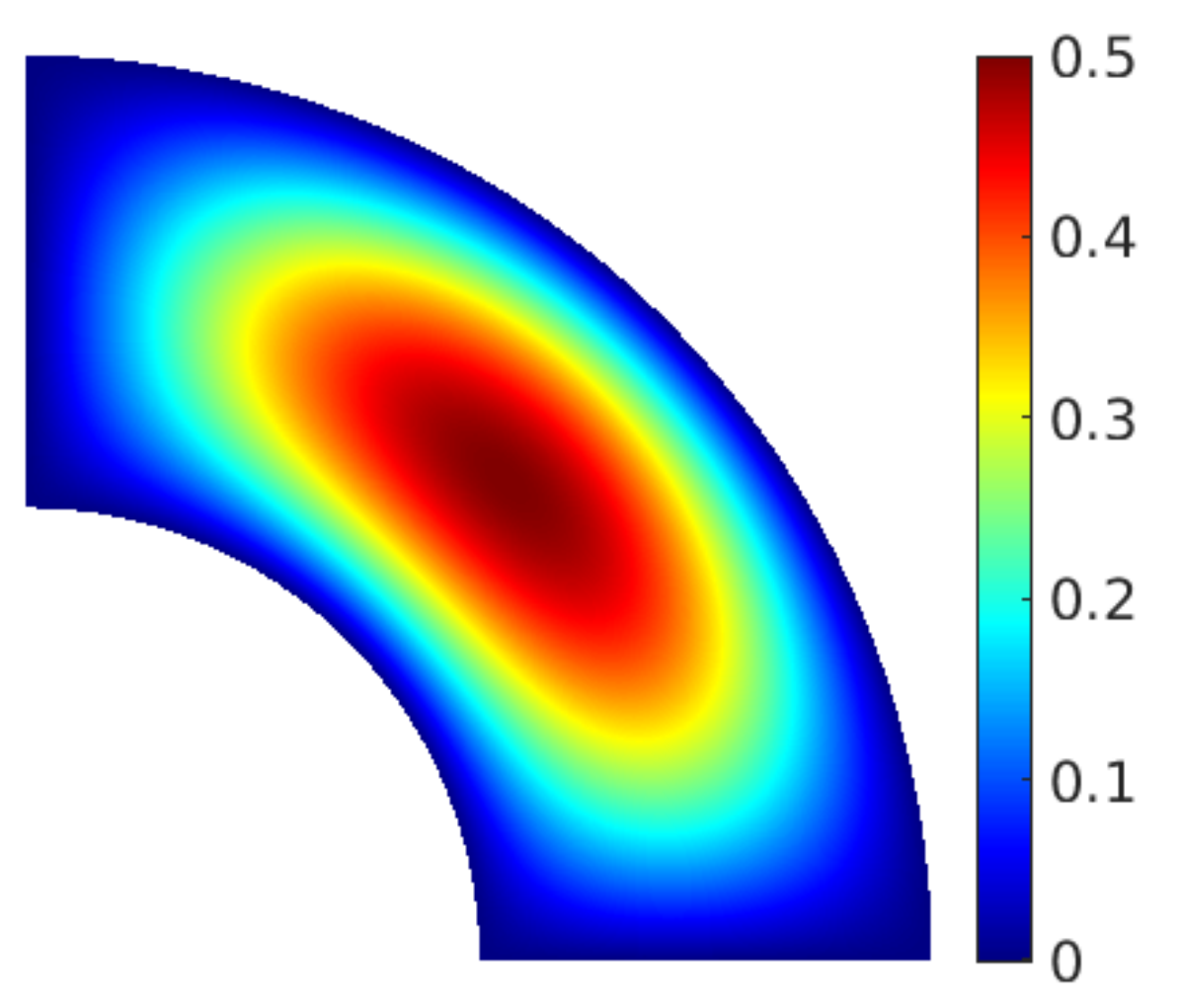}};
	\node[shift={(0.18,-0.01)}] (eins) at (0,0) {\includegraphics[width=0.289\linewidth]{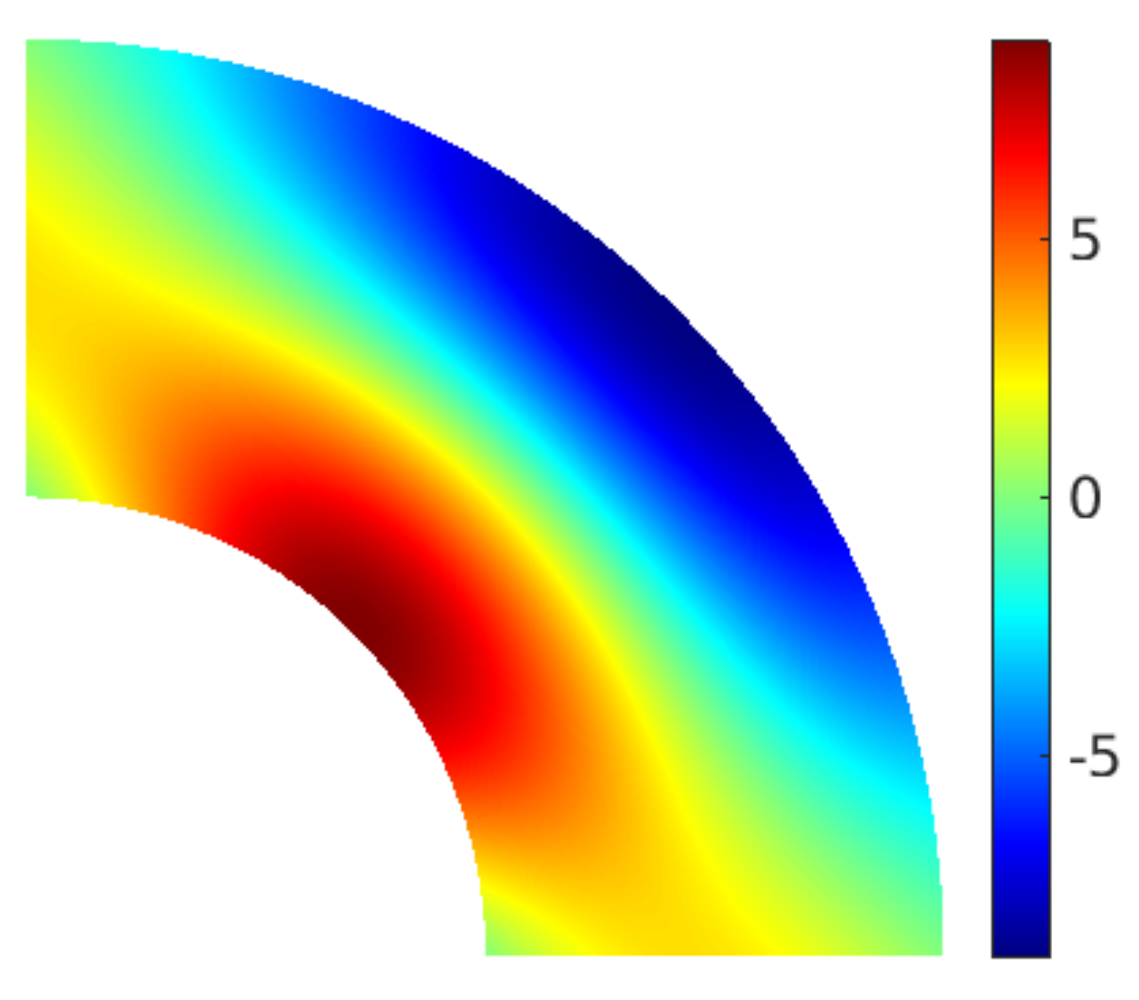}};
	\node at (-8.7,-2.5) {\small (a)  $u_1$};
	\node at (-1,-2.5) {\small (b)  $\sigma_{11}$};
	\end{tikzpicture}	
	\caption{ View on the approximate solution of the $3D$ ring on the middle slice $z=1/2$. ($\textsf{p}=3, \, r=1,\, h=1/6$) }
	\label{Fig:Num_12}
\end{figure}

\begin{figure}[h!]
	\begin{minipage}{0.33\textwidth}
		\begin{tikzpicture}[scale=1]
		\begin{loglogaxis} [
		title={(a) \hspace*{0.5cm} $\n{\sigma-\sigma_h}_{\d}$ \hspace*{0.5cm}},
		width = 5.5cm,
		height=5.5cm,
		xlabel={Mesh size h},
		xmin=0.05, xmax=0.5,
		ymin=0.5*10^-3, ymax=0.8*10^1,
		xtick={1/16,1/8,0.25,0.5},
		log x ticks with fixed point,
		ytick={10^-3,10^-2,10^-1,10^0},
		legend pos= north west,
		xmajorgrids=true,
		ymajorgrids=true,
		grid style=dashed,
		legend columns=2,
		]
		\addplot
		coordinates {
			(1/2, 1.7266104)
			(1/4,  0.23702499)
			(1/6, 0.082220698)
			(1/8, 0.040413722)
			(1/10, 	0.0238866 )
			(1/12, 	0.015774842)
			(1/14, 0.011207412)

		};
		\addlegendentry{\textsf{p}=2},
		\addplot
		coordinates {
			(1/2,  0.311456307)
			(1/4, 	0.0353651526)
			(1/6, 	0.00783834134)
			(1/8, 	0.00286270757)
			(1/10,	0.00138941098)
			
		};
		\addlegendentry{\textsf{p}=3},
		\addplot[color=blue, mark size=1.8pt,mark=none, dashed]
		coordinates {
			(1/2, 1.6*0.5^2)
			(1/14, 1.6/14/14)
		};
		\addplot[color=red, mark size=1.8pt,mark=none, dashed]
		coordinates {
			(1/2, 0.92*0.5^3)
			(1/10, 0.92*0.1^3)};
		\addlegendentry{\small $\mathcal{O}(h^2)$},
		\addlegendentry{\small $\mathcal{O}(h^3)$},		
		\end{loglogaxis}
		\end{tikzpicture}
	\end{minipage}
	\hspace*{-0.45cm}
	\begin{minipage}{0.33\textwidth}
		\begin{tikzpicture}[scale=1]
		\begin{loglogaxis} [
		title={(b) \hspace*{0.5cm} $\n{u-u_h}_{L^2}$ \hspace*{0.7cm}},
		width = 5.5cm,
		height=5.5cm,
		xlabel={Mesh size h},
		xmin=0.05, xmax=0.5,
		ymin=0.6*10^-5, ymax=0.7*10^-1,
		xtick={1/16,1/8,0.25,0.5},
		log x ticks with fixed point,
		ytick={10^-5,10^-4,10^-3,10^-2,10^-1,10},
		legend pos=north west,
		xmajorgrids=true,
		ymajorgrids=true,
		grid style=dashed,
		]			
		\addplot
		coordinates {
			(1/2,  0.025486426)
			(1/4, 0.0028760066 )
			(1/6, 0.00087007517)
			(1/8,0.00037430249)
			(1/10, 0.00019399575 )
			(1/12, 0.000113192)
			(1/14, 	7.1707176e-05)
			
		};
		\addplot
		coordinates {
		(1/2,  0.00362969937)
		(1/4, 	0.000500295061)
		(1/6, 9.85153935e-05)
		(1/8,3.16182516e-05)
		(1/10,1.31347917e-05)			
		};	
		\addplot[color=blue, mark size=2.2pt,mark=none, dashed]
		coordinates {
			(1/2, 0.12*1/8)
			(1/14, 0.12/14/14/14)
		};
		\addlegendentry[eins]{\textsf{p}=3, r=1},
		\addplot[color=red, mark size=2.2pt,mark=none, dashed]
		coordinates {
			(1/2, 0.08*1/16)
			(1/10, 0.08*1/10*1/10*1/10*1/10)};	
		\addlegendentry[zwei]{\textsf{p}=3, r=1},
		\addlegendentry[drei]{\textsf{p}=3, r=1},
		\addlegendentry[vier]{\textsf{p}=3, r=1},
		\legend{,,\small $\mathcal{O}(h^3)$,\small $\mathcal{O}(h^4)$}
		\end{loglogaxis}
		\end{tikzpicture}
	\end{minipage}
	\hspace*{-0.45cm}
	\begin{minipage}{0.33\textwidth}
		\begin{tikzpicture}[scale=1]
		\begin{loglogaxis} [
		title={(c) \hspace*{0.5cm} $\n{p-p_h}_{L^2}$ \hspace*{0.7cm}},
		width = 5.5cm,
		height=5.5cm,
		xlabel={Mesh size h},
		xmin=0.05, xmax=0.5,
		ymin=0.25*10^-3, ymax=0.6,
		xtick={1/16,1/8,1/4,0.5},
		log x ticks with fixed point,
		ytick={10^-3,10^-3,10^-2,10^-1,10},
		legend pos=north west,
		xmajorgrids=true,
		ymajorgrids=true,
		grid style=dashed,
		]			
		\addplot
		coordinates {
			(1/2,  0.24986851)
			(1/4, 	0.057731336)
			(1/6, 0.024419059)
			(1/8,0.013471312)
			(1/10, 0.0085418207)
			(1/12, 0.0059013827)
			(1/14, 	0.0043221226)

		};
		\addplot
		coordinates {
			(1/2, 0.060960086)
			(1/4, 0.0104578386)
			(1/6,0.00263148909)
			(1/8,0.00103959659)
			(1/10, 0.000516867924 )			
		};	
		\addplot[color=blue, mark size=2.2pt,mark=none, dashed]
		coordinates {
			(1/2, 0.65*1/4)
			(1/14, 0.65*1/14/14)
		};
		\addlegendentry[eins]{\textsf{p}=3, r=1},
		\addplot[color=red, mark size=2.2pt,mark=none, dashed]
		coordinates {
			(1/2, 0.36*1/8)
			(1/10, 0.36*1/10*1/10/10)};	
		\addlegendentry[zwei]{\textsf{p}=3, r=1},
		\addlegendentry[drei]{\textsf{p}=3, r=1},
		\addlegendentry[vier]{\textsf{p}=3, r=1},
		\legend{,,\small $\mathcal{O}(h^2)$,\small $\mathcal{O}(h^3)$}
		\end{loglogaxis}
		\end{tikzpicture}
	\end{minipage}
	\caption{Error decay for the $3D$ ring example. We note that we choose the regularity parameter $r=0$ if $\textsf{p}=2$, but $r=1$ for the case $\textsf{p}=3$. }
	\label{Fig:Num_10}
\end{figure}
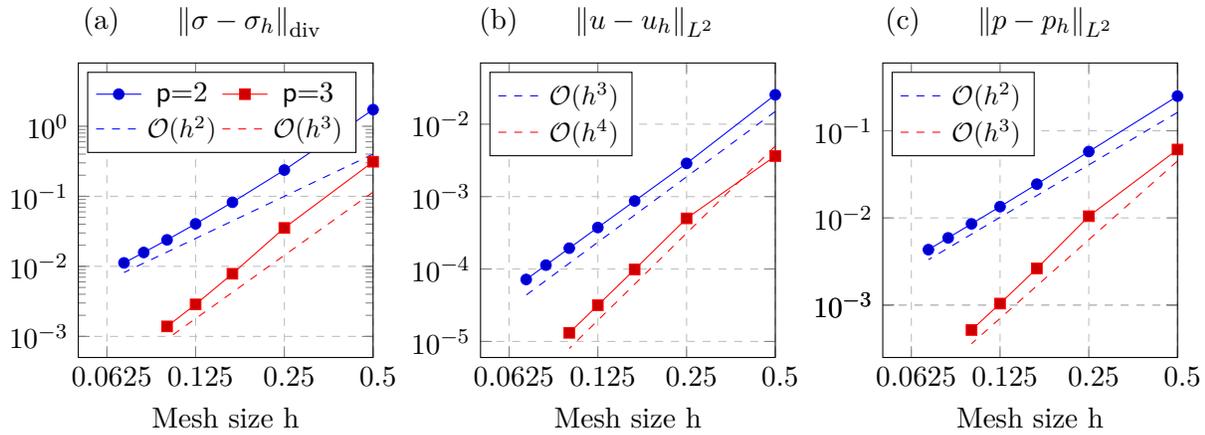

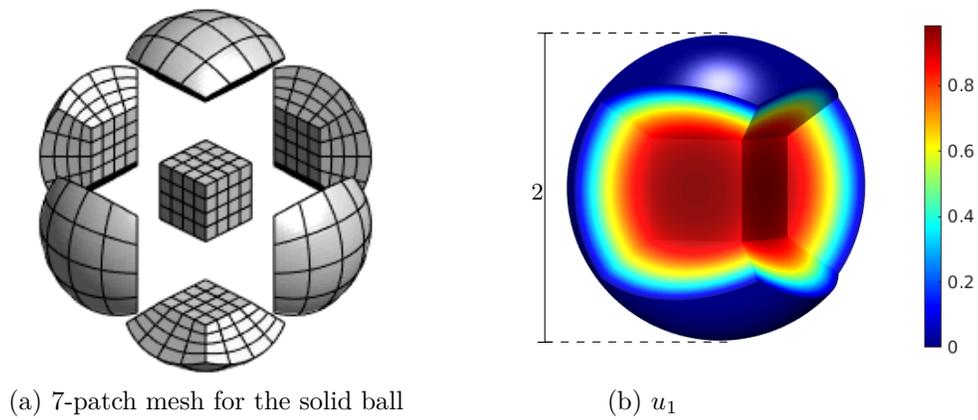
\begin{figure}[h!]
	\centering\begin{tikzpicture}
	\node (eins) at (-7.5,0) {\includegraphics[width=0.28\linewidth]{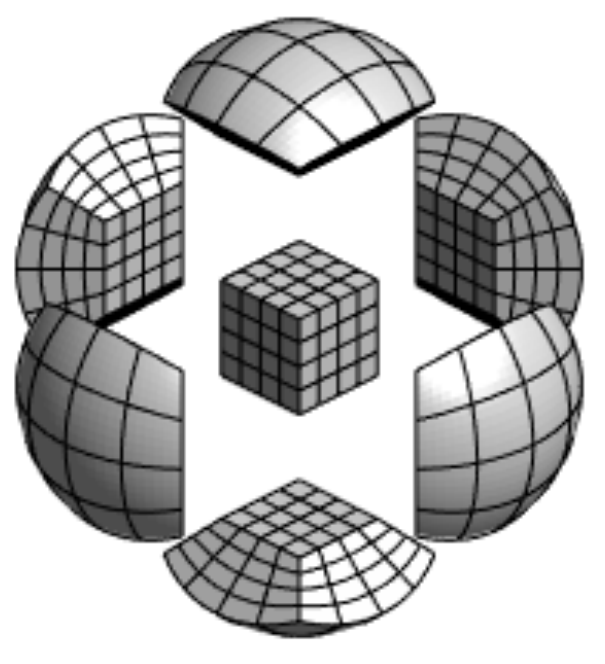}};
	\node (eins) at (0,0.05) {\includegraphics[width=0.337\linewidth]{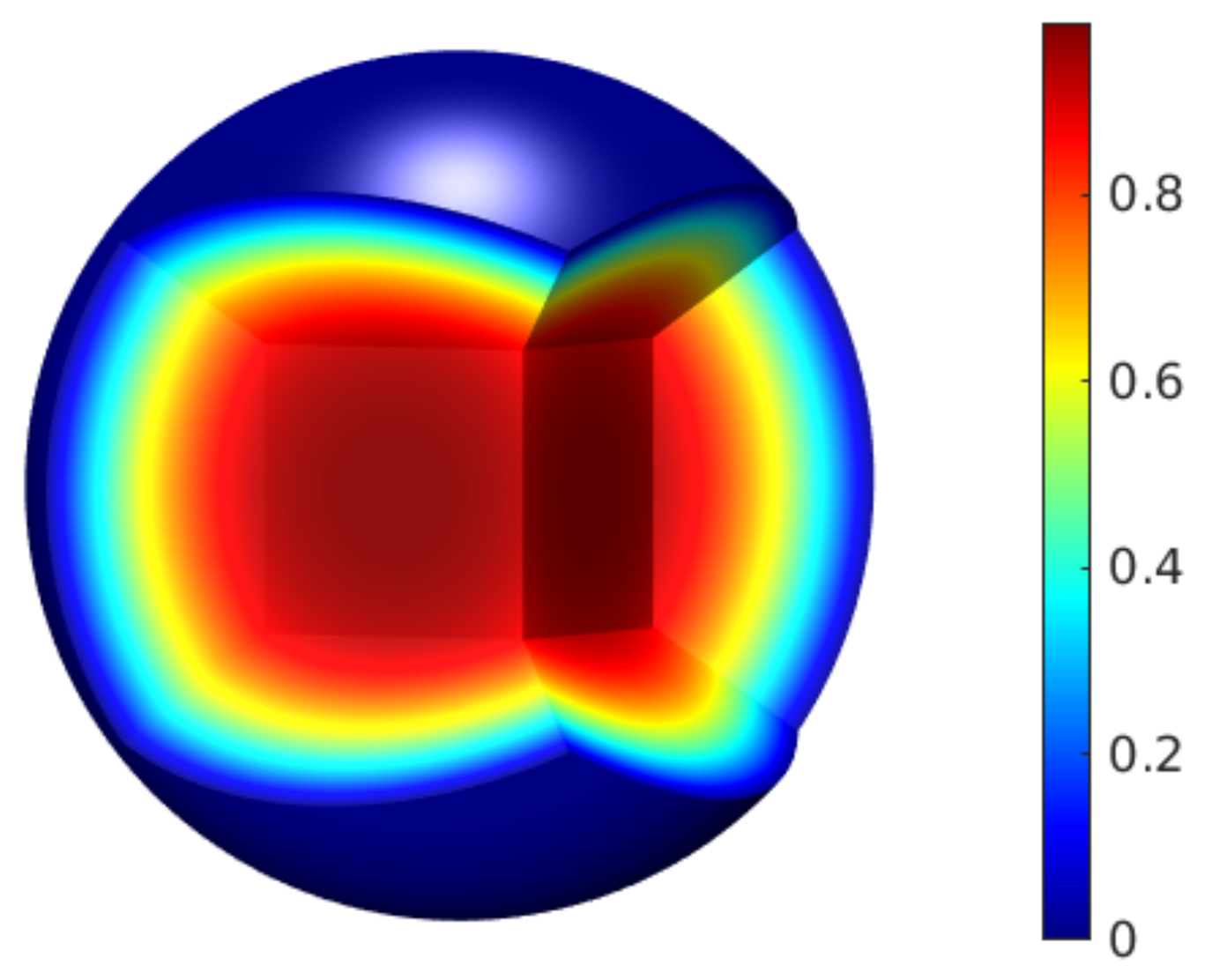}};
	\draw[|-|,shift={(-3,-2)}] (0,0) --(0,4.125);
	\draw[dashed,shift={(0,-1.99)}] (-3,0) --(-0.5,0);
	\draw[dashed,shift={(0,2.12)}] (-3,0) --(-0.5,0);
	\node at (-7.5,-2.8) {\small (a) $7$-patch mesh for the solid ball };
	\node at (-1.7,-2.8) {\small (b) $u_1$ };
	\node at (-3.1,0) {\footnotesize $2$};
	\end{tikzpicture}
	
	\caption{The mesh for the $3D$ multi-patch example ($h=1/4$) on the left. On the right we picture the approximate displacement solution $u_1$ for the mesh on the left and $\textsf{p}=3,r=1,\lambda=10,\mu=1$. }
	\label{Fig:Num_11}
\end{figure}

\newpage
After the different  examples in two and three dimensions we see that the applied mixed spaces lead to reasonable results and we have indeed a unique approximate solution. Admittedly, one has to implement several discrete spaces involving different polynomial degrees and regularity. Consequently, the question arises whether the usage of classical B-spline spaces of the form $S_{\textsf{p},\dots,\textsf{p}}^{r,\dots,r} \circ \p{F}^{-1}$ without special transformations $\mathcal{Y}_i$ and equal degrees leads to comparable results.  To answer this question and to demonstrate the need of the de Rham spaces we show in Fig. \ref{Fig:Num_14}  approximate first displacement components for the case of classical B-spline spaces for all variables $\sigma,u,p$ and  with $\textsf{p}=2, r=0$ w.r.t. each coordinate. In other words, we ignore the required Brezzi stability conditions. Then we obtain for the shown example in $2D$ a nearly singular system matrix and for both $x$-displacements we observe in Fig. \ref{Fig:Num_14}  a completely unstable behavior. Thus standard spaces fail.    

\begin{figure}[h!]
	\centering\begin{tikzpicture}
	\node (eins) at (-7.5,0) {\includegraphics[width=0.33\linewidth]{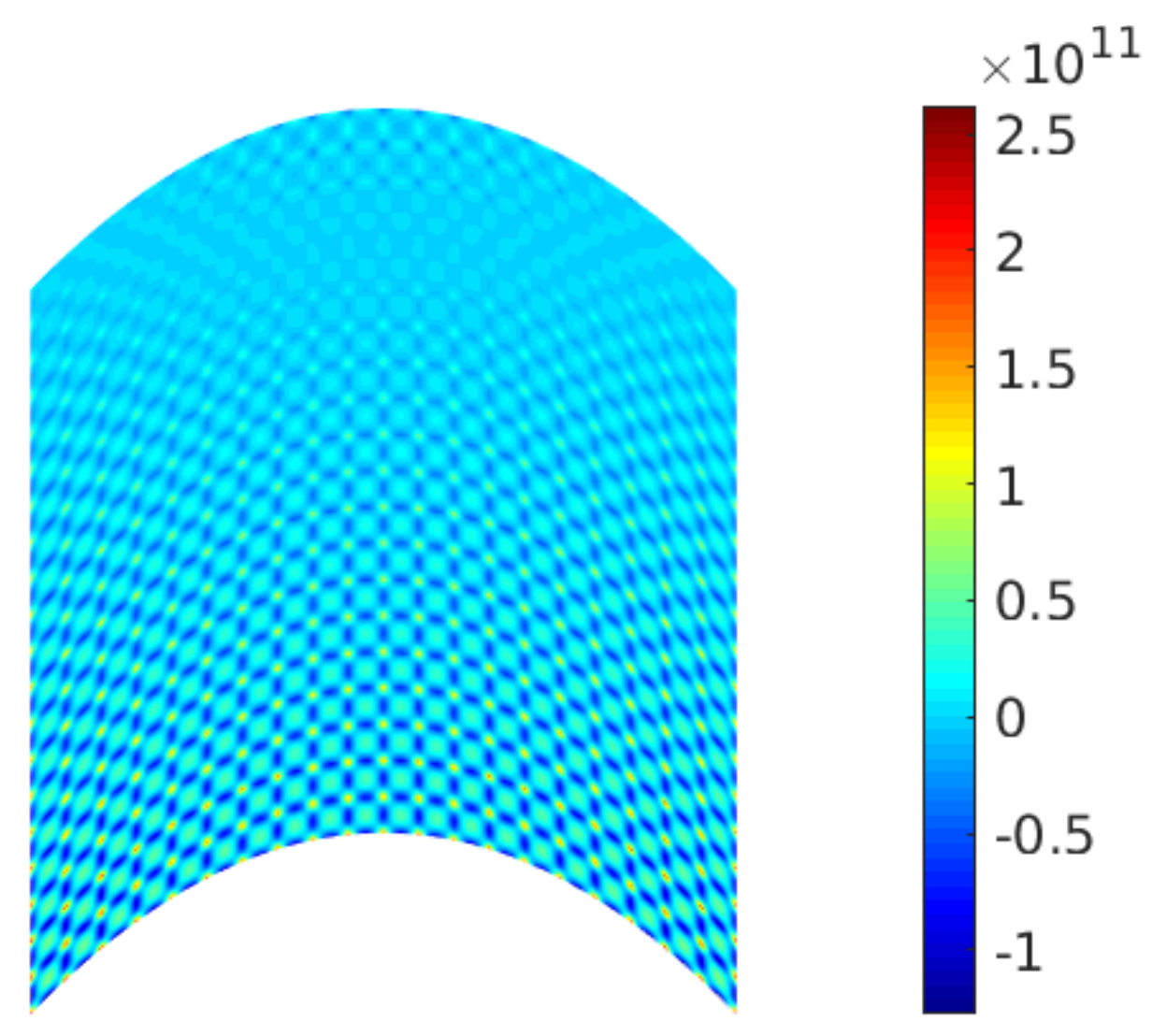}};
	\node[shift={(0.18,-0.01)}] (eins) at (0,0) {\includegraphics[width=0.33\linewidth]{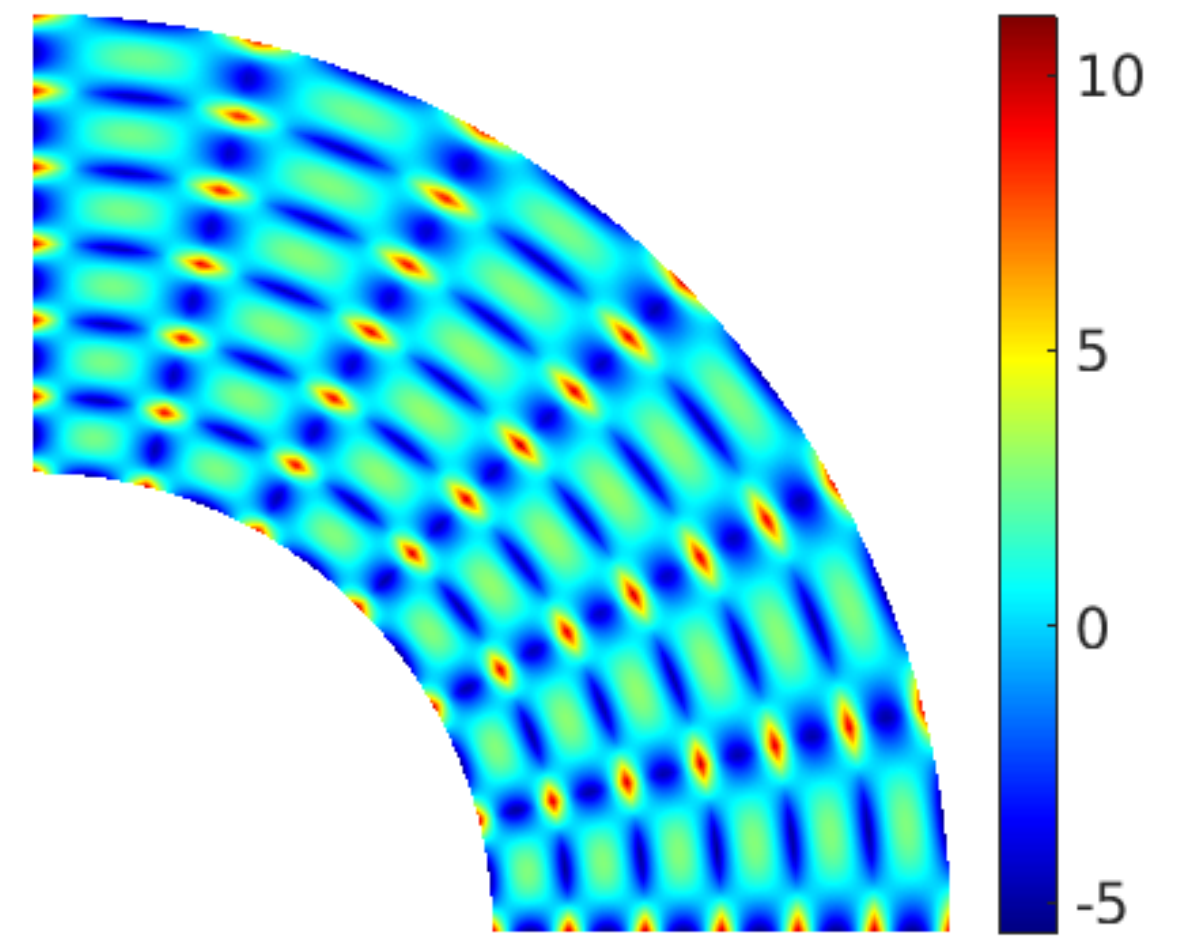}};
	\node at (-8.7,-2.5) {\small (a)  $u_1$};
	\node at (-1,-2.5) {\small (b)  $u_1$};
	\end{tikzpicture}	
	\caption{ The obtained approximate solutions for classical gradient-preserving B-spline basis functions of degree $\textsf{p}=2, \ r=0$ for each variable. For the left example we chose $h=1/20$, whereas the mesh for the right figure is the one from Fig. \ref{Fig:Num_8} (a). The  unstable behavior of the   $x$-displacement demonstrates the need of special B-spline spaces taking the underlying Hilbert complexes into account. }
	\label{Fig:Num_14}
\end{figure}

\subsubsection{Problems and outlook}
The examples from above can be seen as a proof of concept and that the mixed weak form of linear elasticity with weakly imposed symmetry can be discretized exploiting isogeometric spaces. Although different numerical experiments show promising results, there certainly arise aspects which  deserve a closer elaboration. Respectively, for reasons of completeness, we should mention  issues we observed when computing approximations. The first one is quite obvious and already mentioned, namely the large number of degrees of freedom that appear. Especially for three-dimensional domains, one readily obtains large linear systems   which require special solvers, like  iterative schemes. But then we solve such systems only by approximation and an additional error contribution has to be taken into account. \\
Another issue we noted is the presence of local stress oscillations  in case of mixed boundary conditions. To be more precise, at points where displacement boundary and traction boundary parts meet, meaning at points in $\overline{\Gamma_t} \cap \overline{\Gamma_D}$, we get in general beside the natural large stress magnitudes also local stress overshoots which seem problematic. As an example we display in Fig. \ref{Fig:Num_13} (a) the first stress component for the Cook's membrane problem and a coarse mesh $h=1/7$. In the top left corner of the membrane we can recognize the overshoots. Although this problem appears to be crucial, we see approaches to alleviate it. First of all, regions with oscillatory  stresses can be reduced by standard mesh refinements; see \ref{Fig:Num_13} (c). Second, we further observed more stable stresses if we combine high polynomial degrees with high regularity, see Fig. \ref{Fig:Num_13} (b). Furthermore, in all the tests we made, one obtains stable displacement approximations even if the stresses show overshoots. Indeed, a more rigorous explanation of the stress behavior   has to be established and more numerical experiments have to be carried out in order to understand this phenomenon. This could be the starting point of another article. 
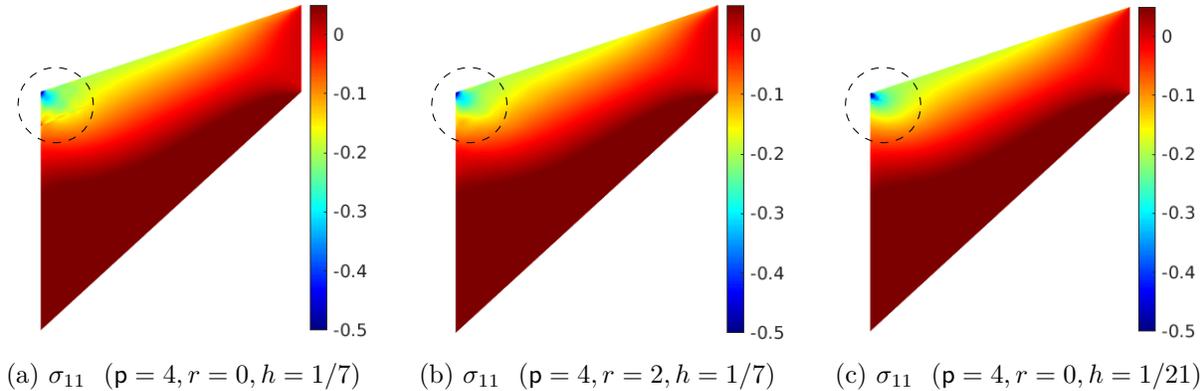
\begin{figure}[h!]
	\centering\begin{tikzpicture}
	
		\node (eins) at (-9.5,0) {\includegraphics[width=0.27\linewidth]{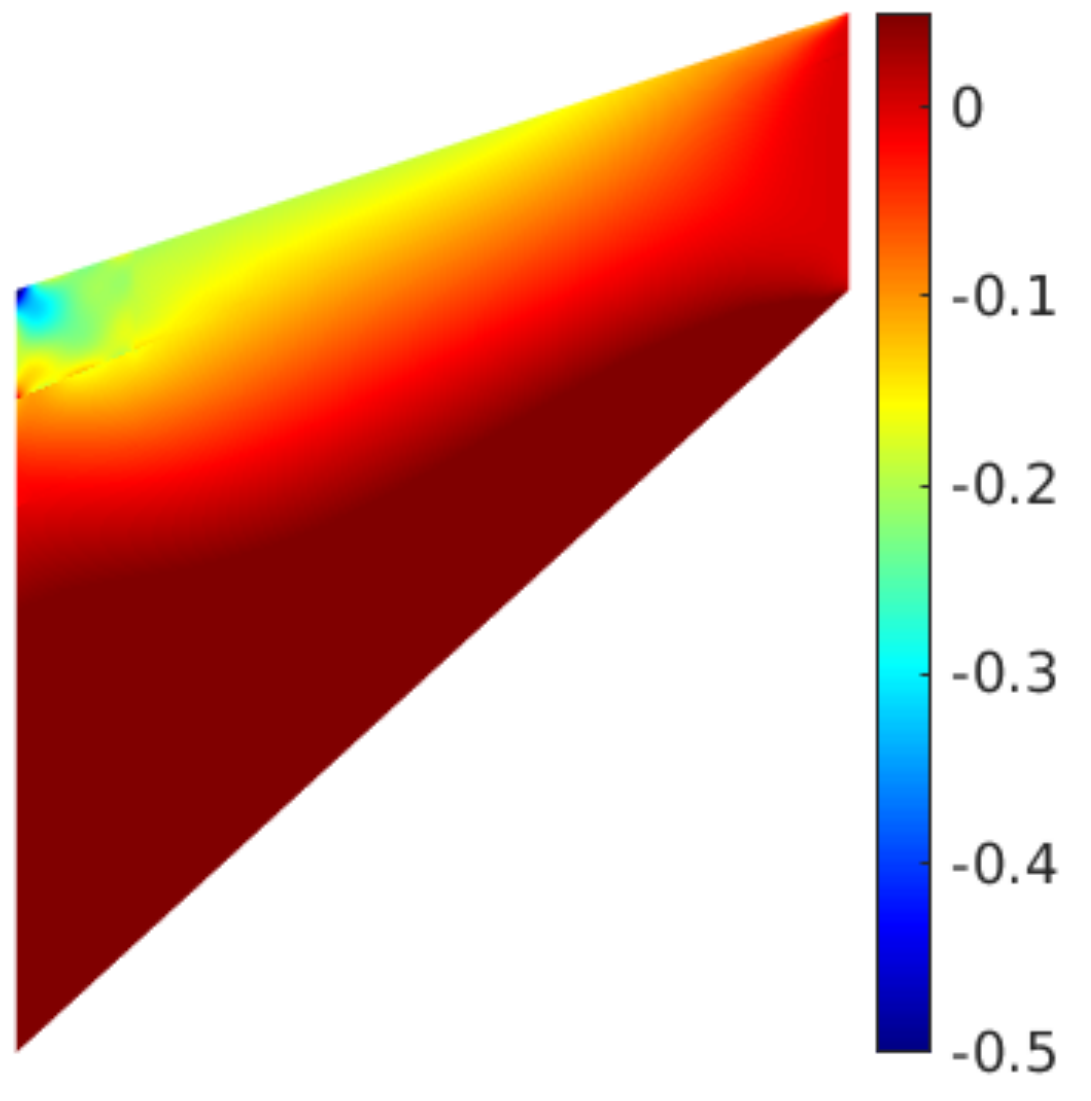}};
	
		\node[shift={(0,0)}] (eins) at (-4,0) {\includegraphics[width=0.27\linewidth]{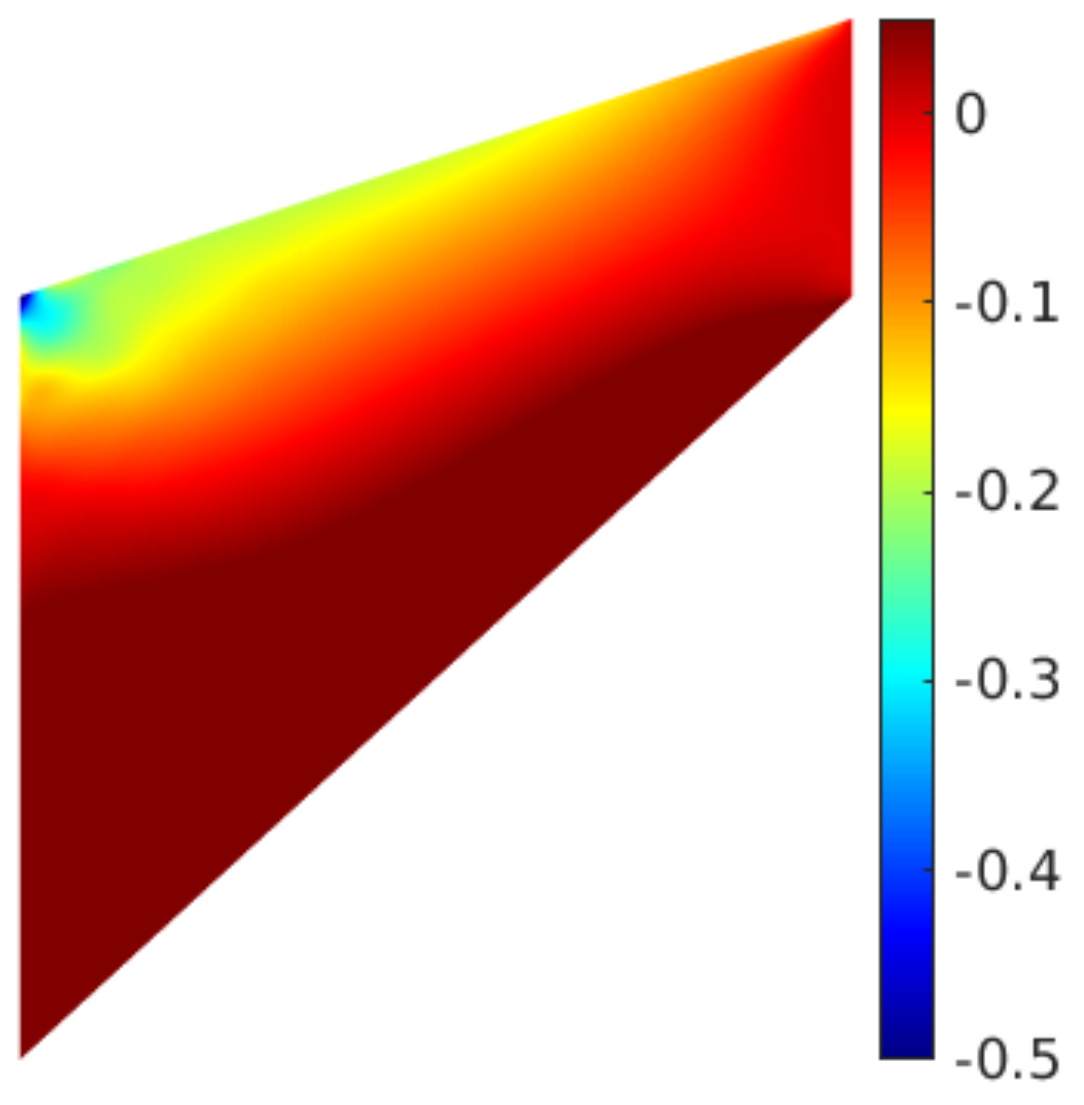}};		
	
			\node[shift={(0,0)}] (eins) at (1.5,0) {\includegraphics[width=0.27\linewidth]{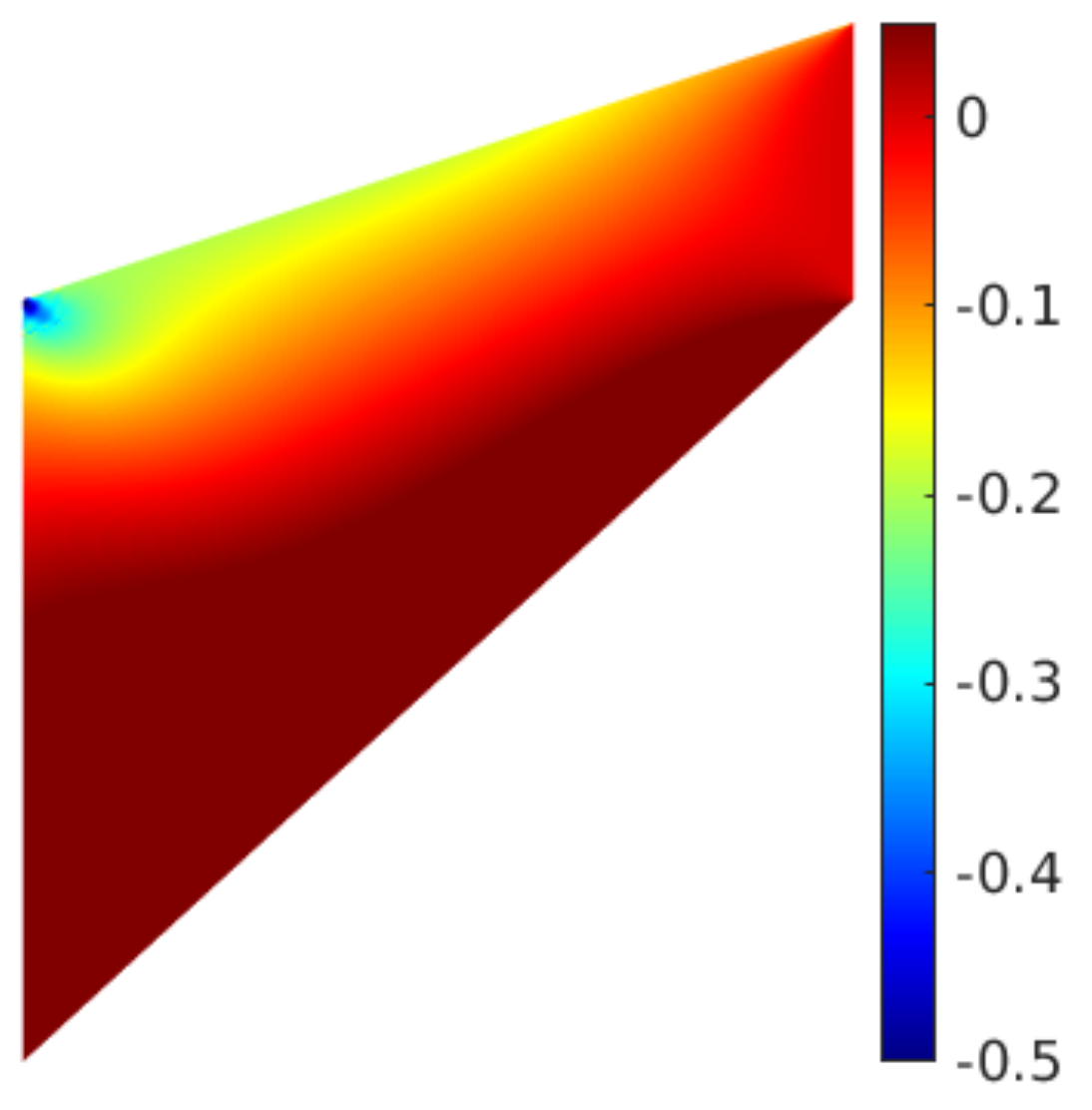}};
	
		\node at (-9.8,-2.7) {\small (a)  $\sigma_{11}$ \ ($\textsf{p}=4,r=0,h=1/7$) };
		\node at (-9.8+5.5,-2.7) {\small (b)  $\sigma_{11}$ \ ($\textsf{p}=4,r=2,h=1/7$) };
			\node at (-9.8+5.5+5.6,-2.7) {\small (c)  $\sigma_{11}$ \ ($\textsf{p}=4,r=0,h=1/21$) };

		\draw[dashed] (-11.5,0.9) circle(0.5);
		\draw[dashed] (-6,0.9) circle(0.5);
		\draw[dashed] (-0.5,0.9) circle(0.5);
	\end{tikzpicture}	
	\caption{ Mixed boundary conditions may trigger local stress oscillations like near the top left corner in Fig. (a). If one refines the mesh, or if splines with high regularity are used, then an improvement regarding these overshoots   can be observed; see Fig. (b)-(c). }
	\label{Fig:Num_13}
\end{figure}

\section{Conclusion}
\label{Sec:Conclusion}
In this article we studied approximation methods for the equations of linear elasticity based on the Hellinger-Reissner formulation with weakly imposed symmetry. Doing so, we made use of isogeometric analysis and isogeometric discrete differential forms. We showed the existence of a unique discrete solution in two and three dimensions and observed a stable behavior  in the incompressible  regime. Furthermore, a convergence estimate was derived that guarantees the possibility of arbitrary high convergence orders, provided smooth enough exact solutions. A main advantage of the approach  with IGA is the exact geometry representation  for various shapes. But there are also disadvantages regarding the proposed method and there were also points mentioned  which should be considered in more detail. First of all,  the mixed ansatz leads to a relatively large number of degrees of freedom. Further, numerical experiments indicate that the error estimation in $3D$ might be too pessimistic, since we saw for the displacement a faster convergence  than theoretically predicted. Besides, a study of the local stress overshoots appearing for mixed boundary conditions is advisable. The mentioned issues could be the starting point of future work.

\nocite{*}
\bibliographystyle{siam}
\bibliography{Literatur}

\appendix
\section{Appendix}
\subsubsection*{Proof of the commutativity relations in Lemma \ref{Lemma:Commutativty_double_complex_lin_ealsticity}}
\label{Proof1_appendix}
Let $v=(v_1,v_2)^T \in H^1(\o,\mathbb{R}^2)$. Then we get
\begin{align*}
\textup{Skew}(\textup{curl} \, v) = \textup{Skew}(\begin{bmatrix}
\partial_2 v_1& -\partial_1 v_1 \\ \partial_2 v_2& -\partial_1 v_2 
\end{bmatrix})  = \partial_1 v_1+ \partial_2 v_2.
\end{align*}
Further, choose $w \in H(\o,\c,\mathbb{M})$. On the one hand, it is \begin{align*}
\nabla \cdot (\Xi w) &= \nabla \cdot (\begin{bmatrix}
-w_{22}-w_{33} & w_{21} & w_{31} \\ w_{12}& -w_{11}-w_{33} & w_{32} \\
  w_{13} & w_{23} & -w_{11}-w_{22}
\end{bmatrix}) = \begin{bmatrix}
-\partial_1w_{22} - \partial_1 w_{33} + \partial_{2} w_{21} +  \partial_{3}w_{31} \\ 
\partial_1w_{12} - \partial_2 w_{11} - \partial_{2} w_{33} +  \partial_{3}w_{32} \\
\partial_1w_{13} + \partial_{2} w_{23} -  \partial_{3}w_{11} -\partial_{3}w_{22}
\end{bmatrix}
\end{align*}
On the other hand we obtain
\begin{align*}
\textup{Skew}(\nabla \times w) &= \textup{Skew}( \begin{bmatrix}
\partial_2 w_{13}-\partial_3 w_{12} & \partial_3 w_{11}-\partial_1 w_{13} & \partial_1 w_{12}-\partial_2 w_{11} \\ \partial_2 w_{23}-\partial_3 w_{22} & \partial_3 w_{21}-\partial_1 w_{23} & \partial_1 w_{22}-\partial_2 w_{21} \\
\partial_2 w_{33}-\partial_3 w_{32} & \partial_3 w_{31}-\partial_1 w_{33} & \partial_1 w_{32}-\partial_2 w_{31}
\end{bmatrix}) \\
&=\begin{bmatrix}
 \partial_3 w_{31}-\partial_1 w_{33}  - \big(  \partial_1 w_{22}-\partial_2 w_{21}  \big) \\
 \partial_1 w_{12}-\partial_2 w_{11} - \big( \partial_2 w_{33}-\partial_3 w_{32} \big) \\
 \partial_2 w_{23}-\partial_3 w_{22} - \big(\partial_3 w_{11}-\partial_1 w_{13}\big)
\end{bmatrix}
= \begin{bmatrix}
-\partial_1 w_{22}-\partial_1 w_{33}+ \partial_2 w_{21} + \partial_3 w_{31}\\
\partial_1 w_{12}-\partial_2 w_{11} -  \partial_2 w_{33} +\partial_3 w_{32} \\
\partial_1 w_{13}+\partial_2 w_{23} - \partial_3 w_{11}-\partial_3 w_{22}
\end{bmatrix}
\end{align*}
The comparison of the the right-hand sides of the last two equality chains finishes the proof.

\end{document}